\documentclass[11pt, a4paper]{article}
\usepackage[english]{babel}
\usepackage{amssymb,amsmath,amsthm,bm,graphicx}
\usepackage{authblk}
\usepackage[numbers]{natbib}

\newtheorem{theorem}{Theorem}
\newtheorem{lemma}{Lemma}

\theoremstyle{definition}

\newtheorem{remark}{Remark}

\DeclareMathOperator{\tr}{tr}
\DeclareMathOperator{\vect}{vect}

\newcommand{\EE}{{\mathbb{E}}}

\renewcommand{\phi}{\varphi}

\begin{document}

\title{\Large\bfseries Online Tracking of a Predictable Drifting Parameter of a Time Series}
\author{\bfseries Eduard Belitser\\ \vspace{-0.75em} {\rm Department of Mathematics, VU University Amsterdam}
\and
\bfseries Paulo Serra\\ \vspace{-0.75em} {\rm Institute for Mathematical Stochastics, University of G\"ottingen}}
\date{\today}

\maketitle

\begin{abstract}
We propose an online algorithm for tracking a multidimensional time-varying parameter of a time series,
which is also allowed to be a predictable process with respect to the underlying time series.
The algorithm is driven by a gain function. Under assumptions on the gain, we derive 
uniform non-asymptotic error bounds
on the tracking algorithm in terms of chosen step size for the algorithm 
and the variation of the parameter of interest. We also outline 
how appropriate gain functions can be constructed. 
We give several examples of different 
variational setups for the parameter process where our result can be applied.
The proposed approach covers many frameworks and models 
(including the classical Robbins-Monro and Kiefer-Wolfowitz procedures) where 
stochastic approximation algorithms comprise the main inference tool for the data analysis.
We treat in some detail a couple of specific models.
\end{abstract}

{\bf Keywords:}
on-line tracking;
predictable drifting parameter;
recursive algorithm;
stochastic approximation procedure;
time series;
time-varying parameter.

\tableofcontents

\section{Introduction}
 
When one analyzes data that arrive sequentially over time, it is important to detect %secular 
changes in the underlying model which can then be adjusted accordingly.
Such problems arise in many engineering (signal processing, speech recognition, 
communication systems), econometric and biomedical applications 
and can be found in an extensive literature widely scattered in these fields. 
Inference on %Estimation or tracking of 
time-varying parameters in stochastic systems is 
therefore of fundamental interest in sequential analysis.
%Furthermore, it arises in many engineering, econometric and biomedical applications 
%and has an extensive literature widely scattered in these fields. 
%Motivated by many applications in signal processing, speech recognition, 
%communication systems, neural physiology, environmental and economic modeling, 
%we consider recursive (online) estimation of a 
%multivariate time-varying parameter of a time series.

Consider an $\mathcal{X}$-valued time series $\{X_k, \, k\in\mathbb{N}_0\}$, 
$\mathbb{N}_0=\mathbb{N}\cup\{0\}$, $\mathcal{X} \subseteq\mathbb{R}^{l_k}$, $l_k \in \mathbb{N}$,
such that at time moment $k=0$ the first observation $X_0\sim \mathbb{P}_0$ and 
subsequently at each time moment $k\in \mathbb{N}$ a new observation $X_k$ arrives 
according to the model $X_k|\bm{X}_{k-1} \sim \mathbb{P}_k(\cdot|\bm{X}_{k-1})$,
%with transition law depending on some multivariate parameter $\theta_k \in \Theta \subseteq \mathbb{R}^d$ 
where $\bm{X}_{k-1}=(X_0, X_1,\dots, X_{k-1})$.
%We are interested in designing an online algorithm of tracking 
Suppose we are interested in certain characteristics of the 
conditional distribution of $X_k$ given the past $\bm{X}_{k-1}=\bm{x}_{k-1}$:
$A_k(\mathbb{P}_k(\cdot|\bm{x}_{k-1}))=\theta_k(\bm{x}_{k-1})$. Here $A_k$ is an operator mapping 
conditional distributions $\mathbb{P}_k(\cdot|\bm{x}_{k-1})$ into measurable $\Theta$-valued 
functions $\theta_k(\bm{x}_{k-1})$, $\bm{x}_{k-1} \in \mathcal{X}^k$, $\Theta$ is a %convex
compact subset of $\mathbb{R}^d$.
The goal is to estimate (or to track) $\theta_k(\bm{X}_{k-1})$ at time instant $k\in\mathbb{N}_0$,
based on the data $\bm{X}_{k}$  
(and prior information) available by that time moment.

The traditional parametric formulation is the most simple particular case of the above setting:
the observations are independent and the parameter $\theta \in \Theta$ is a constant vector. 
The simplest nonparametric formulation deals again with independent observations and 
time-varying parameter $\theta_k\in \Theta$, $k\in\mathbb{N}_0$ (cf.\ %Belitser and Korostelev (1992), Belitser and van de Geer (2000), Belitser (2004)
\citep{Belitser:2000, Belitser&Korostelev:1992, Belitser&Serra:2013}. 
Modeling  observations by a Markov chain with a time varying parameter of the transition law 
would add a next level of complexity (cf.\ for the autoregressive model in %Belitser (2000), Moulines et al. (2005)
\citep{Belitser:2000, Moulines:2005}). 

The proposed time series formulation admits an arbitrary 
dependence structure between the observations.
%is the most general sequential  framework as an arbitrary 
%dependence between observations is allowed within a general time series model. 
Another important and peculiar feature of our approach is that  
the multidimensional parameter $\theta_k\in \Theta \subseteq \mathbb{R}^d$, $k\in \mathbb{N}$, 
besides being time-varying, is also allowed to depend on the past of the time series. It is thus a predictable 
process with respect to the natural filtration: $\theta_k=\theta_k(\bm{X}_{k-1})$. 
An example of such characteristics is the conditional expectation 
$\theta_k(\bm{X}_{k-1})=\EE [X_k| \bm{X}_{k-1}]$. 
This time series formulation, with time-varying parameter of interest which is also allowed to depend on the past,
represents the most general sequential  framework, %setting, 
independent observations and Markov chains
are typical examples of models that fit into this framework.

Since the data arrives in a successive manner, conventional methods based on samples of fixed size are not easy to use. 
%This is in %sharp contrast with the 
A more appropriate approach is based on sequential methods, stochastic recursive 
algorithms, which allow fast updating of parameter or state estimates at each instant as new data arrive 
and therefore can be used to produce an ``online'' inference, that is, during the operation of the system.
%in statistics and engineering, which is more relevant in such a situation.
%Such methods, also known as stochastic recursive %identification algorithms, 
%for fast updating of parameter or state estimates 
%as new data arrive in real time are known as stochastic recursive %identification 
%algorithms.
Stochastic recursive algorithms, also known as stochastic approximation, take many forms and have 
numerous applications in the biomedical, socio-economic and engineering sciences, which highlights 
the interdisciplinary nature of the subject.
%biomedicine, economics and engineering,

There is a vast literature on stochastic approximation beginning with the seminal papers of 
Robbins and Monro \citep{Robbins:1951} and of Kiefer and Wolfowitz \citep{Kiefer:1952}.
There is a big variety of techniques in the area of stochastic approximation which have been 
developed and inspired by the applications from other fields. We mention here the books 
%Wasan (1969), Tsypkin (1971), Nevelson and Khasminskii (1976), Kushner and Clark (1978),
%Ljung and S\"oderstr\"om (1983), Benveniste, Metivier and Priouret (1990) and of Kushner and Yin(2003).
\citep{Wasan:1969, Tsypkin:1971, Nevelson:1976, Kushner:1978, Ljung:1983, Benveniste:1990, Kushner:2003}.

%, like incomplete data analysis,
%spatial statistics, neuro-dynamic programming, control systems, system identification, 
%signal processing, neural network modeling, econometric modeling, 
%perturbation analysis of discrete-event dynamic systems, 

A classical topic in adaptive control concerns the problem of tracking drifting parameters of a 
linear regression model, or somewhat equivalently, tracking the best linear fit when the parameters change slowly.
%and communications theory  %of the time-varying  %system, 
This problem also occurs in communication theory for adaptive equalizers and noise cancellation, etc., 
where the signal, noise, and channel properties change with time.
%many places in adaptive 
Successful stochastic approximation schemes for tracking in the time-varying case were given by 
%Brossier (1992), Delyon and Juditsky (1995), Kushner (1995) (see further references therein).
\citep{Brossier:1992, Delyon:1995, Kushner:1995,Kushner:2003} (see further references therein). 
%a nice discussion can be found in 

%% A major issue has been the choice of the 
%% step sizes in the tracking algorithm.
%% % Much attention has been devoted to this problem of step size 
%% % choice; see Benveniste, Metivier and Priouret (1990), Brossier
%% % (1992),Kushner and Yang (1995).
%% An algorithm for adapting 
%% the step size using the same observations which are 
%% used for the tracking was suggested by  Benveniste, Metivier
%% and Priouret (1990), various examples arising in signal processing
%% were worked out by Brossier (1992) and the theory was developed
%% by Kushner and Yang (1995).
%In \citep{Kushner:1995} (see also \citep{Benveniste:1990} and \citep{Brossier:1992}) 
%discuss the very important problem of the choice of the step sizes in the tracking algorithm.
%In general, the step size\index{Step size} of the tracking algorithm is not necessarily decreasing 
%to zero because of considerations concerning robustness of the actual physical model in practical 
%online applications and to allow some tracking of the desired parameter as the system changes over time.
%In signal processing applications, it is usual to keep the step size bounded away from zero.

Coming back to our time series model, the problem of tracking a time-varying parameter 
that is a functional of conditional distribution of the current observation given the past,  is clearly 
unfeasible, especially in such general formulation, without some conditions on the model.
In general, some knowledge about the structure of underlying time seres and some control over the variability 
of the parameter of interest over time  are needed.
%That is, if we do not assume any structure on the increments of the underlying parameter 
%$\theta_k$, $k \in \mathbb{N}$, then it will not be possible to assess the quality of any tracking algorithm.
Interestingly, in this seemingly very general time series framework, we actually do not require the (full) 
knowledge of the observational  model.
Instead, all we do need is to be able to compute a so called \emph{gain vector} at each time moment $k\in \mathbb{N}$, 
which is a certain (vector) function of the previous estimate of the parameter $\theta_k(\bm{X}_{k-1})$, new observation $X_k$ and 
prehistory $\bm{X}_{k-1}$. The essential property of such gain vector is that it, roughly speaking, ``pushes'' in the right direction
of the current value of true parameter to track.
%assume that we have some prior information about the model in the form of a so called \emph{gain vector}.
Although the assumption about the existence of that gain vector seems to be rather strong, 
we demonstrate on a number of interesting examples when such an assumption indeed holds. 
Basically, in case of Markov chain observations, if the form of transition density is known 
as function of the underlying parameter and it satisfies certain regularity assumptions, 
then the gain vector can always be constructed, for example, as a score function corresponding 
to the conditional maximum likelihood method.
Under appropriate regularity conditions (the existence of the conditional Fisher information and 
$L_2$-differentiability of the conditional log likelihood), such a score function 
has always the property of gain vector at least locally.

A gain function, together with a step sequence and new observations from the model, 
can be used to adjust the current approximation of the drifting parameter, resulting in a tracking algorithm.
To ease the verification of our assumptions on the gain function, we formulate them in two equivalent forms.
Under some assumptions on the gain vectors, we establish a uniform non-asymptotic bound the $L_1$ error of 
the resulting tracking algorithm, in terms of the variation of the drifting parameter.
Under the extra assumption that the gain function is bounded, we can strengthen this result to a uniform 
bound on the $L_p$ error (and then an almost sure bound). 
These error bounds constitute our main result and they also guide us in the choice of the step size for the algorithm.
Some extensions are also presented where we allow for approximation terms and approximate gains.

Based on our main result, we specify the appropriate choice for the step sequence in three different 
variational setups for the drifting parameter. We treat first the simple case of a constant parameter.
Although we are mainly concerned with tracking time-varying parameters, our algorithm is still of interest 
in the constant parameter case since it should result in an algorithm which is both recursive and robust.
We also consider a setup where the parameter is stabilizing.
This covers both the case where the parameter is converging and where we sample the signal with 
increasing frequency. The third variational setup covers the important case of tracking smooth signals.
This setup is somewhat different in that we make observations with a certain frequency from an 
underlying continuous time process which is indexed by a parameter changing like a Lipschitz function.
Our result can then either be interpreted as a uniform, non-asymptotic result for each fixed sampling 
frequency or as an asymptotic statement in the observation frequency.

Examples are also given for different possible gain functions.
These fall into two categories: general, score based gain functions for tracking multidimensional 
parameters in regular models and specialized gains for tracking more specific quantities.
The latter include gains to track level sets or maxima of drifting functions (extending the classical 
Robbins-Monro  and Kiefer-Wolfowitz algorithms) and gains to track drifting conditional quantiles.
We also propose modifications for a given gain function (rescaling, truncation, projection) 
which can be used to design gains tailored specifically to verify our assumptions.

We illustrate our method by treating some concrete applications of the proposed algorithm, in particular,  
we elaborate on the problem of tracking drifting parameters in autoregressive models.
Results on tracking algorithms for these models already exist in the literature 
(cf.\
% Belitser (2000) and Moulines et al.\ (2005))
\citep{Belitser:2000, Moulines:2005}) 
and we can derive similar results by choosing an appropriate gain function.
Using our approach, obtaining error bounds on the resulting tracking algorithm reduces to verifying 
our assumptions for the chosen gain function which considerably simplifies the derivation of results.

This paper is structured as follows.
In Section~\ref{sec:preliminaries} we summarize the notation that will be used thought the paper, 
as well as our model and  two equivalent formulations for our assumptions. Section~\ref{sec:main_result} 
contains our main result and respective proof as well as some straightforward extensions of the main result.
The construction and modification of gain functions for different models and different parameters of 
interest is explained in Section~\ref{sec:gains}. Section~\ref{sec:variational_setups} contains three examples 
of variational setups for the time-varying parameter for which we specify the tracking error implied by our main result.
We collect in Section~\ref{sec:examples} some examples of applications.
Section~\ref{sec:proofs} contains the proofs for our lemmas.

\section{Preliminaries}
\label{sec:preliminaries}

First we introduce some notation that we are going to use throughout the paper.
%We use lowercase letters for deterministic vectors and uppercase letters for random.  
%We use bold letters to represent families of vectors and matrices, these can be either upper- or lower-case, and 
All vectors are always column vectors. 
We use bold letters to represent matrices and families of vectors,
uppercase letters for families of random vectors and matrices; %lowercase letters for deterministic. 
denote $\bm{x}_n=(x_0,x_1,\ldots, x_n)$ for a set of vectors $x_0, \ldots, x_n$. 
For $k_0,k \in \mathbb{Z}$,
$\mathbb{N}_k=\{l\in\mathbb{Z}: l\ge k\}$ (of course, $\mathbb{N}=\mathbb{N}_1$), 
$\mathbb{N}_{k_0,k}=\{l\in\mathbb{Z}:k_0\le l\le k\}$. 
%and lowercase letters $n$ and $k$ (with subscripts) 
%are reserved to denote numbers only from $\mathbb{N}_0$. 
We use interchangeably $(a_i, i \in I)=\{a_i, i \in I\} = \{a_i\}_{i\in I}=(a_i)_{i\in I}$ for $I \subseteq \mathbb{N}_0$.
For vectors $x,y \in \mathbb{R}^d$, denote by $\|x\|=\|x\|_2$ and 
$\left\langle x,y\right\rangle=x^Ty$  the usual Euclidean norm and the inner 
product in $\mathbb{R}^d$, respectively, and by $\|x\|_p$ the $l_p$ norm (with $p\ge 1$) 
on vectors in $\mathbb{R}^d$.
For an event ${A}$, we denote by  $\mathbb{I}\{A\}$ %(or $\mathbb{I}_{A}$) 
the indicator of the event ${A}$. For a symmetric $(d\times d)$-matrix 
$\bm{M}$, let $\Lambda_{(1)}(\bm{M})$ and $\Lambda_{(d)}(\bm{M})$ 
be the smallest and the largest eigenvalues of $\bm{M}$ respectively.
Denote  by $\bm{O}$ denote the zero matrix and by $\bm{I}$ the identity matrix,
whose dimensions will be determined by the context. 
Besides, we adopt the convention that $\sum_{i\in\varnothing} \bm{A}_i = \bm{O}$ 
and $\prod_{i\in\varnothing} \bm{B}_i = \bm{I}$ 
for matrices $\bm{A}_i$ and $\bm{B}_i$. % with appropriate dimensions.
%that the corresponding matrix operations are well defined.
When applied to matrices, $\|\cdot\|_p$ represents the matrix norm
induced by the $l_p$ vector norm:
$\|\bm{M}\|_p = %\max_{x\neq 0}\frac{\|A x\|_p}{\|x\|_p}=
\max_{\|x\|=1}\|\bm{M} x\|_p$. %= \max_{x\le1}\|Ax\|_p$.

Assume that by time $n\in\mathbb{N}_0$, time series data $\bm{X}_n=(X_0, X_1, \dots, X_n)$ 
(which we may not be fully observable) occur according to the following model:
\begin{equation}
\label{model}
 X_0\sim \mathbb{P}_0, \qquad X_k|\bm{X}_{k-1} \sim \mathbb{P}_k(\cdot|\bm{X}_{k-1}),\quad k \in \mathbb{N}.
\end{equation}
Here vector $X_k$ takes value in some set 
$\mathcal{X}_k \subseteq \mathbb{R}^{l_k}$ with $l_k \in \mathbb{N}$,
i.e., $\mathbb{P}(X_k \in \mathcal{X}_k )=1$, $k\in\mathbb{N}_0$.
By $\mathbb{P}_k(\cdot|\bm{X}_{k-1})$ we denote the conditional distribution of $X_k$ given $\bm{X}_{k-1}$ and
$\mathcal{X}^{k+1}=\mathcal{X}^k \times\mathcal{X}_{k+1}$ for $k\in\mathbb{N}_0$,
 with $\mathcal{X}^0= \mathcal{X}_0$. Thus, $\bm{X}_k$ takes values in $\mathcal{X}^k$.
Clearly,  the distribution of  $\bm{X}_n$, $n\in\mathbb{N}_0$, is given by
\[
\mathbb{P}^{(n)}=\mathbb{P}^{(n)}(\bm{x}_n)=\prod_{k=0}^n 
\mathbb{P}_k(x_k|\bm{x}_{k-1}), \quad \bm{x}_k\in \mathcal{X}^k, \; k=0,1,\ldots, n,
\] 
where $\mathbb{P}_0(x_0|\bm{x}_{-1})$ should be understood as $\mathbb{P}_0(x_0)$.
For the sake of consistent notation, $\bm{x}_{-1}$ (also $\bm{x}_{-2}$ etc.) means void variables from now on.

Introduce an increasing sequence of $\sigma$-algebras $\{\mathcal{F}_k\}_{k\in\mathbb{N}_{-1}}$ %nested
(i.e., $\mathcal{F}_{k_1} \subseteq \mathcal{F}_{k_2}$ if $k_1\le k_2$) such that 
$\mathcal{F}_k\subseteq \sigma(\bm{X}_k)$, $k\in\mathbb{N}_0$, where $\sigma(\bm{X}_k)$
denote the $\sigma$-algebra  generated by $\bm{X}_k$ and  $\mathcal{F}_{-1}$ is a
$\sigma$-algebra  such that $\mathcal{F}_{-1} \subseteq \mathcal{F}_0$. 
%We will slightly abuse notation by   
%writing $\{\mathcal{F}_k\}_{k\in \bar{\mathbb{N}}}$ (instead of, for example,  
%$\mathcal{F}_{-1} \cup \{\mathcal{F}_k\}_{k \in\mathbb{N}_0}$) 
%for the filtration with the convention that 
%the $\sigma$-algebras $\mathcal{F}_{-2}$ and $\mathcal{F}_{-1}$ are also included 
%in the filtration whenever needed to complete the argument for all time moments $k\in\mathbb{N}_0$.
Unless otherwise specified, we assume that $\mathcal{F}_k=\sigma(\bm{X}_k)$.
As is  discussed in Remark \ref{rem:filtration} below, 
the case $\mathcal{F}_k \subset \sigma(\bm{X}_k)$ is  not conceptually new
because it can be reduced to the situation of the time series $\{\bm{Z}_k, k\in \mathbb{N}_0\}$
with $\bm{Z}_k=h_k(\bm{X}_k)$ for some $(\sigma(\bm{X}_k),\mathcal{F}_k)$-measurable 
functions $h_k$'s such that $\mathcal{F}_k =\sigma(\bm{Z}_k)$, $k\in \mathbb{N}_0$.

Now we describe the statistical model which is a family probability measures for the observed data.
For each $k\in\mathbb{N}_0$, let  $\mathcal{P}_k$ be a given family of conditional probability 
measures of $X_k$ given $\bm{X}_{k-1}=\bm{x}_{k-1}\in \mathcal{X}^{k-1}$. 
Then the underlying statistical model is determined by imposing  
$\mathbb{P}_k(\cdot|\bm{x}_{k-1}) \in \mathcal{P}_k$, $k\in\mathbb{N}_0$.
Thus, at time $n\in \mathbb{N}_0$, the underlying (growing) statistical model is 
$\mathbb{P}^{(n)}\in \mathcal{P}^{(n)} = \Big\{\prod_{k=0}^n \mathbb{P}_k(x_k|\bm{x}_{k-1}): \, 
\mathbb{P}_k(\cdot|\bm{x}_{k-1}) \in \mathcal{P}_k\Big\}$. % k=0,1,\ldots, n
%for some family $\mathcal{P}^{(n)}$ of measures on $\mathcal{X}^n$.

For some  %convex 
compact subset $\Theta$ of $\mathbb{R}^{d}$, 
%%% Is it possible to generalize to moving between 
%%% different dimensions: $\Theta_k\subset \mathbb{R}^{d_k}$? 
denote by $\mathcal{B}_\Theta$  the Borel $\sigma$-algebra on $\Theta$ 
and by $\mathcal{M}_k$ the set  of $\Theta$-valued 
$(\mathcal{F}_k,\mathcal{B}_{\Theta})$-measurable functions on 
$\mathcal{X}^k$, $k\in\mathbb{N}_0$. Consider a sequence of operators 
$A_k:\, \mathcal{P}_k \mapsto \mathcal{M}_{k-1}$, $k\in\mathbb{N}_0$, 
so that for a $\mathbb{P}_k(\cdot|\bm{x}_{k-1}) \in \mathcal{P}_k$
\[
A_k(\mathbb{P}_k(\cdot|\bm{x}_{k-1}))=\theta_k(\bm{x}_{k-1}), \quad \bm{x}_{k-1}\in\mathcal{X}^{k-1},
\]
with  $\theta_k(\bm{x}_{k-1}) \in \mathcal{M}_{k-1}$.
We will often abbreviate  $\theta_k=\theta_k(\bm{x}_{k-1})$, remembering that this is 
a measurable function of $\bm{x}_{k-1}$.
%possibly via some other intermediate function (if $\mathcal{F}_{k-1} \subset \sigma(\bm{X}_{k-1})$).
%In particular, if $\mathcal{F}_k =\sigma(\bm{X}_k)$, then $\theta_k=\theta_k(\bm{x}_{k-1})$.
For $k=0$, since $\bm{x}_{-1}$ is void, $\theta_0(\bm{x}_{-1})=\theta_0 \in \Theta$ means a constant.

Our goal is to design an online algorithm for tracking
the drifting \emph{parameter of interest} $\theta_k(\bm{X}_{k-1})$
at time moment $k\in \mathbb{N}_0$
on the basis of the data $\bm{X}_k$ observed by that time moment.
The time-varying parameter $\theta_k=\theta_k(\bm{X}_{k-1})$, $k\in\mathbb{N}_0$, 
is thus allowed to depend on the past of the time series, i.e., it is a predictable process 
with respect to the filtration $\{\mathcal{F}_k\}_{k \in\mathbb{N}_{-1}}$. 
Recall further that $\theta_k$ is assumed to take values in some %convex 
compact subset $\Theta$ of $\mathbb{R}^d$, to be precise, 
$\mathbb{P}({\theta}_k(\bm{X}_{k-1})\in\Theta) = 1$ for all $k\in \mathbb{N}_0$.
Denote
\begin{equation}
\label{C_Theta} 
\sup_{\theta\in \Theta} \|{\theta}\|^2 = C_\Theta.
\end{equation}

At time $k$, given $\bm{X}_k$, the model $\mathcal{P}_{k+1}$ contains all the relevant 
information about the next observation. Actually, we do not consider the model 
to be (completely) known. 
Instead, we assume that our prior knowledge about the model is formalized as follows:
for each $k\in\mathbb{N}_0$, we have certain 
$(\mathcal{B}_\Theta\times \mathcal{F}_k, \mathcal{B}_\Theta)$-measurable %$\mathbb{R}^d$-valued 
functions $G_k$ at our disposal (which we call \emph{gain functions} 
or \emph{gain vectors} or just \emph{gains}), $G_k: \, \mathbb{R}^d\times\mathcal{X}^k\mapsto\mathbb{R}^d$.
We use these gain functions to construct a recursive algorithm for tracking the sequence 
$\theta_k=\theta_k(\bm{X}_{k-1})\in \Theta\subset \mathbb{R}^d$ from the observations (\ref{model}):
\begin{equation}
\label{eq:algorithm_main}
\hat{\theta}_{k+1} = \hat{\theta}_k + \gamma_k G_k(\hat{\theta}_k,\bm{X}_k), \quad k \in \mathbb{N}_0,
\end{equation}
for  some positive  sequence $\gamma_k\le\Gamma$ 
and some (arbitrary) initial value $\hat{\theta}_0 \in\Theta\subset \mathbb{R}^d$, 
measurable with respect to $\mathcal{F}_{-1}$. 
Since $\hat{\theta}_k=\hat{\theta}_k(\bm{X}_{k-1})$ is 
$\mathcal{F}_{k-1}$-measurable, then 
$\{\hat{\theta}_k\}_{k \in \mathbb{N}_0}$ is predictable with respect to the filtration 
$\{\mathcal{F}_k\}_{k \in\mathbb{N}_{-1}}$. Notice that $\hat{\theta}_0$ can be a random vector
if $\mathcal{F}_{-1}$ is not the trivial $\sigma$-algebra.
%$\mathcal{F}_{-1} \cup \{\sigma(\bm{X}_k)\}_{k \in\mathbb{N}_0}$.

Of course, it is not to be expected that the tracking algorithm (\ref{eq:algorithm_main})
performs well for arbitrary gains. Intuition suggests that 
%in order for tracking algorithm (\ref{eq:algorithm_main}) to perform reasonably, 
the gain $G_k$ should ``push'' $\hat{\theta}_k$ in the direction of $\theta_k$. 
The following conditions formalize this requirement.
\begin{itemize}
\item[(A1)] For all $k\in\mathbb{N}_0$, the quantity, which we call \emph{(conditional) average gain},
\begin{equation}
\label{eq:g}
g_k(\hat{\theta}_k,\theta_k)=g_k(\hat{\theta}_k,\theta_k|\bm{X}_{k-1})=
\EE \big[G_k(\hat{\theta}_k,\bm{X}_k) | \mathcal{F}_{k-1} \big]
\end{equation}
is well defined (recall that $\theta_k=\theta_k(\bm{X}_{k-1}) = A_k(\mathbb{P}_k)$ 
and $\hat{\theta}_k$ is defined by (\ref{eq:algorithm_main}))
and there exist a %(random)  
$\mathcal{F}_{k-1}$-measurable  symmetric positive definite matrix 
$\bm{M}_k=\bm{M}_k(\bm{X}_{k-1})$ 
%=\bm{M}_k(\bm{X}_{k-1})$ 
(its entries are $\mathcal{F}_{k-1}$-measurable functions) 
%with eigenvalues $0<\Lambda_{(1)}(\bm{M}_k)\le\dots\le\Lambda_{(d)}(\bm{M}_k)$ 
such that, almost surely
\begin{equation}
\label{eq:D}
g_k(\hat{\theta}_k,\theta_k|\bm{X}_{k-1}) = -\bm{M}_k(\hat{\theta}_k-\theta_k),
\end{equation}
\begin{equation}
\label{eq:Lambda}
\lambda_1 \le \EE [\Lambda_{(1)}(\bm{M}_k)|\mathcal{F}_{k-2}],\quad 
\lambda_1 \le\Lambda_{(d)}(\bm{M}_k) \le \lambda_2,
\end{equation}
for some fixed constants $0<\lambda_1 \le \lambda_2 <\infty$.
\item[(A2)] There exists a constant $C_g>0$ such that 
\begin{equation}
\label{eq:g2}
\EE \|G_k(\hat{\theta}_k,\bm{X}_k) - g_k(\hat{\theta}_k,\theta_k|\bm{X}_{k-1}) \|^2 \le C_g, \quad k\in\mathbb{N}_0.
%\int\|G_k({\vartheta},x|\bm{X}_{k-1})\|^2\, d\mathbb{P}_k(x|\bm{X}_{k-1})\le\bar{G}_k(\bm{X}_{k-1})+c_g\|\vartheta\|^2,
\end{equation}
\end{itemize}
As one can see from (\ref{eq:Lambda}), a $\sigma$-algebra  $\mathcal{F}_{-2}$,
such that $\mathcal{F}_{-2}\subseteq\mathcal{F}_{-1}$, is also needed.
Without loss of generality, assume that 
$\mathcal{F}_{-2}=\{\varnothing, \mathcal{X}_0\}$, 
the trivial $\sigma$-algebra on $\mathcal{X}_0$. 
We will often use shorthand notation: $G_k = G_k(\hat{\theta}_k, \bm{X}_k)$ and 
$g_k=g_k(\hat{\theta}_k,\theta_k |\bm{X}_{k-1})$.

%The general observations scheme (\ref{model}) includes the classical Robbins-Monro setting when 
%the conditional distribution $\mathbb{P}_{k+1}(x_{k+1}|\bm{X}_k)$ explicitly depends on 
%the $\hat{\theta}_k$ defined by (\ref{eq:algorithm_main}), i.e., 
%$\mathbb{P}_{k+1}(x_{k+1}|\bm{X}_k) =\tilde{\mathbb{P}}_{k+1}(x_{k+1}|\bm{X}_k,\hat{\theta}_k(\bm{X}_k))$
%for some family of (conditional) probability measures $\tilde{\mathbb{P}}_{k+1}$, $k\in\mathbb{N}_0$.

\begin{remark}
Condition (A1) means, in a way, that the gain $G_k(\vartheta,\bm{X}_k)$ 
shifts any $\vartheta=\vartheta(\bm{X}_{k-1})$, on average, 
towards the ``true'' value $\theta_{k}=\theta_{k}(\bm{X}_{k-1})$. 
%Indeed,
%\[
%\EE \big[G_k(\vartheta, X_k |\bm{X}_{k-1}) | \mathcal{F}_{k-1}\big] 
%= g_k(\vartheta, \theta_k|\bm{X}_{k-1}) = - \bm{M}_k(\bm{X}_{k-1})\big(\vartheta-\theta_k\big),
%\]
%for some symmetric, almost surely positive definite matrix $\bm{M}_k(\bm{X}_{k-1})$ satisfying the relations in (A1).
This elucidates the idea of algorithm (\ref{eq:algorithm_main}).
Suppose at time instant $k \in \mathbb{N}$ an observer had a reasonable estimator
$\hat{\theta}_{k-1}=\hat{\theta}_{k-1}(\bm{X}_{k-1})$ of the ``old'' value of the
parameter of interest $\theta_{k-1}=\theta_{k-1}(\bm{X}_{k-2})$, and a new data vector $X_k$ arrives. 
Then the available data is $\bm{X}_k=(X_k,\bm{X}_{k-1})$ and the observer 
can construct an estimator $\hat{\theta}_k=\hat{\theta}_k(\bm{X}_k)$ 
of the new value  of the parameter $\theta_k=\theta_k(\bm{X}_{k-1})$
by calculating the gain $G_k(\hat{\theta}_k,\bm{X}_k)$
and using a rescaled (by a step size $\gamma_k>0$) version of it  to update 
the ``old'' estimator $\hat{\theta}_{k-1}$ towards $\theta_k=\theta_k(\bm{X}_{k-1})$.

The upper bound in relation (\ref{eq:Lambda}) means that the gain is of a bounded magnitude, 
and the lower bound %in (\ref{eq:Lambda}) 
has the meaning of the so called {\em persistence of excitation} as it is termed in control theory literature. 
\end{remark}
 
\begin{remark}
Assumption (A2) is trivially satisfied if 
the gain vectors $G_k(\vartheta,\bm{X}_k)$ are almost surely uniformly bounded.
This is not so difficult to arrange, for example, by dividing the gain vector by a multiple of its length
or by truncating. In doing so, we make the resulting gain vector bounded, whereas retaining its direction.
We discuss these approaches  in more detail in Section \ref{sec:gains}.

Assumption (A2) is also satisfied if 
\begin{equation}
\label{G_bounded}
\EE\|G_k(\hat{\theta}_k,\bm{X}_k)\|^2 \le C_g, \quad k\in \mathbb{N}_0.
\end{equation}
Indeed, for a random vector $X$ with a finite second moment of its norm and any $\sigma$-algebra $\mathcal{F}$,
$\EE\big\|X-\EE(X|\mathcal{F})\big\|^2= \EE \EE \big[\|X\|^2 - 
\|\EE(X | \mathcal{F})\|^2 | \mathcal{F}\big]\le \EE \|X\|^2$.
% By Jensen's inequality, $\EE\big(X-\EE[X | \mathcal{F}]\big)^2
% = \EE \EE \big[X^2 - (\EE[X | \mathcal{F}])^2\mid\mathcal{F}\big]\ge \mbox{Var}(X)$
Combining this with (\ref{G_bounded}) yields
%\begin{equation}
%\label{D_bounded}
$\EE \|G_k - g_k\|^2 
 \le \EE\|G_k\|^2 \le C_g$,  $k\in \mathbb{N}_0$.
%\end{equation}

On the other hand, from (\ref{C_Theta}), (A1) and (A2)  it follows that
\begin{align}
\EE \|G_k\|^2 &= \EE \|G_k-g_k +g_k\|^2 
 \le 2 C_g + 2 \EE\|g_k\|^2 
\le 2 C_g + 2 \lambda_2^2 \EE\|\hat{\theta}_k -\theta_k\|^2  \notag\\
\label{bound_G_k}
& \le 2 C_g + 4 \lambda_2^2 C_\Theta  + 4 \lambda_2^2
\EE\|\hat{\theta}_k\|^2,  \quad  k \in \mathbb{N}_0.
\end{align}
This relation and  Lemma \ref{lemma_bound}  below (thus the conditions 
of Lemma \ref{lemma_bound} must hold) will in turn imply the uniform bound (\ref{G_bounded}).
\end{remark}

\begin{remark}
Generally, there is no universal way to find gain vectors which satisfy
conditions (A1) and (A2). In many practical situations, the model 
$\{\mathcal{P}_k, k \in \mathbb{N}_0\}$ is typically specified and 
it is an art to find gain vectors which satisfy (A1) and (A2); 
we discuss this issue in more detail in Section \ref{sec:gains}.
The assumptions above look somewhat unnatural and cumbersome because
they are assumed to hold for all $k\in\mathbb{N}$, whereas 
functions involved in the conditions depend in general on $\bm{X}_{k-1}$
whose dimension increases unlimitedly as 
 $k$  increases.
However, the assumptions become reasonable in the important case of Markov chain
observations $\{X_k, \, k\in \mathbb{N}_0\}$ of order, say, $p$.
In this case, for any $k \in \mathbb{N}$ 
we can use vector of bounded dimension $(X_{k-p}, \ldots, X_{k-1})$ 
instead of $\bm{X}_{k-1}$ (of growing dimension) in all the quantities from conditions (A1) and (A2). 

Independent observations is a next simplification, also important in many practical applications. 
In this case there is no past involved in the function $\theta_k$, $k \in \mathbb{N}_0$, 
it will only be a function of time.
\end{remark}

\begin{remark} 
\label{rem:filtration}
Suppose that conditions (A1) and (A2) hold for the filtration 
$\mathcal{F}_{-1} \cup\{\sigma(\bm{X}_k)\}_{k\in\mathbb{N}_0}$
and for some measurable gain functions $G_k(\hat{\theta}_k, \bm{X}_k)$,
$k\in\mathbb{N}_0$,  but the parameter sequence 
$\{\theta_k\}_{k \in \mathbb{N}_0}$ is predictable with respect to 
a coarser filtration  $\{\mathcal{F}_k\}_{k\in\mathbb{N}_{-1}}$, i.e., 
$\sigma(\bm{Z}_k)=\mathcal{F}_k \subset \sigma(\bm{X}_k)$, $k\in\mathbb{N}_0$, where 
$\bm{Z}_k= h_k(\bm{X}_k)$ for some measurable $h_k$'s.
For example, the vector $X_k$ consist of two subvectors $Z_k$ and $Y_k$ (i.e., $X_k= (Z_k,Y_k)$) and 
$\mathcal{F}_k  = \sigma(\bm{Z}_k)$ with  $\bm{Z}_k=(Z_0,\ldots,Z_k)$
(think of $\bm{Y}_k$ as unobservable part of $\bm{X}_k$ and 
$\bm{Z}_k$ as observable).
Then, by the tower property of the conditional expectation, conditions (A1) and (A2) 
hold for the filtration $\{\mathcal{F}_k\}_{k\in\mathbb{N}_{-1}}$ as well if we take the 
new gain function $\bar{G}_k(\hat{\theta}_k, \bm{Z}_k) = 
\EE \big[G_k(\hat{\theta}_k, \bm{X}_k) | \mathcal{F}_k \big]$, $k\in\mathbb{N}_0$.
In fact, this means that the case of a coarser filtration $\{\mathcal{F}_k\}_{k\in\mathbb{N}_{-1}}$,
with respect to which the parameter sequence $\{\theta_k\}_{k \in \mathbb{N}_0}$ is predictable,
can be reduced to the above described setup in terms of  ``new time series'' $\bm{Z}_k$, $k\in \mathbb{N}_0$,
and the filtration $\{\mathcal{F}_k\}_{k\in\mathbb{N}_{-1}}$ generated by this new time series 
(or other way around). 
Thus, if  the parameter sequence $\{\theta_k\}_{k \in \mathbb{N}_0}$ is known to be predictable 
with respect to a coarser filtration  $\{\mathcal{F}_k\}_{k\in\mathbb{N}_{-1}}$,
we could have considered this coarser filtration and the corresponding time series 
$\{\bm{Z}_k\}_{k\in \mathbb{N}_0}$
from the very beginning, and impose conditions (A1) and (A2) in terms of 
$\{\bm{Z}_k\}_{k\in \mathbb{N}_0}$ and $\{\mathcal{F}_k\}_{k\in\mathbb{N}_{-1}}$.
In fact, the coarser the filtration $\{\mathcal{F}_k\}_{k\in\mathbb{N}_{-1}}$, the weaker the
conditions. 

On the other hand, if the parameter sequence $\{\theta_k\}_{k \in \mathbb{N}_0}$ is known to be predictable 
with respect to a finer filtration  $\{\mathcal{F}_k\}_{k\in\mathbb{N}_{-1}}$
(i.e., $\sigma(\bm{Z}_k)=\mathcal{F}_k \supset \sigma(\bm{X}_k)$, $k\in\mathbb{N}_0$), then 
the key relation (\ref{eq:D}) will most likely not hold since the expression on the left hand side is 
$\sigma(\bm{X}_k)$-measurable and the expression on the right hand side is $\mathcal{F}_k $-measurable.
However, in such situation  (\ref{eq:D}) may still hold  with some (small) error $\eta_k$,
we address this issue in Remark \ref{rem:close_parameter} below.
Intuitively, the information of what we observe should match (or, at least, not be less than) 
the information of what we want to track.

%Clearly, if  $\mathcal{F}_k \subset\sigma(\bm{X}_k)$, $k\in\mathbb{N}_0$, 
%one can then treat $\mathcal{F}_k$ as %for illustrative purpose let 
%$\mathcal{F}_k=\sigma(\bm{Y}_k)$ for $\bm{Y}_k= h_k(\bm{X}_k)$ 
%with some $(\sigma(\bm{X}_k), \mathcal{F}_k)$-measurable $h_k$'s.
%In this case, the parameter sequence $\{\hat{\theta}_k\}_{k \in \mathbb{N}_0}$
%must be predictable with respect to the filtration $\{\mathcal{F}_k\}_{k\in\mathbb{N}_{-1}}$,
%the gain function $G_k$ must be $\mathcal{F}_k$-measurable (i.e., depending on $\bm{Y}_k$) 
%and  condition (A1) must hold with $\bm{Y}_k$ instead of $\bm{X}_k$ in 
%(\ref{eq:g}), (\ref{eq:D}) and (\ref{eq:Lambda}).
%This means in essence that the case of a coarser filtration  $\{\mathcal{F}_k\}_{k\in\mathbb{N}_{-1}}$ 
%(as compared to $\mathcal{F}_{-1} \cup\{\sigma(\bm{X}_k)\}_{k\in\mathbb{N}_0}$)  
%can be reduced to the above described setup in terms of $\bm{Y}_k$'s instead of $\bm{X}_k$'s.
%A principally new situation can only occur if $G_k$ is $\sigma(\bm{X}_k)$-measurable, then   
%$\hat{\theta}_{k+1}$ is $\sigma(\bm{X}_k)$-measurable and the representation
%(\ref{eq:D}) holds  in general with some error $\eta_k$;
%we address this issue in Remark \ref{rem:close_parameter} below.

\end{remark}

\begin{remark}
\label{rem:parametrized}
Consider one particular case of our general setting.
%Let $\Theta$ be some convex compact subset of $\mathbb{R}^d$.
At time $n\in \mathbb{N}_0$,  we observe $\bm{X}_n=(X_0, X_1, \dots, X_n)$ 
such that %according to the following distribution:
\begin{equation}
\label{model2}
X_0\sim \mathbb{P}_{\theta_0}, \;\; X_k|\bm{X}_{k-1} 
\sim \mathbb{P}_{\theta_k}(\cdot|\bm{X}_{k-1}),\qquad 
\theta_k \in \Theta\subset \mathbb{R}^d, \;\; k \in \mathbb{N}.
\end{equation}
The model in this case is 
$\mathcal{P}_k = \mathcal{P}_k(\Theta)=\big\{\mathbb{P}_{\theta}(\cdot|\bm{x}_{k-1}): 
\,  \theta\in \Theta, \, \bm{x}_{k-1}\in \mathcal{X}^{k-1}\big\}$ and the operator 
$A_k(\mathbb{P}_{\theta_k}(\cdot|\bm{x}_{k-1}))=\theta_k$.
This is a convenient formulation when the time series model is parametrized 
by a time-varying parameter which we would like to recover by using an online tracking algorithm.
Also in this case we can actually allow the parameter $\theta_k$ 
to depend on the past of the time series, i.e., $\theta_k=\theta_k(\bm{X}_{k-1})$ so that the sequence
$\{\theta_k, k \in \mathbb{N}_0\}$ is predictable with respect to the filtration
$\{\mathcal{F}_k\}_{k\in\mathbb{N}_{-1}}$.
%% $\mathcal{F}_{-1} \cup \{\sigma(\bm{X}_k)\}_{k \in\mathbb{N}_0}$.
\end{remark}

Condition (A1) can be reformulated as condition (\~A1) below, which gives some intuition about 
the role of the average gain $g_k$ defined in (A1) and which may, in certain situations, be easier to verify.
\begin{itemize}
\item[(\~A1)]
For $k \in \mathbb{N}_0$, the average gain $g_k(\hat{\theta}_k,\theta_k |\bm{X}_{k-1})$ defined in (A1) satisfies, 
almost surely, the following conditions: there exist random variables $\Lambda_{(1)}(\bm{X}_{k-1})$,  
$\Lambda_{(d)}(\bm{X}_{k-1})$ and constants $0<\lambda_1\le \lambda_2<\infty$, $L>0$ 
such that 
%\begin{equation} %nice possibility to make one lebal for  several equations
%\label{some_label}
%\begin{aligned}
%a&=b+c\\
%d&=e+g
%\end{aligned}
%\end{equation}
$\|g_k(\hat{\theta}_k,\theta_k |\bm{X}_{k-1})\|\le L \|\hat{\theta}_k-\theta_k\|$ and
\[
\Lambda_{(1)}(\bm{X}_{k-1}) \|\hat{\theta}_k-\theta_k\|^2 
\le 
-(\hat{\theta}_k-\theta_k)^Tg_k(\hat{\theta}_k,\theta_k |\bm{X}_{k-1})\le
\Lambda_{(d)}(\bm{X}_{k-1}) \|\hat{\theta}_k-\theta_k\|^2,
\]
with $\lambda_1 \le \EE [\Lambda_{(1)}(\bm{M}_k)|\mathcal{F}_{k-2}]$ and 
$\lambda_1 \le\Lambda_{(d)}(\bm{M}_k) \le \lambda_2$.
\end{itemize}
In view of the lemma below, if (A1) holds, then (\~A1) will also hold (and vice versa);
the values of the constants $\lambda_1$ and $\lambda_2$ appearing in the assumptions 
are different, though.
The proof of this lemma is deferred to Section~\ref{sec:proofs}.
\begin{lemma} 
\label{lemma1}
Let $x,y\in \mathbb{R}^d$.
If there exists a symmetric positive definite matrix $\bm{M}$ such that $y=\bm{M}x$ 
and $0<\lambda_1 \le \lambda_{(1)}(\bm{M}) \le \lambda_{(d)}(\bm{M}) \le \lambda_2 <\infty$
for some $\lambda_1,\lambda_2\in\mathbb{R}$, 
then $0<  \lambda_1' \|x\|^2 \le \left\langle x,y\right\rangle \le \lambda_2' \|x\|^2 
< \infty$ and $\|y\|\le C \|x\|$  for some $\lambda_1',\lambda_2', C \in \mathbb{R}$ 
(depending only on $\lambda_1,\lambda_2$) such that $0<\lambda_1' \le \lambda_2' <\infty$ and $C>0$.

Conversely, if $0<  \lambda_1' \|x\|^2 \le \left\langle x,y\right\rangle \le \lambda_2' \|x\|^2 
< \infty$ and $\|y\|\le C \|x\|$  for some $\lambda_1',\lambda_2', C \in \mathbb{R}$
 such that $0<\lambda_1' \le \lambda_2' <\infty$ and $C>0$, then 
there exists a symmetric positive definite matrix $\bm{M}$ such that $y=\bm{M}x$ 
and $0<\lambda_1 \le \lambda_{(1)}(\bm{M}) \le \lambda_{(d)}(\bm{M}) \le \lambda_2 <\infty$
for some constants $\lambda_1,\lambda_2\in\mathbb{R}$
depending only on $\lambda_1', \lambda_2'$ and $C$.
\end{lemma}

\section{Main results}
\label{sec:main_result}

We start with a lemma which we will need in the proof of the main result.
Heuristically, since the gain vector $G_k(\hat{\theta}_k,\bm{X}_k)$ 
moves, on average, $\hat{\theta}_k$ towards $\theta_k$ and the sequence $\theta_k \in \Theta$
is bounded  (since $\Theta$ is compact), the resulting estimating sequence $\hat{\theta}_k$ should also be well-behaved.
The following lemma states that the second moment of $\hat{\theta}_k$ 
is uniformly bounded in $k\in\mathbb{N}_0$ for sufficiently small $\gamma_k$'s.
\begin{lemma}
\label{lemma_bound}
Let assumptions (A1) and (A2) hold. Then 
for sufficiently small $\gamma_k$ there exists a constant $\bar{C}_\Theta$ such that
\[
\EE  \|\hat{\theta}_k\|^2 \le \bar{C}_\Theta, \qquad k \in
\mathbb{N}_0.
\]
\end{lemma}
The proof of this lemma is given in Section~\ref{sec:proofs}.
In fact, it is enough to assume that $\gamma_k$ is sufficiently small for 
all $k \ge N$ for some fixed $N\in\mathbb{N}$. This is the case if $\gamma_k \to 0$ as $k \to \infty$,
which is typically assumed. 
This lemma will be used in the proof of the main theorem below.
From now on we assume that the sequence $\gamma_k$ is such that Lemma~\ref{lemma_bound} holds.

The following theorem is our main result, it provides a non-asymptotic upper bound on the quality 
of the tracking algorithm  (\ref{eq:algorithm_main}) in terms of of the algorithm step sequence
$\{\gamma_i, i \in \mathbb{N}_0\}$ and oscillation of the process to track 
$\{\theta_i, i\in \mathbb{N}_0\}$ between  arbitrary time moments $k_0,k \in \mathbb{N}_0$, $k_0 \le k$.

\begin{theorem}
\label{theo:bound}
Let assumptions (A1) and (A2) hold, the tracking sequence $\hat{\theta}_k$ be defined by (\ref{eq:algorithm_main})
and $\delta_k = \delta_k(\bm{X}_{k-1}) = \hat{\theta}_k - \theta_k$, $k\in \mathbb{N}_0$.
Then for any $k_0, k \in \mathbb{N}_0$ and sequence $\{\gamma_k, k\in \mathbb{N}_0\}$ 
(satisfying the conditions of Lemma \ref{lemma_bound}) such that  $k_0\le k$ and 
$\gamma_i\lambda_2\le 1$ for all $i \in \mathbb{N}_{k_0,k}$, the following relation holds:
\begin{equation}
\label{eq:boundE}
\EE \|\delta_{k+1}\|\le
C_1\exp\Big\{-\frac{\lambda_1}{2}\!\!\sum_{i=k_0}^k \!\!\!\gamma_i\Big\}+ 
C_2 \Big[\sum_{i=k_0}^k \gamma_i^2\Big]^{1/2} \!\!+ C_3\max_{k_0\le i\le k} \EE \|\theta_{i+1}-\theta_{k_0}\|,
\end{equation}
where $C_1 = \sqrt{2}(\bar{C}_\Theta+C_\Theta)^{1/2}$, 
$C_2=C_g^{1/2}(1+\lambda_2/\lambda_1)$,
$C_3=(1+\lambda_2/\lambda_1)$, constants $\lambda_1,\lambda_2, C_g$ are from assumptions (A1) and (A2),
$C_\Theta$ is defined by (\ref{C_Theta}) and $\bar{C}_\Theta$ is from Lemma~\ref{lemma_bound}.
\end{theorem}

\begin{remark}
\label{rem:norms}
By using (\ref{eq:boundE}), one can derive a bound for $\EE \|\delta_{k+1}\|_p$, with $p\ge 1$.
Indeed, as $\|x\|_s \le \|x\|_r \le d^{1/r-1/s}\|x\|_s$ for any $x\in \mathbb{R}^d$ and $s\ge r\ge 1$, 
we obtain that $\|\delta_{k+1}\|_p\le \|\delta_{k+1}\|_2 = \|\delta_{k+1}\|$ for $p\ge 2$ and
$\|\delta_{k+1}\|_p \le d^{1/p-1/2}\|\delta_{k+1}\|_2 \le d^{1/2}\|\delta_{k+1}\|$ 
for $1\le p< 2$.
\end{remark}

\begin{proof}
For the sake of  brevity, denote %$\theta_k = \theta_k(\bm{X}_{k-1})$,  
$\Delta_k=\Delta_k(\bm{X}_k) = \theta_k-\theta_{k+1}$,
$G_k=G(\hat{\theta}_k,\bm{X}_k)$, $g_k = g(\hat{\theta}_k,\theta_k|\bm{X}_{k-1})$ 
and $D_k = G_k - g_k$, $k\in\mathbb{N}_0$.
%Recall that $\mathcal{F}_k = \sigma(\bm{X}_k)$, $k\in \mathbb{N}_0$, 
%%is the $\sigma$-field generated by $\bm{X}_k=(X_0, X_2, \dots, X_k)$
%and ${\mathcal F}_{-1}=\{\varnothing, \mathcal{X}\}$ is the trivial $\sigma$-field.
We have
\[
\EE [D_k|\mathcal{F}_{k-1}]=
\EE [(G_k-g_k)|\mathcal{F}_{k-1}]=%g_k(\hat{\theta}_k,\theta_k|\bm{X}_{k-1}) = 
g_k-g_k=0, \quad k \in \mathbb{N}_0.
\]
It follows that $\{D_k, k\in\mathbb{N}_0\}$, is a (vector) martingale 
difference  sequence with respect to the filtration 
$\{\mathcal{F}_k\}_{k \in\mathbb{N}_{-1}}$.

Rewrite the algorithm equation (\ref{eq:algorithm_main}) as 
\[
\delta_{k+1}=\delta_k + \Delta\theta_k+\gamma_k D_k	+	\gamma_k g_k, \quad k \in \mathbb{N}_0.
\]
In view of (A1), %(or equivalently Assumption (\~A1)), 
the decomposition $g_k = -\bm{M}_k \delta_k$ holds almost surely, 
with an $\mathcal{F}_{k-1}$-measurable
symmetric positive definite matrix $\bm{M}_k$, so that
%\begin{equation}
%\label{eq:recursion}
\[
\delta_{k+1} = \Delta\theta_k + \gamma_k D_k + (\bm{I}- \gamma_k \bm{M}_k)\delta_k, 
\quad k \in \mathbb{N}_0.
\]
%\end{equation}
By iterating the above relation, we obtain that  for any $k_0 =0,\ldots, k$
\begin{align}
\delta_{k+1} &= (\bm{I}- \gamma_k \bm{M}_k)(\bm{I}- \gamma_{k-1} 
\bm{M}_{k-1})\delta_{k-1} + \Delta\theta_k+\gamma_k D_k \notag\\
&\quad  + (\bm{I}- \gamma_k \bm{M}_k)(\Delta\theta_{k-1}+\gamma_{k-1} D_{k-1}) \notag\\ %=\ldots \\
&\label{eq:recursed}
= \Big[\prod_{i=k_0}^k(\bm{I}- \gamma_i \bm{M}_i)\Big] \delta_{k_0} +
\sum_{i=k_0}^k \Big[\prod_{j=i+1}^k(\bm{I}- \gamma_j \bm{M}_j)\Big] (\Delta\theta_i+\gamma_i D_i).
\end{align}
%with the convention that $\sum_{i=m+l}^m A_i = O$ and 
%$\prod_{i=m+l}^m A_i = I$ for any $m,l \in \mathbb{N}$, %$l\ge 1$, 
%and any $d\times d$ matrices $A_i$.
Denote $A_i=\sum_{j=k_0}^i\gamma_j D_j$, $B_i=\sum_{j=k_0}^i\Delta\theta_j$ and $H_i =A_i+B_i$.
Applying the Abel transformation (Lemma \ref{lemma:abel}) 
to the second term of the  right hand side of (\ref{eq:recursed}) yields 
\begin{equation}
\label{abel_transform}
\sum_{i=k_0}^k \! \Big[\!\prod_{j=i+1}^k(\bm{I}-\gamma_j\bm{M}_j)\Big] (\Delta\theta_i+\gamma_i D_i)
= H_k - \sum_{i=k_0}^{k-1}\! \gamma_{i+1}\bm{M}_{i+1}
\Big[\!\prod_{j=i+2}^k(\bm{I}-\gamma_j \bm{M}_j)\Big]H_i.
\end{equation}
In particular, note that if we take $d=1$, $\bm{M}_j = \lambda_1$ and
$\Delta \theta_j=0$ for $j=k_0,\ldots, k$, $D_{k_0}=1$ and $D_j =0$ 
for $j=k_0+1, \ldots, k$, we derive that (since $0\le \gamma_j\lambda_1 \le1$ for $j=k_0, \ldots, k$) 
\begin{align}
\label{eq:bound_on_coefs}
\sum_{i=k_0}^{k-1}\lambda_1\gamma_{i+1}\prod_{j=i+2}^k(1- \gamma_j \lambda_1)
= 1- \prod_{j=k_0+1}^k (1- \gamma_j \lambda_1)
\leq 1,
\end{align}
which we will use later.

Using (\ref{abel_transform}), we can rewrite our expansion of $\delta_{k+1}$ in (\ref{eq:recursed}) as follows:
%\begin{align*}
%\delta_{k+1} &=  
%\Big[\prod_{i=k_0}^k(I- \gamma_i M_i)\Big] \delta_{k_0} +
%\sum_{i=k_0}^k \Big[\prod_{j=i+1}^k(I- \gamma_j M_j)\Big] \Delta\theta_i\notag\\
%&
%\quad
%+A_k - \sum_{i=k_0}^{k-1}\gamma_{i+1}M_{i+1}
%\Big[\prod_{j=i+2}^k(I-\gamma_jM_j)\Big]A_i.
%\end{align*}
\begin{align*}
\delta_{k+1} &=  
\Big[\prod_{i=k_0}^k(\bm{I}- \gamma_i \bm{M}_i)\Big] \delta_{k_0}
+H_k - \sum_{i=k_0}^{k-1}\gamma_{i+1}\bm{M}_{i+1}
\Big[\prod_{j=i+2}^k(\bm{I}-\gamma_j\bm{M}_j)\Big]H_i.
\end{align*}

The previous display, the Minkowski inequality and the sub-multiplicative property 
of the operator norm  ($\|\bm{A} \bm{B}\|\le \|\bm{A}\|\|\bm{B}\|$) imply that 
%the same is true for any $\ge 1$, the norm becomes $\|\cdot\|_p$ instead of $\|\cdot\|_2$.
\begin{align}
\|\delta_{k+1}\|&\leq
\|\delta_{k_0}\|\prod_{i=k_0}^k  \|\bm{I}- \gamma_i \bm{M}_i\| + \|H_k\| \notag\\
& 
\label{eq:after_triangle_lp}
\quad + \sum_{i=k_0}^{k-1}\gamma_{i+1}\|\bm{M}_{i+1}\| \|H_i\|
\prod_{j=i+2}^k \|\bm{I}-\gamma_j\bm{M}_j\|.
\end{align}

%For $i=k_0,\ldots, k$, let $\Lambda_{(1),i}$ and $\Lambda_{(d),i}$ denote 
%the smallest and the largest eigenvalues of the matrix $\bm{M}_i$, respectively.
In view of  (A1) and the condition $\gamma_i\lambda_2\le 1$ for $i=k_0,\ldots, k$, 
$\gamma_i\Lambda_{(d)}(\bm{M}_i)\le \gamma_i\lambda_2<1$, $i=k_0,\ldots, k$, 
almost surely. Hence $0\le \gamma_i \Lambda_{(1)}(\bm{M}_i)\le\gamma_i\Lambda_{(d)}(\bm{M}_i)\le1$, $i=k_0,\ldots, k$,
almost surely. This, Lemma~\ref{lemma:eig}  and the fact (see (\ref{eq:Lambda}) from (A1)) that 
$0<\lambda_1 \le \EE [\Lambda_{(1)}(\bm{M}_i)|\mathcal{F}_{i-2}]$, $i=k_0,\ldots, k$, 
almost surely, imply that
\begin{align}
\EE \prod_{i=k_0}^k  & \|\bm{I}- \gamma_i \bm{M}_i\|^2
=\EE  \prod_{i=k_0}^k  \big(1- \gamma_i \Lambda_{(1)}(\bm{M}_i)\big)^2 \le
\EE  \EE \bigg[\prod_{i=k_0}^k  \big(1- \gamma_i \Lambda_{(1)}(\bm{M}_i)\big)\Big|\mathcal{F}_{k-2}\bigg]\notag\\
&=\EE \Big[ \EE \big[ \big(1-\gamma_k \Lambda_{(1)}(\bm{M}_k)\big)\big|\mathcal{F}_{k-2}\big]
\prod_{i=k_0}^{k-1}  \big(1- \gamma_i \Lambda_{(1)}(\bm{M}_i)\big) \Big] \notag\\
& \label{eq:m_i}
\le ( 1-\gamma_k \lambda_1 )\EE  \prod_{i=k_0}^{k-1}\big(I-\gamma_i \Lambda_{(1)}(\bm{M}_i)\big)
\le\ldots \le  \prod_{i=k_0}^k (1- \gamma_i \lambda_1 ).
\end{align}
By (\ref{C_Theta}) and Lemma~\ref{lemma_bound}, we have
$\EE \|\delta_{k_0}\|^2 \le 2 \big(\EE \|\theta_{k_0}\|^2
+\EE \|\hat{\theta}_{k_0}\|^2\big)  
\le 2(C_\Theta+\bar{C}_\Theta)=C^2_1$.
Using this fact, the Cauchy-Schwartz inequality, (\ref{eq:m_i})
and the elementary inequality $1+x\le e^x$, $x\in\mathbb{R}$, leads to
%to the first term of the right hand side of (\ref{eq:after_triangle_lp}) to obtain
\begin{align}
\EE \Big[\|\delta_{k_0}\|\prod_{i=k_0}^k  \|\bm{I}- \gamma_i \bm{M}_i\|  \Big] 
&\le \Big[\EE  \|\delta_{k_0}\|^2\, \EE \prod_{i=k_0}^k \|\bm{I}- \gamma_i \bm{M}_i\|^2\Big]^{1/2} \notag\\
&\label{eq:products}
\le 
C_1\prod_{i=k_0}^k (1- \gamma_i \lambda_1 )^{1/2} \le
C_1\exp\Big\{-\frac{\lambda_1}{2}\sum_{i=k_0}^k\gamma_i\Big\}.
\end{align}

%For all $j\ge i$,  $\|\bm{M}_{i+1}\|\|B_i\|$ is $\mathcal{F}_{j}$-measurable. 
%Therefore, by (A1) and Lemma \ref{lemma:eig},
%\begin{align}
%\EE \Big[\|\bm{M}_{i+1}\| \|&B_i\|\prod_{j=i+2}^k\|\bm{I} - \gamma_j \bm{M}_j\| \Big]
%=\EE  \EE \Big[\|\bm{M}_{i+1}\|\|B_i\| \prod_{j=i+2}^k \|\bm{I}- \gamma_j \bm{M}_j\| \Big| \mathcal{F}_{k-2} \Big]
%\notag\\
%& =\EE  \Big[ \EE \big[1-\gamma_k \Lambda_{(1),j} \big| \mathcal{F}_{k-2} \big] \|\bm{M}_{i+1}\|\|B_i\| 
%\prod_{j=i+2}^{k-1} (1- \gamma_j \Lambda_{(1),j}) \Big] \notag\\
%&\le(1-\gamma_k\lambda_1) \EE \Big[\|\bm{M}_{i+1}\|\|B_i\| \prod_{j=i+2}^{k-1} ( 1- \gamma_j \Lambda_{(1),j})\Big] 
%\notag\\
%& \label{eq:b_i} 
%\le \EE \big[\|\bm{M}_{i+1}\| \|B_i\|\big] \prod_{j=i+2}^k(1- \gamma_j \lambda_1)
%\le \lambda_2 \EE\|B_i\| \prod_{j=i+2}^k (1- \gamma_j \lambda_1).
%\end{align}

Let  $D_{kl}$ denote the $l$-th coordinate of the vector $D_k$.
Clearly, for each $l=1,\ldots, d$, $\{D_{kl}, \, k\in\mathbb{N}_0\}$ is a martingale difference
with respect to the filtration $\{\mathcal{F}_k\}_{k \in\mathbb{N}_{-1}}$.
Using  (\ref{eq:g2}) from (A2) and the fact that martingale difference terms 
are uncorrelated, we derive that for all $i=k_0, \ldots, k$
\begin{align*}
\EE \|A_i\|^2
&=
\EE \sum_{l=1}^d \Big(\sum_{j=k_0}^i \gamma_j D_{jl}\Big)^2
= \sum_{l=1}^d \sum_{j=k_0}^i \gamma_j^2 \EE  D_{jl}^2
=\sum_{j=k_0}^i \gamma_j^2 \EE  \|D_j\|^2
\le C_g \sum_{j=k_0}^k \gamma_j^2.
\end{align*}
Denote for brevity 
$\Gamma^2_{k_0,k}=\sum_{j=k_0}^k \gamma_j^2$, 
so that 
$\EE \|A_i\|^2 \le C_g \Gamma^2_{k_0,k}$, $i=k_0, \ldots, k$.
%Using (\ref{eq:Lambda}), the Cauchy-Schwartz inequality, the last relation and (\ref{eq:m_i}), we obtain
% \begin{align}
%\EE \Big[\|\bm{M}_{i+1}\| \|A_i\|\prod_{j=i+2}^k\|\bm{I} - \gamma_j \bm{M}_j\| \Big]
%&\le\lambda_2 \Big[ \EE\|A_i\|^2  \, \EE \prod_{i=k_0}^k  \|\bm{I}- \gamma_i \bm{M}_i\|^2 \Big]^{1/2} \notag\\
%& \label{eq:a_i}
%\le \lambda_2 \Gamma_{k_0,k}
%\prod_{j=i+2}^k \Big(1- \frac{\gamma_j \lambda_1}{2}\Big),
%\end{align}
%$i=k_0, \ldots, k$. In the last display, we also used the elementary inequality 
%$(1+x)^\alpha \le 1+\alpha x$ for all $x \ge -1$ and $\alpha \in (0,1)$. 

%Combining this with (\ref{eq:b_i}) and (\ref{eq:a_i}) yields
%\begin{align}
%\EE & \Big[\|\bm{M}_{i+1}\| \|H_i\|\prod_{j=i+2}^k\|\bm{I} - \gamma_j \bm{M}_j\| \Big]
%\le\EE \Big[\|\bm{M}_{i+1}\|\big(\|A_i\| +\|B_i\|\big)\prod_{j=i+2}^k \|\bm{I}- \gamma_j \bm{M}_j\|\Big]\notag\\
%& \label{eq:c_i}
%\le \lambda_2 \Big[\Gamma_{k_0,k}
%\prod_{j=i+2}^k \Big(1- \frac{\gamma_j \lambda_1}{2}\Big) 
%+\EE \|\theta_{i+1}-\theta_{k_0}\| \prod_{j=i+2}^k(1- \gamma_j \lambda_1)\Big].
%\end{align}
The obtained relation for $\EE \|A_i\|^2$, together with the Minkowski and H\"older inequalities, 
imply that for all  $i=k_0, \ldots, k$
\begin{equation}
\label{eq:Eh_i}
\EE \|H_i\| \le \EE\|A_i\|+ \EE \|B_i\|  \le  \big(\EE\|A_i\|^2 \big)^{1/2} + \EE\|B_i\| \le
C_g^{1/2} \Gamma_{k_0,k} + \EE\|B_i\|.
\end{equation}

%Alternatively,
%\[
%\EE \|B_i\|_p = \EE \big\|\sum_{j=k_0}^i\Delta\theta_j\big\|_p 
%\le \sum_{j=k_0}^i\EE \|\Delta\theta_j\|_p \le \sum_{j=k_0}^k\EE \|\Delta\theta_j\|_p.
%\]

Notice that $\|\bm{M}_{i+1}\|\|H_i\|$ from (\ref{eq:after_triangle_lp}) 
is $\mathcal{F}_{j}$-measurable for all $j\ge i$. 
Therefore, by (\ref{eq:Lambda}) from (A1), (\ref{eq:Eh_i}) and Lemma \ref{lemma:eig},
\begin{align}
\EE \Big[\|\bm{M}_{i+1}\| \|&H_i\|\prod_{j=i+2}^k\|\bm{I} - \gamma_j \bm{M}_j\| \Big]
=\EE  \EE \Big[\|\bm{M}_{i+1}\|\|H_i\| \prod_{j=i+2}^k \|\bm{I}- \gamma_j \bm{M}_j\| \Big| \mathcal{F}_{k-2} \Big]\notag\\
& =\EE  \Big[ \EE \big[1-\gamma_k \Lambda_{(1)}(\bm{M}_k) \big| \mathcal{F}_{k-2} \big] \|\bm{M}_{i+1}\|\|H_i\| 
\prod_{j=i+2}^{k-1} \big(1- \gamma_j \Lambda_{(1)}(\bm{M}_j)\big) \Big] \notag\\
&\le(1-\gamma_k\lambda_1) 
\EE \Big[\|\bm{M}_{i+1}\|\|H_i\| \prod_{j=i+2}^{k-1} \big(1- \gamma_j \Lambda_{(1)}(\bm{M}_j)\big)\Big] \le \ldots \notag\\
& \le \EE \big(\|\bm{M}_{i+1}\| \|H_i\|\big) \prod_{j=i+2}^k(1- \gamma_j \lambda_1)
\le \lambda_2 \EE\|H_i\| \prod_{j=i+2}^k(1- \gamma_j \lambda_1) \notag\\
&
 \label{eq:c_i}
\le \lambda_2 \Big[C_g^{1/2} \Gamma_{k_0,k}
+\EE\|B_i\| \Big] \prod_{j=i+2}^k(1- \gamma_j \lambda_1).
\end{align}

Now we take the expectation of relation (\ref{eq:after_triangle_lp}) and use 
relations (\ref{eq:products}), (\ref{eq:Eh_i}), (\ref{eq:c_i})   
and (\ref{eq:bound_on_coefs}) to derive that $\EE \|\delta_{k+1}\| $ is bounded from above by 
\begin{align*}
%& \EE \|\delta_{k+1}\| \\
&\EE \Big[\|\delta_{k_0}\| \prod_{i=k_0}^k \|\bm{I}- \gamma_i \bm{M}_i\| \Big]
+\EE  \|H_k\| + \sum_{i=k_0}^{k-1}
\gamma_{i+1} \EE \Big[ \|\bm{M}_{i+1}\| \|H_i\| \prod_{j=i+2}^k \|\bm{I}-\gamma_j\bm{M}_j\| \Big]\\
&\le 
C_1\exp\Big\{-\frac{\lambda_1}{2}\sum_{i=k_0}^k\gamma_i\Big\}  
+\EE\|H_k\| + \lambda_2 \sum_{i=k_0}^{k-1} \gamma_{i+1}  
\EE\|H_i\| \prod_{j=i+2}^k (1- \gamma_j \lambda_1) 
\\
&\le C_1\exp\Big\{-\frac{\lambda_1}{2}\sum_{i=k_0}^k\gamma_i\Big\}+
\EE\|H_i\| \Big[1+\sum_{i=k_0}^{k-1}\lambda_2 \gamma_{i+1}
\prod_{j=i+2}^k \big(1-\gamma_j \lambda_1\big)\Big] \\
&\le  
C_1\exp\Big\{-\frac{\lambda_1}{2}\sum_{i=k_0}^k\gamma_i\Big\}+
\Big[C_g^{1/2} \Gamma_{k_0,k}+\max_{k_0\le i\le k} \EE \|\theta_{i+1}-\theta_{k_0}\| \Big]
\Big(1+\frac{\lambda_2}{\lambda_1}\Big),
\end{align*}
where we also used in the last bound that
$B_i=\sum_{j=k_0}^i\Delta\theta_j=\theta_{i+1}-\theta_{k_0}$, $ i=k_0, \ldots, k$, 
is a telescopic sum. This completes the proof of the the theorem.
\end{proof}

\begin{remark}
At this stage,  it may not be clear how the non-asymptotic bound from Theorem \ref{theo:bound}
can be utilized. 
%The result is given in terms of the algorithm step sequence  and oscillations of the process to track
%between two arbitrary time moments $k_0, k \in \mathbb{N}_0$.    
%So far, we did not impose any structure on the behavior of the oscillations of the parameter 
%process $\{\theta_k, k \in \mathbb{N}_0\}$, 
The obtained result is not useful unless we assume some sort of damping of the 
oscillations of the parameter process $\{\theta_k, k \in \mathbb{N}_0\}$.
%, either by stabilizing the parameter function 
%$\theta_k$, or by making more observations per time unit.
Looking ahead, in Section \ref{sec:variational_setups} we impose certain 
settings   for damping of the parameter process oscillations  (either ``stabilizing'' in 
time or increasing the observation frequency) and derive results in various asymptotic regimes 
by using our main Theorem \ref{theo:bound}.
\end{remark}

\begin{remark}
If we assume a slightly stronger version of (\ref{eq:Lambda}) in (A1),
\[
0<\lambda_1\le \Lambda_{(1)}(\bm{M}_i) \le\Lambda_{(d)}(\bm{M}_i)\le \lambda_2, 
\quad k\in\mathbb{N}_0,
\] 
almost surely, then a slightly better version of bound (\ref{eq:products}) holds:
\[
\EE  \Big[ \|\delta_{k_0}\| \prod_{i=k_0}^k  \|\bm{I}- \gamma_i \bm{M}_i\| \Big] \le 
\bar{C}_1 \prod_{i=k_0+1}^k (1- \gamma_i \lambda_1 ) \le 
\bar{C}_1\exp\Big\{-\lambda_1\sum_{i=k_0}^k\gamma_i\Big\},
\]
since $\EE \|\delta_{k_0}\| \le \EE \|\hat{\theta}_{k_0}\| + \EE \|\theta_{k_0}\| 
\le \bar{C}_\Theta^{1/2}+C_\Theta^{1/2}=\bar{C}_1$, by (\ref{C_Theta}) 
and Lemma~\ref{lemma_bound}.

We can derive a bound alternative 
to (\ref{eq:products}), which leads to slightly better constants  
in the first term of the right hand side of  (\ref{eq:boundE}).  
Indeed, $\delta_k$ is $\mathcal{F}_{k-1}$ -measurable, and, instead of (\ref{eq:products}), 
we derive
%This, Lemma~\ref{lemma:eig}  and the fact (see (\ref{eq:Lambda}) from (A1)) that 
%$0<\lambda_1 \le \EE [\Lambda_{(1),i}|\mathcal{F}_{i-2}]$, $i=k_0,\ldots, k$, 
%almost surely, imply that
\begin{align*}
\EE  & \Big[ \|\delta_{k_0}\| \prod_{i=k_0}^k  \|\bm{I}- \gamma_i \bm{M}_i\| \Big]
%=\EE  \Big[ \|\delta_{k_0}\| \prod_{i=k_0}^k  \big(1- \gamma_i \Lambda_{(1),i}\big)\Big] 
=\EE  \EE \Big[\|\delta_{k_0}\| \prod_{i=k_0}^k \big(1- \gamma_i \Lambda_{(1)}(\bm{M}_i)\big)\big|\mathcal{F}_{k-2}\Big] \\
%=\EE  \EE \big[ \big(1-\gamma_k \Lambda_{(1),k}\big)^2\big|\mathcal{F}_{k-2}\big]
%\prod_{i=k_0}^{k-1}  \big(1- \gamma_i \Lambda_{(1),i}\big)^2\\
& \le\EE \Big[ \|\delta_{k_0}\|   \prod_{i=k_0}^{k-1}\big(1-\gamma_i \Lambda_{(1)}(\bm{M}_i)\big)\Big] ( 1-\gamma_k \lambda_1 )
\le\ldots \\
&\le \EE \big[\|\delta_{k_0}\|\big(1-\gamma_{k_0} \Lambda_{(1)}(\bm{M}_{k_0})\big) \big]
\prod_{i=k_0+1}^k (1- \gamma_i \lambda_1 ) 
\le \EE \big[ \|\delta_{k_0}\|\big] \prod_{i=k_0+1}^k (1- \gamma_i \lambda_1 )\\
&\le (\bar{C}_\Theta^{1/2}+C_\Theta^{1/2}) \prod_{i=k_0+1}^k (1- \gamma_i \lambda_1 ).
\end{align*}
\end{remark}

\begin{remark}
If we assume that $\gamma_i \lambda_2 \le 1$ for all $i \in \mathbb{N}_0$ and 
$\sum_{i=1}^\infty \gamma_i^2 < \infty$, then we can prove
Lemma~\ref{lemma_bound} in another way: first  take the expectation of the second power 
of the relation (\ref{eq:after_triangle_lp}) with $k_0=0$ to establish that 
$\EE \|\delta_k\|^2<C$ is uniformly bounded in $k\in\mathbb{N}_0$, and then 
$\EE \|\hat{\theta}_k\|^2 \le 2 \EE \|\delta_k\|^2+2\EE \|\theta_k\|^2 \le 2(C+C_\Theta)$, 
by (\ref{C_Theta}). %We skip technical details.
\end{remark}

\begin{remark}
\label{rem_version_theorem}
One can try to establish a version of Theorem \ref{theo:bound} where, instead of  
(\ref{eq:Lambda}), one assumes 
\begin{equation}
\label{different_bound}
\lambda_1 \le \Lambda_{(1)} \big(\mathbb{E} [\bm{M}_k |\mathcal{F}_{k-2}]\big) \le
 \Lambda_{(d)} \big(\mathbb{E} [\bm{M}_k |\mathcal{F}_{k-2}]\big) \le \lambda_2 \quad \mbox{almost surely}.
\end{equation}
The point is that there may be situations with certain gain functions 
when (\ref{eq:Lambda}) does not hold but (\ref{different_bound}) does; 
see Remark \ref{rem_mouilines} below.
%Such a theorem can be proved. 
The idea of the proof would be to first introduce $\bar{\bm{M}}_i = \mathbb{E} [\bm{M}_i |\mathcal{F}_{i-2}]$,
$i\in \mathbb{N}_0$, and then, beginning with the relation (\ref{eq:recursed}), work with
the representation $\bm{M}_i =\bar{\bm{M}}_i + \bm{M}_i -\bar{\bm{M}}_i$ instead of just $\bm{M}_i$, using 
the relation (\ref{different_bound}) for $\bar{\bm{M}}_i$ and the fact that  
$\{\bm{M}_k-\bar{\bm{M}}_k, k\in\mathbb{N}_0\}$, is a (matrix) martingale 
difference  sequence with respect to the filtration 
$\{\mathcal{F}_{k-1}\}_{k \in\mathbb{N}_{-1}}$.
%(\ref{eq:recursion}) as follows:
%\[
%\delta_{k+1} = \Delta\theta_k + \gamma_k D_k 
%+ (\bm{I}- \gamma_k \bm{M}_k)\delta_k
%= \Delta\theta_k + \gamma_k D_k +\gamma_k M_k\delta_k
%+ (\bm{I}- \gamma_k m_k )\delta_k.
%\]
%where $m_k = -\mathbb{E} [\bm{M}_k |\mathcal{F}_{k-2}]$ and $M_k =-(\bm{M}_k -m_k)$.
We will not pursue this %approach 
here.
\end{remark}

Imposing  somewhat stronger versions of conditions (A1) and (A2) enables us to 
derive a similar non-asymptotic bound for the expectation of   
$\|\delta_{k+1}\|_p^p$ for all $p\ge 1$. Of course, the bigger $p$, the bigger the 
constants involved  in the bound.
The next theorem is a strengthened version of the previous result.
\begin{theorem} 
\label{theo:bound2}
Suppose that the conditions of Theorem \ref{theo:bound} are fulfilled.
If, in addition (to assumption (A1)),  $\Lambda_{(1)}(\bm{M}_i) \ge\lambda_1$ and
$\|G_i(\hat{\theta}_i,\bm{X}_i)\| \le \bar{G}$ (instead of (A2)) 
almost surely for all $i=k_0\ldots, k$, then for any $p\ge 1$
%\begin{equation}
\begin{align}
\EE \|\delta_{k+1}\|_p^p &
\le 
C'_1 \EE \| \delta_{k_0} \|_p^p\exp\Big\{-\!p\lambda_1\!\sum_{i=k_0}^k\gamma_i\Big\}  \notag\\
& \label{eq:boundAS}
\quad\quad 
+C'_2\Big[ \sum_{i=k_0}^k \gamma_i^2\Big]^{p/2}
+ C'_3\max_{k_0\le i\le k}\EE \|\theta_{i+1}-\theta_{k_0}\|_p^p,
\end{align}
%\end{equation}
where $C'_1 = 3^{p-1}K_p^p$, $C'_2 =3^{p-1}2^p d
B_p \bar{G}^p \big(1+K_p^2\lambda_2/\lambda_1\big)^p$, 
$C'_3 = 3^{p-1}\big(1+K_p^2\lambda_2/\lambda_1\big)^p$ 
and $K_p=K_p(d)$ is the constant from Lemma \ref{lemma:eig}.
% with $C$ as in (A2).
%If the algorithm is truncated\index{Truncation} then $\EE \| \delta_{k_0} \|_p^p$ is bounded, 
%otherwise $\EE \| \delta_{k_0} \|_p^p$ grows no faster then $k_0^p$.
\end{theorem}

\begin{proof}
Now we have stronger versions of assumptions (A1) and (A2):
\begin{equation}
\label{eq:a1_stronger}
0<\lambda_1\le \Lambda_{(1)}(\bm{M}_i) \le\Lambda_{(d)}(\bm{M}_i)\le \lambda_2, \quad 
\|G_i(\hat{\theta}_i,\bm{X}_i)\| \le \bar{G}, \quad i=k_0\ldots, k, 
\end{equation}
hold almost surely.  Along the same lines as for (\ref{eq:after_triangle_lp}),
by using Lemma \ref{lemma:eig},  (\ref{eq:a1_stronger}),  (\ref{eq:bound_on_coefs})
and the elementary inequality $1-x\le e^{-x}$, we obtain that 
\begin{align*}
\|\delta_{k+1}\|_p &\leq
K_p \|\delta_{k_0}\|_p \prod_{i=k_0}^k  (1- \gamma_i \lambda_1)+\max_{k_0\le i\le k}\|C_i\|_p
\Big[1+K_p^2\sum_{i=k_0}^{k-1}\gamma_{i+1} \lambda_2\!\!\! \prod_{j=i+2}^k (1- \gamma_i \lambda_1) \Big]\\
& \le K_p \|\delta_{k_0}\|_p \exp\Big\{-\lambda_1\sum_{i=k_0}^k\gamma_i\Big\}
+ \Big[1+\frac{K_p^2\lambda_2}{\lambda_1}\Big] \Big(\max_{k_0\le i\le k} \|A_i\|_p+\max_{k_0\le i\le k} \|B_i\|_p\Big)
\end{align*}
almost surely, where constant $K_p=K_p(d)$ is from Lemma \ref{lemma:eig}. 
Take now the $p$-th power of both sides of the inequality and apply 
the H\"older inequality
$|\sum_{i=1}^m a_i|^p\le m^{p-1}\sum_{i=1}^m |a_i|^p$ for $m=3$ to get
\begin{align*}
\|\delta_{k+1}\|_p^p &\leq
3^{p-1} K_p^p \|\delta_{k_0}\|_p^p \exp\Big\{-p\lambda_1\sum_{i=k_0}^k\gamma_i\Big\} \\
& \quad\quad + 3^{p-1} \Big(1+\frac{K_p^2\lambda_2}{\lambda_1} \Big)^p 
\Big(\max_{k_0\le i\le k} \|A_i\|_p^p + \max_{k_0\le i\le k} \|B_i\|_p^p\Big).
\end{align*}

Recall that the sequence $\big\{ \sum_{j=k_0}^i \gamma_j D_j,\, i \ge k_0 \big\}$
 is a martingale\index{Martingale} with respect to the filtration $\{\mathcal{F}_i, i\ge k_0\}$ 
and that the coordinates of $D_j$ verify $|D_{jl}| \le 2 \| G_j\| \le 2 \bar{G}$ almost surely, 
$l=1,\ldots d$, $j=k_0, \ldots, k$.
Applying the maximal Burkholder inequlity for $p>1$ and the Davis inequality for $p=1$
 (cf.\ \citep{Shiryaev:1996}) yields 
\begin{align*}
\EE  \max_{k_0\le i\le k} \|A_i\|_p^p
&=
\EE  \max_{k_0\le i\le k} \sum_{l=1}^d \bigg|\sum_{j=k_0}^i \gamma_j D_{jl} \bigg|^p 
\le\sum_{l=1}^d\EE  \max_{k_0\le i\le k} \bigg|\sum_{j=k_0}^i \gamma_j D_{jl} \bigg|^p\\
&\le
B_p \sum_{l=1}^d\EE  \bigg[\sum_{j=k_0}^k \gamma_j^2 D_{jl}^2 \bigg]^{p/2}
\le d B_p 2^p \bar{G}^p \bigg[\sum_{j=k_0}^k \gamma_j^2 \bigg]^{p/2},
\end{align*}
for some constant $B_p$.
One can take $B_p = ((18p^{5/2})/(p-1)^{3/2})^p$ for $p>1$, cf.\ \citep{Shiryaev:1996}. 
The second inequality of the theorem now follows by taking expectations on both sides of the 
bound on $\|\delta_{k+1}\|_p^p$ above and by using the  last inequality.
 %Note that by definition of the algorithm, for every $k\in\mathbb{N}$,
%\[
%\|\hat{\theta}_k\|_p\le\|\hat{\theta}_0\|_p + \sum_{j=1}^k \gamma_j \|G_j\|_p \le \|\hat{\theta}_0\|_p + k d^{1/p}C \Gamma
%\]
%implying
%\[
%\EE \|\delta_k\|_p^p\le 3^{p-1}(\EE \|\hat{\theta}_0\|_p^p + (k d^{1/p} C \Gamma)^p + \EE \|\theta_k\|_p^p).
%\]
%If the algorithm is started from a random vector with (component-wise) bounded $p$-th 
%moments then this expectation grows no faster then $k^p$.
%If on the other hand the algorithm is truncated\index{Truncation} then the $\EE \|\delta_k\|_p^p$ is simply bounded.
%The second statement of the theorem follows by noting that $\|C_i\|_p^p\le 2^{p-1}\|A_i\|_p^p + 2^{p-1}\|B_i\|_p^p$ 
%and noting $B_i$ is a telescopic sum.
%As before, by treating the summation involving $\Delta\theta_i$ and the summation involving $\gamma_iD_i$ separately, 
%we conclude that the maximum of the increments of the signal may be replaced by a summation of the increments.
\end{proof}

\begin{remark}
 One can derive a similar result for the $\EE\|\delta_{k+1}\|^p$, by simply taking the $p$-th 
power of the inequality (\ref{eq:after_triangle_lp}) and then proceeding in the same way as in the proof of
Theorem \ref{theo:bound2}, with minor modifications in the argument for the martingale $A_i$.
 
Once a bound on $\EE\|\delta_{k+1}\|^p$ is established, one can use it for proving 
Theorem \ref{theo:bound2}  in another way. Namely, since
$\|x\|_s \le \|x\|_r \le d^{1/r-1/s}\|x\|_s$ for any $x\in \mathbb{R}^d$ 
and $s\ge r\ge 1$,  $\|\delta_{k+1}\|_p \le R_2^p \|\delta_{k+1}\|$, with $R_2^p=1$ 
if $p\ge 2$ and $R_2^p=d^{(2-p)/(2p)}\le d^{1/2}$ if $1\le p< 2$.
Thus, a bound for $\EE\|\delta_{k+1}\|^p_p$ will immediately follow from the obtained bound for 
$\EE\|\delta_{k+1}\|^p$. The bound will be of the same form as in Theorem \ref{theo:bound2}, but
with different constants $C_1, C_2, C_3$.
%In doing so, we reduced the consideration  to the bound for $\|\delta_{k+1}\|$.
\end{remark}

\begin{remark}
\label{rem:close_parameter}
Consider the following situation, which we will call Case I.
Suppose we are not interested in tracking the, say, \emph{natural} parameter 
$\theta_k$ of the model, but rather some other time-varying parameter 
$\theta^*_k$,  which is also assumed to be predicable with respect to the filtration 
$\{\mathcal{F}_k\}_{k\in\mathbb{N}_{-1}}$. 
Denote  $\Delta\theta^*_k = \theta^*_k-\theta^*_{k+1}$, $k\in\mathbb{N}_0$.
%which is, on average, close to $\theta_k$.
%may not have access to a gain function which allows us to track 
The difference $\varepsilon_k=\theta_k-\theta^*_k$, $k \in \mathbb{N}_0$, 
can be seen as an approximation error.
Similar to (\ref{eq:recursed}), the following expansion can be 
derived for the quantity $\delta_k^*=\hat{\theta}_k-\theta^*_k$:
\begin{align*}
\delta_{k+1}^*& 
= \delta_k^* + \Delta\theta^*_k + \gamma_kD_k - \gamma_k\bm{M}_k(\hat{\theta}_k-\theta_k)
= \Delta\theta^*_k + \gamma_k\bm{M}_k\varepsilon_k + \gamma_k D_k 
+ (\bm{I}-\gamma_k \bm{M}_k)\delta_k^*\\
&= \Big[\prod_{i=k_0}^k(\bm{I}- \gamma_i \bm{M}_i)\Big] \delta_{k_0}^* 
+ \sum_{i=k_0}^k \Big[\prod_{j=i+1}^k(\bm{I}- \gamma_j \bm{M}_j)\Big] 
(\Delta\theta^*_i + \gamma_i\bm{M}_i \varepsilon_i+\gamma_i D_i).
\end{align*}

Now consider Case II:  we want to track 
the natural parameter $\theta_k$ but the average gain makes an error $\eta_k$, i.e.,
$g_k = -\bm{M}_k(\hat{\theta}_k-\theta_k) + \eta_k$, $k\in \mathbb{N}_0$. 
The error term $\eta_k$ may be random but must be measurable with 
respect to $\mathcal{F}_{k-1}$. Again, similar to (\ref{eq:recursed}), we can derive
\[
\delta_{k+1}	= \Big[\prod_{i=k_0}^k(\bm{I}- \gamma_i \bm{M}_i)\Big] \delta_{k_0} 
+ \sum_{i=k_0}^k \Big[\prod_{j=i+1}^k(\bm{I}- \gamma_j \bm{M}_j)\Big] 
(\Delta\theta_i +\gamma_i\eta_i + \gamma_i D_i).
\]
Now notice that Case I can actually be reduced to Case II by putting 
in the last relation $\delta_i=\hat{\theta}_i-\theta^*_i$ and 
$\eta_i=\bm{M}_i \varepsilon_i$ (where $\varepsilon_i=\theta_i-\theta^*_i$), 
$i\in \mathbb{N}_0$. Therefore, consider  only Case II from now on.

Under the conditions of Theorem \ref{theo:bound},
in the same way as for (\ref{eq:boundE}),
we can derive the following bound:  
\begin{equation}
\label{eq:boundE2}
\begin{aligned}
\EE \|\delta_{k+1}\| \le
&C_1 \exp\Big\{-\lambda_1\sum_{i=k_0+1}^k\gamma_j\Big\} 
+ C_2 \Big[\sum_{i=k_0}^{k-1} \gamma_i^2\Big]^{1/2}\\
&+C_3\EE\max_{k_0\le i\le k} \|\theta_{i+1}-\theta_{k_0}\|
+ C_3 \EE  \sum_{i=k_0}^k \gamma_i\|\eta_i\|.
\end{aligned}
\end{equation}
Similarly, under the conditions of Theorem \ref{theo:bound2}, 
\begin{equation}
\label{eq:boundAS2}
\begin{aligned}
\EE \|\delta_{k+1}\|_p^p\le
&C'_1 \exp\Big\{ -p \lambda_1\sum_{i=k_0}^k\gamma_j\Big\} 
+ C'_2 \Big[\sum_{i=k_0}^{k-1} \gamma_i^2\Big]^{p/2}\\
&+C'_3 \EE \Big[\max_{k_0\le i\le k} \|\theta_{i+1}-\theta_{k_0}\|_p 
+\sum_{i=k_0}^k\gamma_i\|\eta_i\|_p\Big]^p.
\end{aligned}
\end{equation}
Clearly, (\ref{eq:boundE2}) and (\ref{eq:boundAS2}) 
generalize the bounds of Theorems \ref{theo:bound} and  \ref{theo:bound2},  
where we had   $\eta_i=0$, $i\in \mathbb{N}_0$.

In Case I, we have $\delta_i=\hat{\theta}_i-\theta^*_i$ and $\eta_i=\bm{M}_i\varepsilon_i$
with $\varepsilon_k=\theta_k-\theta^*_k$,
$i\in \mathbb{N}_0$, in relations (\ref{eq:boundE2}) and (\ref{eq:boundAS2}). 
Noting that $\|\eta_i\|_p =\|\bm{M}_i\varepsilon_i\|_p<\lambda_2 K_p\|\varepsilon_i\|_p$
for all $p\ge 1$ and $i\in \mathbb{N}_0$, 
we can rewrite bounds (\ref{eq:boundE2}) and (\ref{eq:boundAS2}) in terms of 
$\|\varepsilon_i\|_p$ instead of $\|\eta_i\|_p$ with appropriate 
adjustments of corresponding constants.
\end{remark}

%\begin{remark}
%If we are interested in tracking $\vartheta_k=\varphi(\theta_k)$, a smooth 
%functional of the parameter $\theta_k$, then by using Taylor's Theorem, 
%our Theorem~\ref{theo:bound} above straightforwardly delivers a bound 
%on the expectation of $\|\hat{\vartheta}_k - \vartheta_k\|_p=\|\varphi(\hat{\theta}_k)-\varphi(\theta_k)\|_p$ and its powers.
%\end{remark}

\section{Construction of gain functions}
\label{sec:gains}

Any gain function for which 
conditions (A1) and (A2) hold may be used with our algorithm, and whether a particular 
gain function is suitable or not depends %exclusively 
on the model under study and the quantity that we wish to track. 
For certain types of models and quantities to track, 
there are natural choices for the gain function. Many different settings are investigated
in the literature. 
%In this section, as before, we abbreviate  $\theta_k=\theta_k(\bm{X}_{k-1})$.
In this section we consider the construction of appropriate gain functions 
to be used in the algorithm (\ref{eq:algorithm_main}) in several traditional settings.
In particular, we
relate our general approach to well known classical procedures such as Robbins-Monro and 
Kiefer-Wolfowitz algorithms and outline possible extensions. 

\subsection{Signal + noise setting}
The traditional `signal+noise' situation can be represented by the  following observation model:
\[
X_k = \theta_k + \xi_k, \quad k \in \mathbb{N}_0,
\]
where $l=d$, $\{\theta_k\}_{k \in \mathbb{N}_0}$ is a predictable process 
($\theta_k=\theta_k(\bm{X}_{k-1})$) we are interested in tracking, 
$\{\xi_k\}_{k\in \mathbb{N}_0}$ is a martingale difference noise, with respect to 
the filtration $\{\mathcal{F}_k\}_{k\in\mathbb{N}_{-1}}$. %, i.e. $\EE [ \xi_k |\mathcal{F}_{k-1}]=0$.
We use the algorithm (\ref{eq:algorithm_main}) for tracking $\theta_k$, and in this case 
we can simply take the following gain function
\begin{equation}
\label{eq:gain_regression}
G_k(\hat{\theta}_k,\bm{X}_k)	= - (\hat{\theta}_k-X_k), \quad k \in \mathbb{N}_0,
\end{equation}
since  
\[
g_k(\hat{\theta}_k,\theta_k|\bm{X}_{k-1}) = 
\EE[G_k(\hat{\theta}_k, \bm{X}_k)|\bm{X}_{k-1}] = - (\hat{\theta}_k - \theta_k), \quad k \in \mathbb{N}_0,
\]
i.e., $\bm{M}_k(\bm{X}_{k-1})=\bm{I}$.
Clearly, condition (A1) holds and condition (A2) follows as well if we assume 
$\EE \|\xi_k\|^2\le c$, $k\in\mathbb{N}_0$. Indeed, 
according to (\ref{G_bounded}), it is enough to show the boundedness of the second moment of 
$G_k$: 
\[
\EE \|G_k(\hat{\theta}_k,\bm{X}_k)\|^2 \le  
3 \big[ \EE \|\hat{\theta}_k\|^2 + \EE \|\theta_k\|^2 +  \EE \|\xi_k\|^2 \big] \le C,
\quad  k \in \mathbb{N}_0,
\]
by virtue of the H\"older inequality, Lemma \ref{lemma_bound} and (\ref{C_Theta}).
The classical  nonparametric regression model fits into this  framework so that our results 
can be applied. For example, the simplest nonparametric regression model with an equidistant design on 
$[0,1]$ is as follows: $X_k= \theta(k/n) +\xi_k$, $k=1,\ldots, n$, with 
independent noises $\xi_k$'s, $\EE \xi_k =0$, $\EE \xi^2_k=\sigma^2$;
we will return to this issue in subsection \ref{sec:variational_setups:lipschitz}.
\begin{remark}
\label{gain_autoregression}
Possibly, $\EE[X_k|\bm{X}_{k-1}] =\phi(\theta_k)$ for some smooth function $\phi$.
In this case, one should consider 
$G_k(\hat{\theta}_k,\bm{X}_k)= -(\phi(\hat{\theta}_k)-X_k)$, 
$\in \mathbb{N}_0$, so that $g_k(\hat{\theta}_k,\theta_k|\bm{X}_{k-1})=
- \big(\phi(\hat{\theta}_k) - \phi(\theta_k)\big)$, $k \in \mathbb{N}_0$,
which should be comparable to $- (\hat{\theta}_k - \theta_k)$.
Autoregressive models, for example, fall into this category (cf.\ Section~\ref{sec:examples:ARd}).
\end{remark}

\subsection{Robbins-Monro setting: tracking roots}
 
Let us turn to more dynamical situations where the observations themselves depend on our tracking sequence.
In their seminal paper, \citep{Robbins:1951} studied the problem of
finding the unique $\alpha$-root $\theta$ of a monotone function $f$, i.e., the equation
$f(x)=\alpha$ has a unique solution at $x=\theta$.
The function $f$ can be observed at any point $x$ but with noise $\xi$:
$X(x) = F(x, \xi)$ so that  $\EE F(x,\xi)=f(x)$. 
A stochastic approximation algorithm of design points converging to $\theta$
is known as classical Robbins-Monro procedure.
We now illustrate how this also fits into our general tracking algorithm scheme.

In fact, the following model essentially extends the original setup  of \citep{Robbins:1951}.
Suppose there is a time series $\{Y_k, k\in \mathbb{N}_0\}$ 
(with $\bm{Y}_k$ taking values in $\mathcal{Y}^k$) running at the background,
which is not (fully) observable. Instead, some other $d$-dimensional (related) time series 
$\{X_k, k\in \mathbb{N}_0\}$ is observed, 
which we introduce below. 
As usual, let $\mathcal{F}_k =\sigma(\bm{X}_k)$, $ k\in \mathbb{N}_0$.
Further, for a sequence of functions
$f_k: \mathbb{R}^d \times \mathcal{Y}^k\mapsto \mathbb{R}^d$, let 
a $d$-dimensional  measurable function  $\theta_k=\theta_k(\bm{X}_{k-1})$
be the unique solution of the equation $\alpha_k(\bm{X}_{k-1})=\bar{f}_k(\theta_k,\bm{X}_{k-1})$, 
where $\bar{f}_k(\theta_k,\bm{X}_{k-1})=\EE [f_k(\theta_k,\bm{Y}_k)| \mathcal{F}_{k-1}]$, 
for some measurable function $\alpha_k(\bm{X}_{k-1})$, $k\in \mathbb{N}_0$.
(Here $\bm{Y}_k$ may contain $\bm{X}_{k-1}$.)

The goal is to track the sequence $\{\theta_k\}_{k\in\mathbb{N}_0}$.
At a time moment $k\in \mathbb{N}_0$, we observe
the noise corrupted value of $f_k(\hat{\theta}_k, \bm{Y}_k)$ 
at some design point $\hat{\theta}_k$ (which can be picked on the basis of the 
previous observations $\bm{X}_{k-1}$, i.e., $\hat{\theta}_k=\hat{\theta}_k(\bm{X}_{k-1})$):
\begin{equation}
\label{robbins_monro}
X_k = f_k(\hat{\theta}_k,\bm{Y}_k)  =\bar{f}_k(\theta_k,\bm{X}_{k-1}) + \xi_k, \quad k \in \mathbb{N}_0,
\end{equation}
where $\{\xi_k\}_{k\in\mathbb{N}_0}$ is a martingale difference noise 
sequence with respect to the filtration $\{\mathcal{F}_k\}_{k \in\mathbb{N}_{-1}}$
(indeed, let simply $\xi_k = f_k(\hat{\theta}_k,\bm{Y}_k)-\bar{f}_k(\theta_k,\bm{X}_{k-1})$).
Of course, we could assume a more general model $X_k = f_k(\hat{\theta}_k,\bm{Y}_k) + \xi_k$,  $k \in \mathbb{N}_0$,
but this would not have made any principal difference, since variable $\xi_k$
 can be incorporated into the vector $Y_k$.  

Let the design points $\{\hat{\theta}_k, k \in \mathbb{N}_0\}$ in (\ref{robbins_monro}) be
determined by the algorithm  (\ref{eq:algorithm_main}) and
we want this algorithm to track $\theta_k$. 
Theorem \ref{theo:bound} is applicable if the gain $G_k$ in (\ref{eq:algorithm_main}) satisfies (A1) and (A2).
%and there 
%is a filtration $\{\mathcal{F}_k\}_{k \in\mathbb{N}_{-1}}$ such that $\mathcal{F}_k\subseteq \sigma(\bm{X}_k)$.
%For illustrative purpose let $\mathcal{F}_k=\sigma(\bm{Z}_k)$ for $\bm{Z}_k= g_k(\bm{X}_k)$ 
%with some $(\sigma(\bm{X}_k), \mathcal{F}_k)$-measurable $g_k$'s. 
%At time moment $k\in \mathbb{N}_0$, we do not fully observe this time series, 
%but only a sub-vector $Y_k$ of  $X_k$. 
%At each time point we can influence the outcomes of this time series by choosing 
%design points $\{\hat{\theta}_k, k \in \mathbb{N}_0\}$ so that we observe 
%
%
%
%
%
%Let design points $\{\hat{\theta}_k, k \in \mathbb{N}_0\}$ be
%determined by the tracking algorithm  (\ref{eq:algorithm_main}) and 
%we observe noisy values $f_k(\hat{\theta}_k, \bm{X}_{k-1})$ at those design points:
%\begin{equation}
%\label{robbins_monro}
%Y_k = f_k(\hat{\theta}_k, \bm{X}_{k-1}) + \xi_k(\bm{X}_{k-1}), \quad k \in \mathbb{N}_0,
%\end{equation}
%where $\{\xi_k(\bm{X}_{k-1}), k\in\mathbb{N}_0\})$ 
%is a martingale difference noise sequence with respect to the filtration $\{\mathcal{F}_k\}_{k \in\mathbb{N}_{-1}}$.
%We want the algorithm  (\ref{eq:algorithm_main}) to track $\theta_k$, Theorem \ref{theo:bound}
%is applicable if the gain $G_k$ in (\ref{eq:algorithm_main}) satisfies (A1) and (A2).
As in the Robbins-Monro algorithm, the gain is taken to be
\[
G_k(\hat{\theta}_k,\bm{X}_k)= -\big(X_k-\alpha_k(\bm{X}_{k-1})\big), \quad k \in \mathbb{N}_0.
\]
Then  $g_k(\hat{\theta}_k,\theta_k|\mathcal{F}_{k-1}) 
= -\big(\bar{f}_k(\hat{\theta}_k,\bm{X}_{k-1})-\alpha_k(\bm{X}_{k-1})\big)$, 
$k \in \mathbb{N}_0$, and (A2) is fulfilled if, for example, $\EE \|f_k(\hat{\theta}_k, \bm{Y}_k)\|^2 \le c$ 
and $\EE\|\xi_k\|^2 \le C$, $k \in \mathbb{N}_0$. 
Condition (A1), or equivalently (\~A1), is fulfilled if, for some $0<\lambda_1 \le \lambda_2$, 
\[
\lambda_1 \|\hat{\theta}_k-\theta_k\|^2
\le (\hat{\theta}_k-\theta_k)^T \big(\bar{f}_k(\hat{\theta}_k,\bm{X}_{k-1})
-\bar{f}_k(\theta_k,\bm{X}_{k-1})\big)  \le \lambda_2 \|\hat{\theta}_k-\theta_k\|^2,
\quad k \in \mathbb{N}_0,
\]
almost surely.
In the last display, one should recognizes the usual regularity requirements 
%(which provide identifiability of the root $\theta$) 
for the function $f$  in the classical Robbins-Monro setting: 
$d=1$, $\alpha_k(\bm{X}_{k-1})=\alpha_k$ and  $f_k(\vartheta, \bm{Y}_k)=f(\vartheta)$ 
(so that $\theta_k=\theta$ is the non-random solution of the equation
$\alpha= f(\theta)$):
$ \lambda_1 \le (f(\vartheta)-\alpha)/(\vartheta - \theta) \le \lambda_2$.
In multidimensional case, this can be seen as a generalized  identifiability  requirement 
for the sequence $\{\theta_k, k \in \mathbb{N}_0\}$. 
For example,  if $\bar{f}_k(\vartheta, \bm{x}_{k-1})$ is a differentiable mapping in $\vartheta\in\Theta$
for each $ \bm{x}_{k-1} \in \mathcal{X}^{k-1}$, then a sufficient condition for (A1) is
 positive definiteness of the Jacobian matrix of $\bar{f}_k(\vartheta, \bm{x}_{k-1})$ 
(with respect to $\vartheta$), uniformly
in $ \bm{x}_{k-1} \in \mathcal{X}^{k-1}$ and over the support of $\hat{\theta}_k$.
One can possibly relax this to a vicinity of the root $\theta_k$ under other appropriate conditions
which guarantee that $\hat{\theta}_k$ eventually gets into a neighborhood of $\theta_k$.

\begin{remark}
%Namely, vector $X_k$ consists of two subvectors: 
%$X_k = (Z_k, Y_k)$, where only subvector $Y_k$ is observed 
A particular example is 
$\alpha_k=\bar{f}_k(\theta_k) = \EE [f_k(\theta_k,Z_k)| \mathcal{F}_{k-1}]$, where $Z_k$ 
is a subvector of $Y_k$, independent of $\mathcal{F}_{k-1}$.
%The classical Robbins-Monro setting corresponds to the situation: 
%$d=1$, $\alpha_k(\bm{X}_{k-1})=\alpha$ 
%and  $f_k(\vartheta, \bm{Y}_k)=f(\vartheta)$, $k\in\mathbb{N}_0$; 
%so that $\theta_k=\theta$ is the (non-random) solution of the equation
%$\alpha= f(\theta)$.
\end{remark}

\subsection{Kiefer-Wolfowitz setting: tracking maxima}

Another classical example is the algorithm of
%Kiefer and Wolfowitz (1952)
\citep{Kiefer:1952} 
for successive estimating the maximum of a function $f$  
which can be observed at any point, but gets corrupted with a martingale difference noise
(similarly, one can formulate the problem of tracking minima of a sequence of functions).
The algorithm is based on a gradient-like method, 
the gradient of $f$  being approximated by using finite differences. There are many modifications  
of the procedure, including multivariate extensions, and they are all based on estimates of the gradient of $f$. 
The following scheme essentially contains many such procedures considered in the literature
and even extends them to a time-varying predictable maxima process $\{\theta_k, k \in \mathbb{N}_0\}$.

As in the previous subsection, suppose there is a time series 
$\{Y_k, k\in \mathbb{N}_0\}$, with $\bm{Y}_k$ taking values in $\mathcal{Y}^k$, 
running in the background, which is not (fully) observable. Instead, some other related time series 
$\{X_k, k\in \mathbb{N}_0\}$ is observed, 
which we introduce below. Let $\mathcal{F}_k =\sigma(\bm{X}_k)$, $ k\in \mathbb{N}_0$.
Suppose we are given a sequence of measurable functions 
$F_k:\Theta\times \mathcal{Y}^k\mapsto\mathbb{R}$, $\Theta\subset\mathbb{R}^d$,
$k\in\mathbb{N}_0$, 
such that the function $\bar{F}_k(\vartheta,\bm{X}_{k-1}) 
=\EE \big[F_k(\vartheta,\bm{Y}_k)|\mathcal{F}_{k-1}\big]$ has a unique maximum 
$\theta_k=\theta_k(\bm{X}_{k-1})$ on $\Theta$, i.e., 
\[
\max_{\vartheta \in \mathbb{R}^d} \bar{F}_k(\vartheta,\bm{X}_{k-1})=
\max_{\vartheta \in \Theta}\bar{F}_k(\vartheta,\bm{X}_{k-1}) =\bar{F}_k(\theta_k,\bm{X}_{k-1}), \quad
k \in \mathbb{N}_0. 
\]
Again, $\bm{Y}_k$ may contain $\bm{X}_{k-1}$.

We want to track the sequence $\{\theta_k\}_{k\in\mathbb{N}_0}$.
For that we use the sequence of design points $\{\hat{\theta}_k, k \in \mathbb{N}_0\}$ 
defined by the tracking algorithm  (\ref{eq:algorithm_main}), 
with $\mathcal{F}_k$-measurable gain functions $G_k$ to be specified later.
At a time moment $k\in \mathbb{N}_0$, we observe the so called ``noisy approximate gradients'' 
at those design points:
\begin{equation}
\label{kiefer_wolfowitz}
X_k = f_k(\hat{\theta}_k,\bm{Y}_k) + \xi_k, \quad k \in \mathbb{N}_0,
\end{equation}
where $\EE \|f_k(\hat{\theta}_k,\bm{Y}_k)\|^2 \le c$ and $\EE \|\xi_k\|^2 \le C$ for all
$k \in \mathbb{N}_0$, and $\{\xi_k\}_{k\in\mathbb{N}_0}$ is a martingale difference 
noise sequence with respect to the filtration $\{\mathcal{F}_k\}_{k\in\mathbb{N}_{-1}}$. 
The $d$-dimensional approximate gradient $f_k(\hat{\theta}_k,\bm{Y}_k)$ is not necessarily 
the exact pathwise $\vartheta$-derivative of
$F_k(\vartheta,\bm{Y}_k)$ at $\hat{\theta}_k$ (i.e., $f_k(\vartheta,\bm{Y}_k) =
\nabla_\vartheta F_k(\vartheta,\bm{Y}_k)$)
 but  such that  %can be represented as 
\begin{equation}
\label{appr_gradient}
\EE \big[f_k(\hat{\theta}_k,\bm{Y}_k) | \mathcal{F}_{k-1}\big] =
\bar{f}_k(\hat{\theta}_k,\bm{X}_{k-1})=-\bm{M}_k(\hat{\theta}_k-\theta_k) +\eta_k, \quad k \in \mathbb{N}_0,
\end{equation}
almost surely, where a symmetric positive definite matrix 
$\bm{M}_k=\bm{M}_k(\bm{X}_{k-1})$ satisfies conditions (\ref{eq:Lambda})
and $\eta_k=\eta_k(\bm{X}_{k-1})$ is some predictable approximation error.
Of course, such a representation (\ref{appr_gradient}) is always possible: simply take
$\eta_k =\bar{f}_k(\hat{\theta}_k,\bm{X}_{k-1})+\bm{M}_k(\hat{\theta}_k-\theta_k)$ for some 
symmetric positive definite matrix $\bm{M}_k$ satisfying  (\ref{eq:Lambda}); useful ones are those 
for which the $\eta_k$'s are under control -- basically, the $\eta_k$'s should be small.  

As a choice for the gain, take now $G_k(\hat{\theta}_k,\bm{X}_k)= X_k$,
so that $g_k(\hat{\theta}_k,\theta_k |\bm{X}_{k-1}) =
\EE \big[f_k(\hat{\theta}_k,\bm{Y}_k) | \mathcal{F}_{k-1}\big] =
\bar{f}_k(\hat{\theta}_k,\bm{X}_{k-1})=-\bm{M}_k(\hat{\theta}_k-\theta_k) +\eta_k$, $k \in \mathbb{N}_0$.
Clearly, (A2) holds in view of moment conditions on the quantities in (\ref{kiefer_wolfowitz}), however
(A1) is not satisfied in general since there is an approximation (possibly nonzero) term $\eta_k$ involved.
Yet, we are in the position of Remark \ref{rem:close_parameter} and thus
%since there is an approximation term $\eta_k$ involved. 
the bound (\ref{eq:boundE2}) for the tracking error holds in this case.
This bound is however useful only if the approximation errors $\eta_k$'s get sufficiently 
small as $k$ gets bigger. %Ideally, 
The most desirable situation is when $\eta_k=0$, $k \in \mathbb{N}_0$.

For each particular model of form (\ref{kiefer_wolfowitz}), one needs to determine 
conditions that should be imposed on the approximate gradients $f_k$'s in order to be able to 
claim a reasonable quality of the tracking algorithm by using our general result.
%to our general conditions (A1) and (A2).
Conditions on approximate gradients $f_k$'s from (\ref{kiefer_wolfowitz}) %in terms of matrix $\bm{M}_k$ 
which provide control on the magnitude of the approximation errors $\eta_k$'s
are comparable to the ones proposed in many papers. 
%in the original formulation of the Kiefer-Wolfowitz algorithm.
%cf.\ also with simultaneous gradient perturbation method of Spall (1992).
Examples can be found in \citep{Kushner:2003,Bach&Moulines:2011};
%Kushner and Yin (2003) and Bach and Moulines (2011)
see further references therein. 
Commonly, a finite difference form of the gradient estimate is used as noisy approximate gradient.  
Below we outline two settings.

First consider the following situation which is very close to the classical Kiefer-Wolfowitz setting:
$F_k(\vartheta,\bm{Y}_k)= F_k(\vartheta,Z_k)$ for some subvector $Z_k$ of $Y_k$,
independent of $\bm{X}_{k-1}$ defined below and we wish to maximize the function 
$\EE  F_k(\vartheta,Z_k) =\bar{F}_k(\vartheta)$. For simplicity, let  $F_k(\vartheta,Z_k)=F(\vartheta,Z_k)$
and all $Z_k$'s are identically distributed (although the generalization to the time-varying case is straightforward)
so that $\EE F(\vartheta,Z_k) =\bar{F}(\vartheta)$ is to be maximized: 
$\max_{\vartheta \in \Theta} \bar{F}(\vartheta) = \bar{F}(\theta)$. Let $\{c_k\}_{k \in \mathbb{N}_0}$ be a positive
sequence, $\{e_i, i=1,\ldots, d\}$ be the standard orthonormal basis vectors in $\mathbb{R}^d$, 
$Z_{k,i}^+$ en $Z_{k,i}^-$ have the same distribution as $Z_k$, $i=1,\ldots, d$. Denote
$\bm{Z}_k^+ = (Z_{k,1}^+,\ldots, Z_{k,d}^+)^T$, 
%$\bm{Z}_k^-= (Z_{k,1}^-,\ldots, Z_{k,d}^-)^T$. 
$\bm{F}(\hat{\theta}_k + c_k \bm{e}, \bm{Z}_k^+)=
(F(\hat{\theta}_k + c_k e_1, Z_{k,1}^+), \ldots, F(\hat{\theta}_k + c_k e_d, Z_{k,d}^+) )^T$, 
likewise for $\bm{F}(\hat{\theta}_k-c_k \bm{e}, \bm{Z}_k^-)$ and $\bar{\bm{F}}(\hat{\theta}_k\pm c_k \bm{e})$.
The observations are the noisy finite difference 
estimates of the gradient:
\[
X_k^{\pm}=\bm{F}(\hat{\theta}_k\pm c_k \bm{e}, \bm{Z}_k^\pm) + \xi_k^{\pm},
\quad k \in \mathbb{N}_0.
\]
Here $\{\xi_k^{\pm}\}_{k\in\mathbb{N}_0}$  is a martingale difference  noise sequence 
with respect to the filtration $\{\mathcal{F}_k\}_{k\in\mathbb{N}_{-1}}$, 
$\hat{\theta}_k$ denotes the $k$th estimate of the maximum point $\theta$ according to 
the algorithm (\ref{eq:algorithm_main}) with the gain $G_k(\hat{\theta}_k,\bm{X}_k)=\frac{X^+_k - X_k^-}{2c_k}$.
Then, under some regularity conditions,
\[
g_k(\hat{\theta}_k,\theta_k |\bm{X}_{k-1}) =
\frac{\bar{\bm{F}}(\hat{\theta}_k+c_k \bm{e})-\bar{\bm{F}}(\hat{\theta}_k-c_k \bm{e})}{2c_k}
=\nabla \bar{F}(\hat{\theta}_k) + \eta_k =
 -\bm{M}_k(\hat{\theta}_k-\theta) +\eta_k, %\quad k \in \mathbb{N}_0.
\]
where the magnitude of $\eta_k$ is controlled by $c_k$. Usually $c_k\to 0$ as $k\to \infty$ in an appropriate way.
%???
%Typically, $F_k(\vartheta,\bm{Y}_k)= F_k(\vartheta,Z_k)$ for some subvector $Z_k$ of $Y_k$,
%independent of $\bm{Y}_{k-1}$ 
%(so that $\bar{F}_k(\vartheta,\bm{X}_{k-1}) =\bar{F}_k(\vartheta)=\EE F_k(\vartheta,Z_k)$) 
%and  $f_k(\vartheta,\bm{Y}_k)) = f_k(\vartheta, Z_k) = \nabla_\vartheta  F_k(\vartheta,Z_k)$. 
To ensure that $\nabla \bar{F}(\hat{\theta}_k)=-\bm{M}_k(\hat{\theta}_k-\theta)$ 
(possibly with a small approximation error)
for some positive definite matrix $\bm{M}_k$ satisfying (\ref{eq:Lambda}), 
concavity of $\bar{F}$ is typically required, either global or over a compact set 
which is known to include the maximum location $\theta$. 
For example, if function $\bar{F}$ is sufficiently smooth and strongly concave, then %$\eta_k=0$ and, 
by Taylor's expansion  $-\bm{M}_k=H(\bar{F})(\theta^*_k)$, the Hessian matrix of $\bar{F}$
at some point $\theta^*_k$ between $\hat{\theta}_k$ and $\theta$, 
the relations (\ref{eq:Lambda}) are fulfilled and the approximation error $\eta_k$ is small if 
$c_k$ is small.
%Under some regularity conditions, the matrix $\bm{M}_k$ in (\ref{appr_gradient}) 
%can be related to the Hessian matrix of $F_k$ 
%at the maximum point $\theta_k$, namely, $-\bm{M}_k=H(F_k)(\theta_k)$.
%relations (\ref{eq:Lambda}) for the Hessian correspond to the strong concavity 
%of sufficiently smooth function $f_k$.
%and can be significantly relaxed by, for example, considering different types of expansions for $g_k$ 
%depending on how large the norm of $\delta_k=\hat{\theta}_k - \theta_k$ is.

Another approach (due to \citep{Spall:1992}) is based on random direction %\citep{Kushner:2003}) 
instead of the unit basis vectors.
We use the same notations as in the previous setting with one simplification: assume now that 
there are no vectors $Z_k$'s involved in the model so that $\bar{F}(\vartheta) =F(\vartheta)$.
Let $\{D_k, k\in\mathbb{N}\}$ denote a sequence of independent ($D_k$ is also assumed to be independent 
of $\bm{X}_{k-1}$) random unit vectors in $\mathbb{R}^d$.
At time moment $k \in \mathbb{N}_0$ we observe
\[
X_k^\pm=F(\hat{\theta}_k \pm c_k D_k)+\xi_k^{\pm},
\quad k \in \mathbb{N}_0,
\]
where the tracking sequence $\hat{\theta}_k$  is defined by
the algorithm (\ref{eq:algorithm_main}) with the gain function 
$G_k(\hat{\theta}_k,\bm{X}_k^+, \bm{X}_k^-, D_k)=D_k \frac{X_k^+ - X_k^-}{2c_k}$. 
\begin{remark}
Notice that one step in the previous (classical Kiefer-Wolfowitz) observation scheme requires 
in essence $2d$ observations in design points $\hat{\theta}_k\pm c_ke_i$, $i=1,\ldots, d$,
whereas only two measurements must be  made in the case of the above random direction observation scheme.
This property was the main motivation for the random direction method introduced by 
\citep{Spall:1992, Kushner:2003}.
%Spall (1992); cf.\  Kushner and Yin (2003).
\end{remark}
Then, under some regularity conditions,
\begin{align*}
g_k(\hat{\theta}_k,\theta_k|\bm{X}_{k-1})
&=\EE \Big[ D_k\frac{F(\hat{\theta}_k+c_k D_k) -F(\hat{\theta}_k-c_k D_k) }{2c_k}\big| \mathcal{F}_{k-1} \Big] \\
%+D_k\frac{\xi_k^+ -\xi_k^-}{2c_k} \big| \mathcal{F}_{k-1} \Big] %\\
&=\EE \big[ D_kD_k^T \big]\nabla F(\hat{\theta}_k)+ \eta_k
=-\bm{M}_k(\hat{\theta}_k-\theta_k) + \eta_k,
\end{align*}
where $\bm{M}_k =-\EE \big[ D_kD_k^T \big]H(F)(\theta_k^*)$ and again the magnitude of $\eta_k$ is controlled by $c_k$.
%where $\nabla^2f_k(\cdot)$ is the Hessian of $f_k(\cdot)$, $\theta_k^*\in\Theta$ and, for $\vartheta\in\Theta$,
%\[
%H_k(\vartheta) = D_kD_k^T \nabla f_k(\vartheta) - D_k\frac{f_k(\vartheta + e_k D_k)-f_k(\vartheta - e_k D_k)}{2e_k}.
%\]
The relations (\ref{eq:Lambda}) hold if, for example, we assume that the random directions were chosen 
in such a way that $\EE \big[ D_kD_k^T \big]$ are positive definite matrices and the Hessian $H(F)(\theta_k^*)$
is negative definite.
%$\nabla^2f_k(\cdot)$ 
%is positive definite over $\Theta$ and that for appropriately small $e_k$ the expectation $\EE [\|\eta_k\|_p]$ 
%is appropriately small, uniformly over $\vartheta\in\Theta$.
%These conditions are comparable to the ones in the original formulation of the Kiefer-Wolfowitz algorithm, 
%and can be significantly relaxed by, for example, considering different types of expansions for $g_k$ 
%depending on how large the norm of $\delta_k=\hat{\theta}_k - \theta_k$ is.

\begin{remark}
% In both models above, (\ref{robbins_monro}) and (\ref{kiefer_wolfowitz}), 
%we can allow the functions $f_k$ and $f'_k$ %, apart from $\hat{\theta}_k$, 
%to depend also on $\bm{X}_{k-1}$ so that $\theta_k=\theta_k(\bm{X}_{k-1})$ 
%becomes a predictable process in this case, with an appropriate adjustments of the conditions on 
%$f_k$ and $f'_k$. Another possible generalization is to allow $f_k$ and $f'_k$, in 
%(\ref{robbins_monro}) and (\ref{kiefer_wolfowitz}) respectively, to depend on 
%some random vector $V_k$ independent of $\bm{X}_{k-1}$, so that $\theta_k$ is defined by
%$\EE F_k(\theta_k,V_k)=\alpha_k$  for model (\ref{robbins_monro})   and by
%$\max_{\vartheta \in \Theta} \EE F_k(\vartheta, V_k) =\EE F_k(\theta_k,V_k)$ for model (\ref{kiefer_wolfowitz}).
A particular choice of function $F_k$ is $F_k(\vartheta,V_k) =l(\vartheta,V_k)$, $k\in\mathbb{N}_0$, 
where $V_k$'s is a sequence of observations with values on
a measurable space $\mathcal{V}_k$ and $l: \Theta\times \mathcal{V}_k  \mapsto \mathbb{R}_+$ is a loss function.
Then $\EE F_k(\vartheta,V_k)$ is the prediction risk of the predictor given by $\vartheta$. Classical examples are least squares
and logistic regression (cf.\ \citep{Bach&Moulines:2011}):
%Bach and Moulines (2011)): %(linear or non-linear through kernel methods [18, 19]), 
$F_k(\vartheta,V_k) = \frac{1}{2}\big(x_k^T \vartheta -y_k)^2$
or $F_k(\vartheta,V_k) = \log[1+\exp(-y_kx_k^T \vartheta]$, where 
$V_k=(x_k,y_k)$, $x_k\in \Theta$ and $y_k\in\mathbb{R}$, or $y_k \in \{-1, 1\}$ for logistic regression.
\end{remark}

\subsection{Tracking conditional quantiles}

Consider one more example.
Suppose $\mathcal{X}\subset\mathbb{R}$ and we would like to track the conditional 
quantile of the distribution of our observed time series $\{X_k,k\in\mathbb{N}_0\}$, i.e.,  
$\theta_k=\theta_k(\bm{X}_{k-1})$ such that 
$\theta_k=\inf\big\{x\in\mathcal{X}:F_k(x|\bm{X}_{k-1})\ge\alpha_k\big\}$, 
where the levels $\alpha_k \in (0,1)$ are of our choice and 
$F_k(x|\bm{X}_{k-1})$ is the conditional distribution function of $X_k$ 
given the past $\bm{X}_{k-1}$. Assume that this conditional distribution posses a density
$f_k(x|\bm{X}_{k-1})$. In this case it makes sense to use
%\begin{equation}
%\label{eq:gain_quantile}
$G_k(\hat{\theta}_k,\bm{X}_k)=\alpha_k - \mathbb{I}\{X_k-\hat{\theta}_k\le0\}$
%\end{equation}
in the algorithm (\ref{eq:algorithm_main}) for tracking $\theta_k$,
since  
\[
g_k(\hat{\theta}_k,\theta_k|\bm{X}_{k-1}) =
%\EE [G_k(\vartheta,X_k|\bm{X}_{k-1})|\bm{X}_{k-1}] = 
- (F_k(\hat{\theta}_k|\bm{X}_{k-1}) - \alpha_k) \approx -f_k(\theta_k^*|\bm{X}_{k-1})(\hat{\theta}_k-\theta_k),
\]
for some $\theta_k^*$ between $\hat{\theta}_k$ and $\theta_k$.
Under some mild conditions Theorem \ref{theo:bound} is applicable.
Note also that the algorithm based on this gain function only requires knowledge of the values of 
the indicators $\mathbb{I}\{X_k - \hat{\theta}_k \le 0\}$ which means that we may still track 
the required quantiles without explicitly observing $X_k$.
This problem is treated in detail for the case of independent observations in
\citep{Belitser&Serra:2013}.
%Belitser and Serra (2013).

\subsection{Gain function based on score}

For certain models it may not be obvious how gain functions can be constructed, 
especially when tracking multi-dimensional parameters.
It is therefore important to have a general procedure that can be used to construct candidate gain 
functions that can either be used directly or, if needed, modified to verify (A1) and (A2).

In this subsection we assume that we are in the framework of Remark \ref{rem:parametrized}, i.e.,
we are dealing with a parameterized model 
$\mathcal{P}_k = \mathcal{P}_k(\Theta)=\big\{\mathbb{P}_{\theta}(\cdot|\bm{x}_{k-1}): 
\,  \theta\in \Theta, \, \bm{x}_{k-1}\in \mathcal{X}^{k-1}\big\}$. Assume further
that for each $k\in\mathbb{N}$, each distribution from the family of conditional 
distributions $\mathcal{P}_k$ %=\{P_{\theta}(x|\bm{X}_{k-1}), \theta \in \Theta\}$ 
has a density with respect to some $\sigma$-finite dominating measure %$\mu$ 
and denote this conditional density by $p_{\vartheta}(x|\bm{x}_{k-1})$, 
$\vartheta = (\vartheta_1, \dots, \vartheta_d) \in \Theta\subset \mathbb{R}^d$.
Assume also that there is a common support $\mathcal{X}$ for these densities, and that for any 
$x\in\mathcal{X}$ and $\vartheta\in\Theta\subset\mathbb{R}^d$, 
the partial derivatives $\partial p_{\vartheta}(x|\bm{X}_{k-1})/\partial\vartheta_i$, $i=1,\dots,d$, 
exist and are finite, almost surely. 
As before, the ``true'' value of the time-varying  parameter 
at time moment $k\in\mathbb{N}_0$ is denoted by $\theta_k=\theta_k(\bm{X}_{k-1})$. 
Under these assumptions, the \emph{conditional} gradient vector
$\nabla_{\vartheta}\log p_{\vartheta}(x|\bm{X}_{k-1})$ 
%= \Big(\partial \log p_{\theta}(x|\bm{X}_{k-1})/\partial\theta_1, 
%\dots, \partial \log p_{\theta}(x|\bm{X}_{k-1})/\partial\theta_d \Big)
and the random matrices $I_k(\vartheta|\bm{X}_{k-1})$, $k \in \mathbb{N}_0$, with entries
\[
I_{k,{i,j}}(\vartheta|\bm{X}_{k-1})=\EE_{\theta}\Big[\frac{\partial p_{\vartheta}(x|\bm{X}_{k-1})}{\partial\vartheta_i}
\cdot \frac{\partial p_{\vartheta}(x|\bm{X}_{k-1})}{\partial\vartheta_j}\Big], \quad i,j=1,\dots,d,
\]
can be defined, almost surely.
A possible gain function for the algorithm (\ref{eq:algorithm_main}) is simply the conditional score 
of the model, i.e., the gradient vector
\begin{align}
\label{eq:gain_canonical_1}
G_k(\vartheta,\bm{X}_k)=\nabla_{\vartheta}\log p_{\vartheta}(X_k|\bm{X}_{k-1}).
\end{align}
If $I_k(\vartheta|\bm{X}_{k-1})$ is almost surely non-singular in point $\hat{\theta}_k$, then one might also consider
\begin{align}
\label{eq:gain_canonical_2}
G_k(\vartheta,\bm{X}_k)=I_k^{-1}(\vartheta|\bm{X}_{k-1}) \nabla_{\vartheta}\log p_{\vartheta}(X_k|\bm{X}_{k-1}).
\end{align}

We now outline some heuristic arguments %justify now 
why these choices are reasonable.
Take $\vartheta,\theta \in \mathbb{}R^d$. It is not uncommon for the Kullback-Leibler divergence 
$K\big(P_{\theta}(\cdot|\bm{X}_{k-1}), P_{\vartheta}(\cdot|\bm{X}_{k-1})\big)$ to be a quadratic form 
in the distance between the parameters $\theta$ and $\vartheta$, i.e., equal to a multiple of 
$(\vartheta-\theta)^T \bm{M} (\vartheta-\theta)$ for some (eventually random) positive semi-definite matrix $\bm{M}$.
Actually, this is also in general true under some regularity conditions, at least locally, in a vicinity of the ``true'' $\theta$. 
%If so, under some appropriate regularity conditions (for example, that we can interchange integration 
%and differentiation and that $\bm{M}$ does not depend on $\vartheta$), 
%$g_k(\vartheta,\theta|\bm{X}_{k-1})$ will almost surely reduce to
For example, suppose that we can interchange integration 
and differentiation and that $\bm{M}$ does not depend on $\vartheta$, then
\begin{align}
g_k(\vartheta,\theta|\bm{X}_{k-1}) =
& \int\! \nabla_{\vartheta}\log p_{\vartheta}(x|\bm{X}_{k-1}) d P_{\theta}(x|\bm{X}_{k-1}) 
= \nabla_{\vartheta} \!\int \log p_{\vartheta}(x|\bm{X}_{k-1}) d P_{\theta}(x|\bm{X}_{k-1}\notag)\\
=&  \nabla_{\vartheta}\Big(\int \log \frac{p_{\vartheta}(x|\bm{X}_{k-1})}{p_{\theta}(x|\bm{X}_{k-1})} 
d P_{\theta}(x|\bm{X}_{k-1}) + \int \log p_{\theta}(x|\bm{X}_{k-1}) d P_{\theta}(x|\bm{X}_{k-1}) \Big)\notag\\
=& \nabla_{\vartheta} \int \log \frac{p_{\vartheta}(x|\bm{X}_{k-1})}{p_{\theta}(x|\bm{X}_{k-1})} 
d P_{\theta}(x|\bm{X}_{k-1})  =  - \nabla_{\vartheta} K\big(P_{\theta}(\cdot|\bm{X}_{k-1}), P_{\vartheta}(\cdot|\bm{X}_{k-1})\big)\notag\\
\label{KL-gain}
=& -\nabla_{\vartheta} (\vartheta-\theta)^T \bm{M} (\vartheta-\theta) = - 2 \bm{M} (\vartheta-\theta).
\end{align}

The score  in principle depends on the past of the time series $\bm{X}_{k-1}$ and the previous 
argument might only be valid for a certain subset of values $\bm{X}_{k-1}$ in $\mathcal{X}^{k-1}$.
This dependence could prevent (A1) from holding.
In such cases, using a gain of the form (\ref{eq:gain_canonical_2}) might be a good alternative since 
the matrix $I_k^{-1}(\vartheta|\bm{X}_{k-1})$ acts as an appropriate scaling factor.

The dependence of the gain function on the past of the time series is in fact one of the main issues one has 
to deal with when checking (A1) and (A2). On one hand, to ensure that the gain function has, on average, 
the right direction, as required by (\ref{eq:g}), the gain will often need to depend on previous observations.
This might, however, affect either the range or the variance of the gain.
Gain functions, such as  (\ref{eq:gain_canonical_1}) and (\ref{eq:gain_canonical_2}), can be modified, 
or rescaled, to ensure that the respective conditional expectation $g_k(\vartheta,\theta|\bm{X}_{k-1})$ 
verifies the assumptions of Theorem \ref{theo:bound}.
One can for example truncate\index{Truncation} certain entries or factors in both $G_k(\vartheta,\bm{X}_k)$ 
and $I_k(\vartheta|\bm{X}_{k-1})$ to ensure that the resulting $g_k(\vartheta,\theta|\bm{X}_{k-1})$ meets the required assumptions.
Another possibility is to rescale, or directly truncate, the length of a given gain vector and consider, for example, one of the following gains
\begin{align*}
\label{eq:scaled_gains}
{\tilde G}_k(\vartheta,\bm{X}_k) &= \frac{G_k(\vartheta,\bm{X}_k)}{1+\|G_k(\vartheta,\bm{X}_k)\|},\\
{\mathring G}_k(\vartheta,\bm{X}_k) &= G_k(\vartheta,\bm{X}_k) 
\Big[1+\frac{\kappa - \|G_k(\vartheta,\bm{X}_k) \|}{\|G_k(\vartheta,\bm{X}_k)\|} 
\mathbb{I}\{\|G_k(\vartheta,\bm{X}_k)\|\ge\kappa\} \Big],\\
{\bar G}_k(\vartheta,\bm{X}_k) &= G_k(\vartheta,\bm{X}_k) \frac{\min\{s_k(\bm{X}_{k-1}), \kappa \} }{s_k(\bm{X}_{k-1})},
\end{align*}
for $G_k$ an arbitrary gain function, $\kappa>0$ and some functions $s_k:\mathcal{X}^{k-1}\mapsto\mathbb{R}_+$.
Note that ${\tilde G}_k$, ${\mathring G}_k$ and ${\bar G}_k$ all preserve the direction of $G_k$ 
and have norm bounded by respectively $1$, $\kappa$ and the norm of $G_k$, almost surely.

The gain ${\bar G}_k$ is a rescaling of $G_k$ for situations when the corresponding conditional gain $g_k$ is
of the form $g_k(\vartheta,\theta|\bm{X}_{k-1})= -s(\bm{X}_{k-1}) \bm{M}_k(\vartheta-\theta)$, where $\bm{M}_k$ 
has eigenvalues as prescribed by (A1).
Consequently the conditional rescaled  gain is 
\[
{\bar g}_k = -\min\{s(\bm{X}_{k-1}),\kappa\} \bm{M}_k(\vartheta-\theta),
\]
so that the largest eigenvalue of the matrix $\min\big(s(\bm{X}_{k-1}), \kappa\big) \bm{M}_k$ 
is almost surely upper bounded. As to the lower bound, %in (\ref{eq:Lambda}) for ${\bar g}_k$,
in certain situations it will be possible to show that $\EE[\min\big(s(\bm{X}_{k-1}), 
\kappa\big)\Lambda_{(1)}(\bm{M}_k)|\bm{X}_{k-2}]\ge c\lambda_1$ almost surely,
for some $0<c\le1$ and sufficiently large $\kappa$,
by using the fact that $\EE[\Lambda_{(1)}(\bm{M}_k)|\bm{X}_{k-2}]\ge\lambda_1$  almost surely.
This would establish condition (A1)  for the rescaled gain ${\bar G}_k$. 
%Condition (A2) follows too. Indeed
Since $\min(x,\kappa)/x\le 1$ for all $x\in \mathbb{R}_+$, then
\begin{align}
& \EE \EE \big[\|{\bar G}_k- {\bar g}_k\|^2\big|\bm{X}_{k-1} \big]\notag\\
&\quad\quad\quad 
\label{eq:scaled_gain}
=\EE\Big[ \Big(\frac{\min\big( s(\bm{X}_{k-1}), \kappa \big)}{s(\bm{X}_{k-1})}\Big)^2 
\EE \big[\|G_k - g_k\|^2\big|\bm{X}_{k-1} \big] \Big] \le
\EE \|G_k - g_k\|^2.
\end{align}
Thus, if $G_k$ verifies (A2), then so does ${\bar G}_k$.

Another possible modification one might consider is to truncate 
the iterates of the our algorithm (\ref{eq:algorithm_main}). This might be motivated by 
practical considerations in the case where the parameter being tracked has some 
physical meaning and is bounded for that reason. 
The algorithm should be restricted as well. We  can then consider an algorithm of the form
\begin{equation}
\label{eq:algorithm_truncated}
\hat{\theta}_{k+1} = \Pi_{\bar\Theta}\big(\hat{\theta}_k + \gamma_k G_k(\hat{\theta}_k,\bm{X}_k)\big), 
\quad k \in \mathbb{N}_0,
\end{equation}
where $\Pi_{\bar\Theta}(\cdot)$ acts as a projection on a convex compact set 
$\bar\Theta\supset\Theta$: $\Pi_{\bar\Theta}(\cdot)$ is an identity on 
$\bar\Theta$ and maps any point from $\bar\Theta^c$ to the closest point in $\bar\Theta$.
%If this operator is linear then we can write, for $\theta\in\bar\Theta$, $\gamma>0$ and $u$ in the range of the chosen gain function,
%\[
%\Pi_{\bar\Theta}(\theta + \gamma \bm{u}) = P (\theta + \gamma \bm{u}) = P\theta + \gamma P \bm{u} = \theta + \gamma P\bm{u},
%\]
%where $P=P(\theta,\bm{u},\gamma,\bar\Theta)$ is the matrix corresponding to the linear operator.
%If the matrices $P$ are positive definite and their eigenvalues are uniformly bounded from zero and infinity, 
%then if the gain $G_k(\bm{x},\theta|\bm{x}_{k-1})$ verifies (A1) and (A2) so will the gain 
%${\tilde G}_k(\bm{x},\theta|\bm{x}_{k-1})=P(\theta,G_k(\bm{x},\theta|\bm{x}_{k-1}),\gamma,\bar\Theta) G_k(\bm{x},\theta|\bm{x}_{k-1})$.
%
%[EXAMPLE]\\

We provide concrete examples of gain functions later in Section~\ref{sec:examples}.
In Section~\ref{sec:variational_setups}, we present  some examples 
of different types of parameter variation such that our algorithm is capable of adequately tracking 
the time-varying parameter.

\section{Variational setups for the drifting parameter}
\label{sec:variational_setups}

It is clear -- and in fact explicit in (\ref{eq:boundE}) and (\ref{eq:boundAS}) -- 
that the changes in the parameter\index{Parameter!time-changing} have 
a non-negligible contribution to the accuracy of our tracking algorithm.
This is reasonable since, if the parameter changes arbitrarily in-between observations, 
we should not expect it to the ``trackable". We should then specify how the parameter 
is allowed to vary and, based on that assumption, pick an appropriate sequence 
$\gamma_k$ which minimizes the general bounds in (\ref{eq:boundE}) or (\ref{eq:boundAS}).
In this section, we specify different settings  %what these bounds reduce to for concrete examples 
for the variation of the parameter to be tracked. These settings %examples 
refer only to how the 
parameter is assumed to change and are unrelated to the actual model in question;
examples of specific models can be found in Section~\ref{sec:examples}.

To avoid overloaded notations, we use letters $C$ and $c$ for constants 
whose values are not important to us and
which can be different in different expressions.

\subsection{Static parameter}
\label{sec:variational_setups:constant}

We assume in this section that $\theta_i(\bm{X}_{i-1}) = \theta_0$, $i\in\mathbb{N}_0$, almost surely, 
for some unknown $\theta_0\in\Theta$ so that $\Delta\theta_i=0$ (zero vector) for all $i\in\mathbb{N}_0$, 
almost surely, and we are actually in a parametric setup. In this case the second terms in both 
(\ref{eq:boundE}) and (\ref{eq:boundAS}) obviously vanish.

Take then $\gamma_i=C_\gamma i^{-1}\log i$ and for $q\in(0,1)$, 
$n_0 = [q n]$, where $[a]$ is the whole part of $a\in\mathbb{R}$.
Let $n\ge2/q=N_q$ such that $n_0\ge2$.
For any $c>0$ there is a large enough $C_\gamma$  such that, for all $n\ge N_q$,  
\[
\sum_{i=n_0}^n \gamma_i \ge C_\gamma \log n_0\sum_{i=n_0}^n \frac{1}{i}\ge 
C_\gamma \log n_0\sum_{i=n_0}^n \log\Big[1+\frac{1}{i}\Big]
= C_\gamma \log n_0 \log\Big[\frac{n+1}{n_0}\Big]
\ge c \log n.
%\frac{\log n}{2 \lambda_1},
\]
%which implies that $\exp\big\{-p\lambda_1 \sum_{j=n_0}^n \gamma_j \big\}\le n^{-p/2}$ for all $p>0$.
%Note that in the case where we have $\EE \|\delta_{n_0}\|_p^p \le C_0 n_0^p$ 
%we can take the constant $C_\gamma$ to be larger (say take $rC_\gamma$, $r>2$) in which case
%\[
%C_1\exp\Big( -p\lambda_1 \sum_{j=n_0}^n \gamma_j \Big) \le c_1 n^p n^{-rp/2} \le C_1 n^{-p/2}.
%\]
Moreover, it is easy to see that,  under the conditions of  Theorem \ref{theo:bound2}, 
$\EE \|\delta_{n_0}\|_p^p \le C_0 n_0^p$.
Thus, in both (\ref{eq:boundE}) and (\ref{eq:boundAS}) the first term can be upper bounded by
$Cn^{-c}$ for any $c>0$ by taking sufficiently large $C_\gamma$.
Next note that  %since $\sum_{j=n_0}^n\gamma_j^2 \le c (\log n)^2 n^{-1}$ for some constant $c>0$, 
$\Big[\sum_{i=n_0}^n\gamma_i^2 \Big]^{p/2}\le C (n^{-1/2} \log n)^p$, 
%so that the main term in both. %for some $C$.
We conclude that, for a sufficiently large $C_\gamma$, we can rewrite (\ref{eq:boundE}) 
and (\ref{eq:boundAS}) as respectively,
\begin{equation}
\label{constant_parameter}
\max_{n\ge N_q} \EE  \frac{\sqrt n}{\log n}\|\delta_n\| \le C\quad\mbox{and}\quad
\max_{n\ge N_q} \EE  \Big[\frac{\sqrt n}{\log n}\|\delta_n\|_p\Big]^p \le C, \quad p\ge1.
\end{equation}
If we let $n\to \infty$, this is almost (up to a log factor) parametric convergence rate, the $\log$-factor 
in the rate cannot be avoided and is in some sense a price for the recursiveness of the algorithm.

If we are in the situation of Theorem \ref{theo:bound2}, then 
by taking $p>\epsilon^{-1}$ (where $\epsilon>0$ is some small fixed number) 
and by using Markov's inequality  
and the second bound in the previous display, we derive that
\begin{align}
\sum_{n=1}^\infty P \big( n^{1/2-\epsilon} \|\hat{\theta}_n & - \theta_0\|_1 > c \big)
\le \sum_{n=1}^\infty P \big( d^{\frac{p-1}p}n^{1/2-\epsilon} \|\hat{\theta}_n - \theta_0\|_p > c \big) \notag\\
\label{eq:almost_sure}
&\le \sum_{n=1}^\infty \frac{d^{p-1}n^{p/2-p\epsilon} 
\EE  \|\delta_n\|_p^p}{c^p} \le C \sum_{n=1}^\infty \frac{(\log n)^p}{n^{p\epsilon}} < \infty.
\end{align}
In view of the Borel-Cantelli Lemma\index{Borel-Cantelli Lemma}, it follows 
that $\|\hat{\theta}_n - \theta_0\|_1\to0$ as $n\to0$ with 
probability 1 at a rate $n^{1/2-\epsilon}$. %for all $\epsilon>0$.

\begin{remark}
The particular setup presented in this section, where the parameter is fixed, might seem out 
of place since we are mainly concerned with tracking time-changing parameters.
We would like to point out that recursive algorithms in parametric situation can also be useful;
for example, the classical Robbins-Monro and Kiefer-Wolfowitz algorithms deal with the parametric case.
Recursive procedures often produce estimates in a fast, straightforward fashion.
This is an advantage especially over ``offline" estimators obtained, say, as 
solutions to a certain system, which require iterative likelihood or least squares optimization 
or are obtained via other indirect methods, a situation 
which is common when dealing with Markov models (cf. Section~\ref{sec:examples:ARd}.)
\end{remark}

\subsection{Stabilizing parameter}
\label{sec:variational_setups:stabilizing}

Suppose now that the parameter we want to track is stabilizing.
This situation might arise if the expectation of the sequence of values that the parameter 
takes is converging to some limiting value. It could also be the case that the data is being sampled
with increasing frequency from an underlying, continuous time process which depends 
on a parameter varying continuously; in this case, the parameter varies less because it has less time to change.
Regardless, we assume that $\Delta\theta_i = \theta_{i}(\bm{X}_{i-1}) - \theta_{i+1}(\bm{X}_i)$ verifies
\[
\EE  \|\Delta\theta_i \|_p^p \le \rho_i^p, \quad i\in\mathbb{N}_0,
\]
for $p\ge1$ and some positive sequence $\rho_i$.
Assume that $\rho_i = c_\rho i^{-\beta}$ for some $c_\rho>0$ and $\beta\ge0$.

Consider first the case $\beta\ge3/2$.
In this case, the variation of the parameter vanishes so quickly that we are essentially 
in the setup of the previous section, i.e., as if the parameter is constant.
Indeed, take $\gamma_i$ and $n_0$ as in the previous section.
The first and second terms in both (\ref{eq:boundE}) and (\ref{eq:boundAS}) 
can be bounded in the same way as in the previous section.
Using the relations between norms from Remark \ref{rem:norms}, 
we upper bound the third term in (\ref{eq:boundE}) by a multiple of
\begin{equation}
\label{eq:bound00}
\EE \sum_{i=n_0}^n \|\Delta\theta_i\| \le c (n-n_0) \rho_{n_0}\le C (n-n_0) n^{-\beta} \le C n^{-1/2}.
\end{equation}
Using the H\"older inequality, we upper bound the third term in  (\ref{eq:boundAS}) by a multiple of 
\begin{align}
\EE  \Big( \sum_{i=n_0}^n \|\Delta\theta_i\|_p \Big)^p&
\le (n-n_0)^{p-1} \sum_{i=n_0}^n \EE \|\Delta\theta_i\|_p^p
\le C(n-n_0)^p \rho_{n_0}^p \notag\\
\label{eq:bound01}
&\le c\big[(n-n_0)n_0^{-\beta}\big]^p \le C n^{-(\beta-1)p} \le C n^{-p/2}.
\end{align}
Clearly, in both (\ref{eq:boundE}) and (\ref{eq:boundAS}) the third term is of a smaller 
order than the second term. Thus, the relations 
(\ref{constant_parameter}) remain valid for the case $\beta\ge3/2$.
%leading to the same bounds as in the previous section.

Consider now the case $0 < \beta < \frac{3}{2}$.
Let\index{Step size} $\gamma_i = C_\gamma(\log i)^{1/3} i^{-2\beta/3}$, $n_0 = n-n^{2\beta/3}(\log n)^{2/3}$.
By using the elementary inequality $(1+x)^\alpha \le 1+\alpha x$ for 
$0<\alpha <1$ and $x\ge -1$, we obtain that for any $c>0$ there is a sufficiently large constant $C_\gamma>0$ 
such that 
\begin{align*}
\sum_{i=n_0}^n \gamma_i & \ge C_\gamma (\log n_0)^{1/3} 
\sum_{i=n_0}^n \frac{1}{i^{2\beta/3}} \ge C_\gamma (\log n_0)^{1/3} \int_{n_0}^n \frac{dx}{x^{2\beta/3}}\\ 
&=\frac{C_\gamma (\log n_0)^{1/3}}{1-2\beta/3}
\Big[n^{1-2\beta/3} -n^{1-2\beta/3} \big(1-n^{2\beta/3-1} (\log n)^{2/3}\big)^{1-2\beta/3}\Big]\\ 
&\ge\frac{C_\gamma (\log n_0)^{1/3}}{1-2\beta/3}
\Big[n^{1-2\beta/3} -n^{1-2\beta/3} \big(1-n^{2\beta/3-1} (\log n)^{2/3}(1-2\beta/3)\big)\Big]\\
& = C_\gamma (\log n_0)^{1/3} (\log n)^{2/3} \ge c \log n %\frac{\log n}{2h}
\end{align*}
for sufficiently large $n$, i.e., $n\ge N_1=N_1(\beta)$.
This yields the same upper bound for the first term in (\ref{eq:boundE}) and (\ref{eq:boundAS})
as for the static parameter, namely, $Cn^{-c}$ for any $c>0$ by taking sufficiently large $C_\gamma$.
Let us bound now the second term in (\ref{eq:boundE}) and (\ref{eq:boundAS}):
\[
\Big(\sum_{i=n_0}^n \gamma_i^2 \Big)^{1/2} \le 
C\big((\log n)^{2/3}n_0^{-4\beta/3}(n-n_0)\big)^{1/2}\le c (\log n)^{2/3} n^{-\beta/3}
\]
for  $n\ge N_2=N_2(\beta)$.
For sufficiently large  $n$ (i.e., $n\ge N_3=N_3(\beta)$) the third terms in 
(\ref{eq:boundE}) and (\ref{eq:boundAS}) are bounded similarly to (\ref{eq:bound00}) and 
(\ref{eq:bound01}) by, respectively,
\[
\EE \sum_{i=n_0}^n \|\Delta\theta_i\| \le c (n-n_0) \rho_{n_0}\le C (n-n_0) n^{-\beta} 
\le C (\log n)^{2/3} n^{-\beta/3}
\]
and
\[
\EE  \Big(\sum_{i=n_0}^n \|\Delta\theta_i\|_p \Big)^p\le c \big((n-n_0)n_0^{-\beta}\big)^p 
\le  C\big((\log n)^{2/3} n^{-\beta/3}\big)^p.
\] 
Finally we obtain that for $0<\beta <3/2$ and sufficiently large constant 
$C_\gamma$ in the algorithm step $\gamma_i = C_\gamma(\log i)^{1/3} i^{-2\beta/3}$, 
(\ref{eq:boundE}) and (\ref{eq:boundAS}) can be rewritten as respectively
\[
\max_{n\ge N_\beta} \EE \frac{n^{\beta/3}}{(\log n)^{2/3}} \|\delta_n\|_2 \le C
\quad \mbox{and} \quad \max_{n\ge N_\beta} \EE \Big[\frac{n^{\beta/3}}{(\log n)^{2/3}} \|\delta_n\|_p\Big]^p
\le c,
\]
where $N_\beta=\max(N_1,N_2,N_3)$ is the burn-in\index{Burn-in} period of the algorithm.

\begin{remark} 
If we choose $\gamma_i = C_\gamma(\log i)^{\alpha_1} i^{-\alpha}$ and $n_0 
= n - n^{\alpha}(\log n)^{\alpha_2}$, $0<\alpha<1$, $\alpha_1, \alpha_2\ge 0$, 
$\alpha_1+\alpha_2 \ge 1$ in case $0<\beta <3/2$, then we get the following 
bound of the convergence rate: for sufficiently large $n$ and sufficiently large  constant $C_\gamma$ 
\[
\EE \|\delta_n\|_p^p \le C\Big(n^{-\min\{\beta-\alpha,\alpha/2\}}
(\log n)^{\max\{\alpha_2, \alpha_1+\alpha_2/2\}}\Big)^p.
\]
Thus, the choices $\alpha=2\beta/3$, $\alpha_1=1/3$, $\alpha_2=2/3$ are 
optimal in the sense of the minimum of the right-hand side of the above inequality.
\end{remark}

\begin{remark}
Much in the same way as for  (\ref{eq:almost_sure}), we can establish that  for any 
$\epsilon>0$, $\lim_{n\to \infty} n^{\beta/3-\epsilon}\|\delta_n\|_1 = 0$ with probability 1.
\end{remark}

Finally, consider the case  $\beta=0$, i.e., we assume the following weak requirement:
$\EE \|\Delta\theta_i\|_p^p \le c$, $i\in \mathbb{N}_0$, for some uniform constant $c$.
Take  $n-n_0 = N$, $\gamma_i=\gamma$ for some $N\in \mathbb{N}$, $\gamma>0$.
Then Theorem \ref{theo:bound} implies  that 
\[
\max_{n\ge N}  \EE \|\delta_n\|_p^p \le C_1 e^{-p h N\gamma} + C_2 N^{p/2} \gamma^p+C_3 N^p c = D.
\]
We thus have that the algorithm will track down the parameter in the proximity of size $D$, 
which we can try to minimize by choosing appropriate constants $N$ and $\gamma$.

\subsection{Lipschitz signal with asymptotics in the sampling frequency}
\label{sec:variational_setups:lipschitz}

We consider now a different setup where we assume that the parameter is changing, 
on average, like a Lipschitz\index{Function!Lipschitz} function.
In this setup we let the time series (\ref{model}) be sampled from a continuous 
time process $X_t$, $t\in[0,1]$, which we observe with frequency $n$.
This means that we deal with a triangular sequence of models, i.e., 
for each $n\in\mathbb{N}$ we have a different model, namely,
\begin{equation}
\label{model_continuous}
 X_0^n\sim P_{\theta_0^n}, \qquad X_k^n|\bm{X}_{k-1}^n 
\sim P_{\theta_k^n}(\cdot|\bm{X}_{k-1}^n),\quad k\le n \in \mathbb{N},
\end{equation}
where the parameter $\theta_k^n=\theta_k^n(\bm{X}_{k-1}^n)$ verifies, 
for some $p\ge 1$, $\kappa_{d,p}<\infty$,  $0<\beta\le1$, 
\[
\EE \|\theta_k^n(\bm{X}_{k-1}^n) -\theta_{k_0}^n(\bm{X}_{k_0-1}^n)\|_p^p 
\le \kappa_{d,p}^p\Big(\frac{k-k_0}n\Big)^{\beta p}.
\]
Assume for example that $\theta_k^n(\bm{X}_{k-1}^n) = \vartheta(k/n)$, where $\vartheta(\cdot)\in
\mathcal{L}(L, \beta)=\{g(\cdot) : \|g(t_1)-g(t_2)\|_1 \le L|t_1-t_2|^\beta, t_1,t_2\in[0,1]\}$ 
for some $0<\beta\le1$ and $L>0$, a space of vector valued Lipschitz functions.

Let $\gamma_k = C_\gamma (\log n)^{(2\beta-1)/(2\beta+1)}n^{-2\beta/(2\beta+1)}$  
(constant in $k$) for $k=1, \dots, n$, and
\[
k_0=k_0(n) = k-(\log n)^{2/(2\beta+1)}n^{2\beta/(2\beta+1)},
\]
for $k\ge K_n=(\log n)^{2/(2\beta+1)}n^{2\beta/(2\beta+1)}$.
Note that for $K_n/n\to0$ as $n\to\infty$ for any $0<\beta\le1$.
We have
\[
\sum_{i=k_0}^k \gamma_i=C_\gamma (\log n)^{(2\beta-1)/(2\beta+1)}n^{2\beta/(2\beta+1)} (k-k_0) 
\ge C_\gamma \log n, %\ge \frac{\log n}{3\lambda_1},
\]
so that once again the first term in (\ref{eq:boundE}) and (\ref{eq:boundAS})
can be upper bounded by $Cn^{-c}$ for any $c>0$ by taking sufficiently large $C_\gamma$.
As to the second term, we evaluate
\[
\Big(\sum_{i=k_0}^k \gamma_i^2\Big)^{1/2}	\le
C(\log n)^{\frac{2\beta-1}{2\beta+1}}n^{-\frac{2\beta}{2\beta+1}}(k-k_0)^{1/2} 
=C(\log n)^{\frac{2\beta}{2\beta+1}}n^{-\frac{\beta}{2\beta+1}}.
\]
From our assumption on the variation of the parameter, we have
\begin{align*}
\max_{k_0\le i \le k}\EE \|\theta_{i+1}^n-\theta_{k_0}^n\|_p^p  \le 
c \Big(\frac{k-k_0}n\Big)^{\beta p} \le 
C\Big[(\log n)^{\frac{2\beta}{2\beta+1}} n^{-\frac\beta{2\beta+1}} \Big]^p.
\end{align*}
Combining these three bounds, we get that (\ref{eq:boundE}) (we also need  
the relations between norms from Remark \ref{rem:norms}) and (\ref{eq:boundAS}) imply
\begin{align*}
\sup_{\vartheta\in\mathcal{L}(L,\beta)} \max_{i\ge K_n} \EE  \|\delta_i\| 
&\le C (\log n)^{\frac{2\beta}{2\beta+1}}n^{-\frac{\beta}{2\beta+1}}, \\
\sup_{\vartheta\in\mathcal{L}(L,\beta)} \max_{i\ge K_n} \EE  \|\delta_i\|_p^p 
&\le C \Big[(\log n)^{\frac{2\beta}{2\beta+1}}n^{-\frac{\beta}{2\beta+1}} \Big]^p.
\end{align*}

\begin{remark}
If we consider step sizes of the form $\gamma_k= C_\gamma (\log n)^{\alpha_1}n^{-\alpha_2}$,
the above proposed choices of $\alpha_1$ and $\alpha_2$ are optimal in the sense of tracking error minimum.
\end{remark}
\begin{remark}
Note that the obtained convergence rate (the asymptotic regime: the observation frequency $n\to \infty$) 
coincides, up to a log factor, with the minimax rate of convergence in the problem 
of estimating nonparametric regression function over Lipschitz functional class 
$\mathcal{L}(L, \beta)$.
\end{remark}

\section{Some applications of the main result}
\label{sec:examples}

%% DOES NOT WORK, THE MARKOV CHAIN IS IN THE ERRORS AND WE DO NOT OBSERVE THOSE SO WE CANT COMPUTE THE GAIN
%\subsection{Tracking an MA(d) parameter}
%\label{sec:examples:MAd}
%
%
%[WHAT IS THE FILTRATION IN THIS CASE? GENERATED BY THE ERRORS?]
%
%
%Consider the Moving Averages models MA($d$), with a drifting mean and drifting moving average parameters,
%\begin{align}
%\label{eq:ma_q_model}
%X_k	& =	\mu_k	+	\xi_k + \sum_{i=1}^d \theta_{k,i}\xi_{k,i}, \quad	k\in\mathbb{N},
%\end{align}
%with $\xi_k$, $k\in\mathbb{Z}$ are independent errors with mean zero and variance 1, not necessarily identically distributed.
%Lets assume that $\mu_t\in\mathbb{R}$, and for each $i=1, \dots, d$, $k\in\mathbb{N}$, $0 < m \le \theta_{k,i} \le M < \infty$.
%
%
%Note that for all $k\in\mathbb{N}$, $\EE X_k = \mu_k$ and $\EE (X-\mu_k)^2 = 1 
% + \|\theta_k\|_2^2$ with $\theta_k = (\theta_{k,1}, \dots, \theta_{k,d})$.
%Consider then the gain
%\begin{equation}
%\label{eq:ma_q_gain}
%G\big( x, \mu, \theta \big)	=
%\left(
%\begin{array}{c}
%\frac{x-\mu}{1+\|\theta\|_2^2}\\
%\theta_1\frac{(x-\mu)^2 - 1 - \|\theta\|_2^2}{\big( 1+\|\theta\|_2^2 \big)^2}\\
%\vdots\\
%\theta_d\frac{(x-\mu)^2 - 1 - \|\theta\|_2^2}{\big( 1+\|\theta\|_2^2 \big)^2}
%\end{array}
%\right).
%\end{equation}
%The gain has expectation

In this section we present some examples of particular models to which our algorithm may be applied.
We start with two toy examples and present thereafter some more involved examples.
The toy examples illustrate the type of results that can be obtained from our main result and its extensions, 
how a gain function can be picked and modified, and how conditions (A1) and (A2) are checked.

\subsection{Tracking the intensity function of a Poisson process}
\label{sec:examples:poisson}

Suppose we are monitoring $n\in\mathbb{N}$ independent Poisson processes on $[0,1]$ with unknown 
intensity function $\lambda(\cdot)$. %, for fixed $n\in\mathbb{N}$.
This is equivalent to observing $N(t)=N(t,n)$, a Poisson process with intensity $n\lambda(t)$, $0\le t \le1$.
We would like to track the intensity function $\lambda(\cdot)$ which is assumed to be %uniformly 
upper bounded by $L$.
%is a member of the Lipschitz class $\mathcal{L}(L, \beta)$.We will also assume that the intensity function $\lambda(\cdot)$

Assume that we observe the process with frequency $n$, in that our observations are 
$X_k^n = N(k/n)$, so that for each $n\in\mathbb{N}$ we have a Markov model
\[
X_0^n = 0,	\quad	X_{k+1}^n|X_k^n \sim P_{\theta_k^n}(\cdot | X_k^n) = P_{\theta_k^n}(\cdot - X_k^n), \quad k=1, \dots, n,
\]
where $P_\theta(\cdot)$ represents a Poisson law with parameter $\theta\in\mathbb{R}^+$.
From now on, we will skip the dependence on $n$ for notational simplicity: 
write $X_k$ instead of $X_k^n$, $\theta_k$ instead of $\theta_k^n$ etc.
Introduce the conditional, shifted Poisson mass function given by
\[
p_\theta(x|y) =\frac{e^{-\theta}\theta^{x-y}}{(x-y)!}, \quad x,y\in\mathbb{N}, \; x\ge y.
\]
The moving parameter is given by
$\theta_k=\theta_k^n=\int_{(k-1)/n}^{k/n} n \lambda(t) dt$, $k=1, \dots,n$,
which is the average of function $\lambda(t)$ over the interval $[(k-1)/n, k/n]$. 
Assume that $\lambda(t)$ is continuous, then 
$\theta_k^n \approx \lambda(k/n)$ for $n$ large enough.

Consider now the gain function $G_k$ of the type (\ref{eq:gain_canonical_2})  
for the algorithm (\ref{eq:algorithm_main}) so that
\begin{align}
\label{eq:gain_poisson}
G_k(\hat{\theta}_k,\bm{X}_k) & =X_k - X_{k-1} - \hat{\theta}_k,\\
g_k(\hat{\theta}_k,\theta_k|\bm{X}_{k-1})& 
=\EE [X_k - X_{k-1}- \hat{\theta}_k |\bm{X}_{k-1}] =- (\hat{\theta}_k-\theta_k), \notag
\end{align}
%with $\EE _{\theta_k}[\,\cdot\,|X_{k-1}^n]$ the expectation with respect to $p_{\theta_k}(\cdot|X_{k-1}^n)$.
It follows that
\[
\EE |G_k(\hat{\theta}_k, \bm{X}_k) - g_k(\hat{\theta}_k,\theta_k|\bm{X}_{k-1})|^2 \le 
2\EE |X_k- X_{k-1}|^2 +2 |\theta_k|^2 \le C,
\]
since $\sup_{t\in[0,1]}\lambda(t)\le L$. We thus conclude 
that the gain function (\ref{eq:gain_poisson}) satisfies both (A1) and (A2).

This gain function can now be used for the three setups outlined in Section~\ref{sec:variational_setups} 
and we can attain the rates indicated there.
For a constant intensity function $\lambda(\cdot)=\theta$, $0<\theta\le L$, 
the algorithm will simply estimate the parameter of the underlying homogeneous Poisson process $\theta$,
%the parameter of the model $\theta_k^n$ reduces to the constant $\theta$ and we simply track the rate of the process.
because we matched the sampling frequency $1/n$ with the sample size $n$.
If we had sampled the process with frequency, say, $2/n$, then $\theta_k=2\theta$ and
the algorithm would track $2\theta$ and not $\theta$.
The tracking sequence would then have to be rescaled 
by a factor $1/2$ to obtain a tracking sequence for $\theta$ itself.

In the setup where we assume that the parameter is stabilizing, take $n=1$ so that
$\theta_k=\int_{k-1}^k \lambda(t)\, dt$ is the 
mean number of events per time unit $[(k-1)/n, k/n]$. Note that
\[
|\Delta\theta_k|=\Big|\int_{k-1}^k \lambda(t)dt - \int_k^{k+1} \lambda(t) dt\Big| 
=|\theta_k- \theta_{k+1}|,
\]
and the average number of events per time unit will stabilize in time if, for example, $\lambda(t) \to \lambda$ as $t\to \infty$.
%We then assume that the average number of events is stabilizing 
%in such a way that the previous display is upper bounded by say 
%$c_\beta k^{-\beta}$, for $\beta\ge 0$ and $c_\beta>0$.
The algorithm will then track the mean number of events per time unit.

We can also assume that the intensity function $\lambda(\cdot)$ belongs 
to $\mathcal{L}(L, \beta)=\{g(\cdot) : |g(t_1)-g(t_2)| \le L|t_1-t_2|^\beta, t_1,t_2\ge0\}$ 
for some $0<\beta\le1$ and $L>0$. Let $\vartheta_k=\vartheta_k^n = \lambda(k/n)$, $k,n\in\mathbb{N}$.
It follows that
\begin{align*}
\big|\Delta\vartheta_k\big|& =\big|\lambda\big(k/n\big) - \lambda\big((k+1)/n\big)\big| \le L\, n^{-\beta},\\
\big|\theta_k-\vartheta_k\big|& =\Big| \int_{(k-1)/n}^{k/n} n \lambda(t)\, dt - \lambda(k/n) \Big| 
\le n \int_{(k-1)/n}^{k/n} \big|\lambda(t)-\lambda(k/n)\big|\, dt \le L\, n^{-\beta}.
\end{align*}
The tracking sequence based on the gain (\ref{eq:gain_poisson}) will then track the sequence 
$\vartheta_k=\lambda(k/n)$, $k,n\in\mathbb{N}$ (as well as $\theta_k$) with 
the asymptotics seen in Section \ref{sec:variational_setups} (cf.\ Remark~\ref{rem:close_parameter}).

\subsection{Tracking the mean function of a conditionally Gaussian process}
\label{sec:examples:gaussian}

Assume that we observe, with fixed frequency $n\in\mathbb{N}$, a process $X(t)$, $t\ge 0$, 
taking values on $\mathcal{X}\subset\mathbb{R}^d$, $d\in\mathbb{N}$.
The observations available  up to time moment $k/n$ is a random vector 
$\bm{X}_k=\bm{X}_k^{n} = \big(X_0, X_1, \dots, X_1\big)$, with $X_k =X_k^n= X(k/n)$.
We again skip the dependence on $n$, although all the quantities below do depend on $n$.
The increments $X_k- X_{k-1}$ are assumed to be conditionally Gaussian 
in the sense that given the past of the process, each increment has a multivariate normal distribution:
\[
X_0 \sim N\big(\theta_0,\bm{\Sigma}_0\big),\quad	
X_{k+1}| \bm{X}_k \sim N\big(\theta_k(\bm{X}_{k-1}),\bm{\Sigma}_k(\bm{X}_{k-1})\big),
\quad k=1,\dots,n.
\]
The dependence on the past in the model comes from the fact that both the mean and the 
covariance processes of the above conditional distributions are predictable,
i.e., $\theta_k=\theta_k(\bm{X}_{k-1})$ and $\bm{\Sigma}_k=\bm{\Sigma}_k(\bm{X}_{k-1})$, $k\in\mathbb{N}_0$,
%processes (of corresponding dimensions)  
with respect to the filtration $\{\mathcal{F}_k\}_{k \in\mathbb{N}_{-1}}$.

If the covariance structure of the process is known, 
we can use the gain\index{Gain function} (\ref{eq:gain_canonical_1}) which verifies 
 \begin{align}
\label{eq:gain_gaussian1}
G_k(\vartheta,\bm{X}_k)
=\bm{\Sigma}_k^{-1}(X_k - \vartheta) \quad \mbox{and} \quad
g_k(\vartheta,\theta_k|\bm{X}_{k-1})
=-\bm{\Sigma}_k^{-1} (\vartheta - \theta_k).
\end{align}
For this gain, we assume that almost surely
\[
0<\lambda_1 \le \Lambda_{(1)}(\bm{\Sigma}_k) \le \Lambda_{(d)}(\bm{\Sigma}_k) \le \lambda_2<\infty, \quad k \in \mathbb{N}_0,
\]
for some positive $\lambda_1 <\lambda_2$.
We then obtain that 
\[
\EE \|G_k(\hat{\theta}_k, \bm{X}_k) - g_k(\hat{\theta}_k,\theta_k|\bm{X}_{k-1})\|^2= 
\EE \|\bm{\Sigma}_k^{-1}(X_k - \theta_k) \|^2 \le \big(\lambda_2/\lambda_1\big)^2 = C, 
\]
and assumptions (A1) and (A2) are thus met for the gain from (\ref{eq:gain_gaussian1}).

Now suppose that the covariance matrix of the process is unknown or difficult to invert.
%or that the assumption on the eigenvalues of the covariance matrix does not hold.
Then we can use the gain (\ref{eq:gain_canonical_2}), so that 
\begin{equation}
\label{eq:gain_gaussian2}
G_k(\vartheta,\bm{X}_k)=\bm{X}_k - \vartheta, \quad
g_k(\vartheta,\theta_k|\bm{X}_{k-1})=-(\vartheta - \theta_k).
\end{equation}
Clearly, assumptions (A1) and (A2) are again met for the gain from (\ref{eq:gain_gaussian2}) 
if $\Lambda_{(d)}(\bm{\Sigma}_k) \le C$ for some $C>0$, $k \in \mathbb{N}_0$, almost surely.

The results of Section~\ref{sec:variational_setups} can be applied to the algorithm based 
on the gain functions presented above for all three considered asymptotic regimes:
constant parameter process, stabilizing (on average) process and Lipschitz on average.

\begin{remark}
Although designed for different frameworks, it is interesting to compare the above 
resulting tracking algorithm with the famous {\em  Kalman filter}.
For simplicity, consider the one dimensional situation. Suppose we observe
\begin{equation}
\label{model_kalman}
 X_k= \theta_k + \xi_k, \quad \xi_k \sim N(0, \sigma^2_\xi),  \quad \quad k \in \mathbb{N},
\end{equation}
where the parameter of interest $\theta_k$, evolves according to 
\[
\theta_k = \theta_{k-1} + \delta_k \varepsilon_k \quad
\varepsilon_k \sim N(0,1),  \quad \quad k \in \mathbb{N},
\] 
with $\theta_0\sim N(m_0,\sigma^2_0)$.
At each step, the initial state and the noises
$\theta_0, \xi_1, \ldots, \xi_k, \varepsilon_1, \ldots, \varepsilon_k $ are assumed to be mutually independent.
One can show (by combining both prediction and update steps) that the Kalman filter in this case reduces to
% in the same way as on page 113 in the (russian) book "Theory of stochastic processes", by Pankov and Miller 
\begin{align}
\label{kalman_algorithm}
\hat{\theta}_k &= \hat{\theta}_{k-1} + \gamma_k(X_k- \hat{\theta}_{k-1}), \quad
\hat{\theta}_0 = %\mathbb{E}\theta_0=
m_0, \qquad k \in \mathbb{N},\\
\label{kalman_step}
\gamma_k& = \frac{\gamma_{k-1} + \delta_k^2/\sigma^2_\xi}{\gamma_{k-1} + \delta_k^2/\sigma^2_\xi+1},
\quad  \gamma_0 = %\frac{\mbox{Var}(\theta_0)}{\mbox{Var}(\xi_0)}=
\frac{\sigma^2_0}{\sigma^2_\xi}, \qquad k\in\mathbb{N}.
\end{align}
We also derive the exact expression for the mean squared error of the algorithm:
\[
\mathbb{E}(\hat{\theta}_k - \theta_k)^2 = \sigma^2_\xi \gamma_k, \quad k\in\mathbb{N}.
\] 

Coming back to our framework, suppose we have observations  (\ref{model_kalman}) 
with predictable process $\{\theta_k\}_{k\in\mathbb{N}_0}$ such that 
$\mathbb{E}(\theta_k-\theta_{k-1})^2 \le \delta^2_k$, $k\in\mathbb{N}$; 
cf.\ Section \ref{sec:variational_setups:stabilizing}. Then the Kalman filter 
(\ref{kalman_algorithm}) coincides with our tracking algorithm with the
gain (\ref{eq:gain_gaussian1}) and a particular choice of the step sequence $\gamma_k$ 
given by (\ref{kalman_step}).   One should keep in mind that 
the two frameworks are different, but it would still be interesting to compare 
the convergence rates for some particular settings for
stabilizing the parameter $\theta_k$. For example, one can consider 
$\delta_k = c k^{-\beta}$, $0<\beta< 3/2$, as in  Section \ref{sec:variational_setups:stabilizing}.
The above Kalman filter setting has more structure and we expect therefore that the rate in this case  
(which is of order $\sqrt{\gamma_n}$, with
$\gamma_n$ defined by (\ref{kalman_step}))
should be faster than 
the rate $(\log n)^{2/3}/n^{\beta/3}$ obtained in Section \ref{sec:variational_setups:stabilizing}
for our general framework.
We were however unable to solve the recursive rational difference equation (\ref{kalman_step}) 
for $\delta_k = c k^{-\beta}$. Note that the trivial case $\delta_k =0$ leads to the situation 
of a constant parameter $\theta_k= \theta$ and the sample mean $\hat{\theta}_k = \bar{X}_k$
as an estimator for that parameter.
\end{remark}

\subsection{Tracking an ARCH(1) parameter}
\label{sec:examples:ARCH1}

Consider the following ARCH($1$) model with drifting parameter 
\begin{equation}
\label{eq:arch_d_model}
X_k=(1 + \theta_k X_{k-1}^2)^{1/2} \epsilon_k, \quad k \in \mathbb{N},
\end{equation}
where $|X_0| \le 1$ almost surely, $\{\theta_k\}_{k\in\mathbb{N}}$ is predictable and 
$\{\epsilon_k\}_{k\in\mathbb{N}}$ is a martingale difference noise 
with respect to the filtration $\{\mathcal{F}_k\}_{k \in\mathbb{N}_0}$, 
$\mathbb{E}[\epsilon^2_k| \mathcal{F}_{k-1}]=\sigma_k^2$ 
for some known $\sigma_k^2$, $k\in \mathbb{N}$.
Without loss of generality assume $\sigma_k^2=1$.
Assume further that  $0\le \theta_k \le C_\Theta^{1/2}$ and
$\EE[\epsilon_k^4| \mathcal{F}_{k-1} ] \le \rho$, $k\in\mathbb{N}$, for some $\rho>0$.  

Consider the gain function
\begin{equation}
\label{eq:gain_arch_1}
G_k(\vartheta,\bm{X}_k) = \frac{\min(X_{k-1}^2, T)}{X_{k-1}^2}(X_k^2-1-\vartheta X_{k-1}^2),
\end{equation}
for some truncating constant $T>0$.
Since $\EE X_k = 0$ and $\EE[ X_k^2|\bm{X}_{k-1}] = 1+\theta_k X_{k-1}^2$,  
\[
g_k(\vartheta, \theta_k|\bm{X}_{k-1}) = 
\EE \Big[ \frac{\min(X_{k-1}^2, T)}{X_{k-1}^2}(X_k^2-1-\vartheta X_{k-1}^2) \big|\mathcal{F}_{k-1}\Big] 
= - \min(X_{k-1}^2, T)(\vartheta-\theta_k).
\]
We have that  $\min(X_{k-1}^2, T)\le T$ almost surely.
Besides,
\begin{align*}
\EE \big[\min\{X_{k-1}^2, T\}\big|\bm{X}_{k-2}\big]
&=
\EE \big[\min\{(1+\theta_{k-1}X_{k-2}^2)\epsilon_{k-1}^2, T\} \big|\bm{X}_{k-2}\big] \\
&\ge\EE \big[\min(\epsilon_{k-1}^2, T)|\bm{X}_{k-2}\big].
\end{align*}
Using the H\"older inequality and the facts that $\min(a,b)=(a+b)/2-|a-b|/2$ 
and $|a+b|^{1/2} \le |a|^{1/2}+|b|^{1/2}$, it is straightforward to check that 
\begin{align*}
2 \EE \big[\min (\epsilon_{k-1}^2, T)|\bm{X}_{k-2}\big] 
&=\EE\big[ T+\epsilon_{k-1}^2 - |\epsilon_{k-1}^2-T| \big|\bm{X}_{k-2}\big]  \\
&\ge  T+1- \big(\EE[(\epsilon_{k-1}^2-T)^2|\bm{X}_{k-2}]\big)^{1/2} \\
&\ge T+1+2T - \big(\EE[\epsilon_{k-1}^4 |\bm{X}_{k-2}]+T^2\big)^{1/2} \\
&\ge1+2T - \big(\EE[\epsilon_{k-1}^4 |\bm{X}_{k-2}]\big)^{1/2} 
\ge 1,
\end{align*}
as long as $T^2\ge  \EE[\epsilon_k^4|\bm{X}_{k-2}]  /4$, $k\in\mathbb{N}$. 
For example, we can take $T\ge \sqrt{\rho}/2$.
We conclude that (A1) holds for the gain (\ref{eq:gain_arch_1}).

To ensure (A2), we evaluate
%it is enough to prove $\EE X_k^4 \le C$ for all $k \in \mathbb{N}_0$ 
%since (A2) would then follow immediately:
\begin{align*}
\EE |G_k - g_k|^2 &= 
\EE \Big[\Big(\frac{\min(X_{k-1}^2, T)}{X_{k-1}^2}\Big)^2 \big(X_k^2-1 -\theta_k X_{k-1}^2\big)^2 \Big] \\
& \le
3\EE \Big[\Big(\frac{\min(X_{k-1}^2, T)}{X_{k-1}^2}\Big)^2 \big(
(2+2\theta_k^2X_{k-1}^4 )\epsilon_k^4 +1+\theta_k^2 X_{k-1}^4\big) \Big] \\
&\le  9 +  3 \EE \Big[\Big(\frac{\min(X_{k-1}^2, T)}{X_{k-1}^2}\Big)^2 
(2C_\Theta \rho +C_\Theta) X_{k-1}^4\Big] \\
& \le 9+3C_\Theta T^2(2\rho+1).
%= C, \quad k \in \mathbb{N}_0.
\end{align*}
%Let us show therefore that $\EE X_k^4 \le C$ for all $k \in \mathbb{N}_0$.
%Note first that $\EE [X_k^2|\mathcal{F}_{k-1}^2]=1+\theta_k X_{k-1}^2$
%so that
%\[
%\EE X_k^2 \le 1+\rho \EE X_{k-1}^2 \le \sum_{i=0}^k \rho^i \le \frac{1}{1-\rho}, \quad k \in \mathbb{N}_0.
%\]
%Similarly, we obtain that $\EE[X_k^4|\mathcal{F}_{k-1}] = (1 + 2\theta_k X_{k-1}^2 + \theta_k^2 X_{k-1}^4)\EE\epsilon_k^4$, 
%and then, since $\rho^2\EE[\epsilon_k^4|\mathcal{F}_{k-1}] \le q < 1$,
%\[
%\EE X_k^4 \le c + q \EE X_{k-1}^4 \le c \sum_{i=0}^k q^i \le C, \quad k \in \mathbb{N}_0.
%\]

\subsection{Tracking an AR(d) parameter}
\label{sec:examples:ARd}

In this section we use the notation $X_{k,d}=\big(X_{k}, X_{k-1}, \dots, X_{k-(d-1)}\big)$ 
for the vector of the $d$ consecutive observations ending with $X_k$.

Consider an autoregressive\index{Autoregressive!model} model with $d$ time varying  parameters:
\begin{align}
\label{eq:ar_p_model}
X_k& = \sum_{i=1}^d\theta_{k,i} X_{k-i} + \xi_k = 
\theta_k^T X_{k-1,d} + \xi_k, \quad k\in\mathbb{N},
\end{align}
where $\theta_k=(\theta_{k,1},\dots, \theta_{k,d})$ is  $\mathcal{F}_{k-d-1}$-measurable, 
$\{\xi_k\}_{k\in\mathbb{N}}$ is a martingale difference noise 
with respect to the filtration $\{\mathcal{F}_k\}_{k \in\mathbb{N}_0}$ such that 
$\mathbb{E}|\xi_k|^2 \le C$, $k\in\mathbb{N}$,  starting random vector 
$X_{0,d}$ is given and such that $\mathbb{E} \|X_{0,d}\|^2 \le c$, 
for some $C,c>0$.

%=(X_0, X_1,\dots, X_{d-1})$ have $p$ bounded moments. %(cf. the end of this section).
%We would like to track the vector $\theta_k=(\theta_{k,1},\dots, \theta_{k,d})$
%which may be random but must be measurable with respect to the $\sigma$-algebra generated by $\bm{X}_{k-2d-1}$.

%In analogy with the non-drifting AR($d$) model, 
For $\vartheta=(\vartheta_1, \ldots, \vartheta_d)$, associate with the AR(d) model its %autoregressive 
polynomial   
\begin{equation}\label{eq:ar_p_polynomial}
t(z, \vartheta) =  1 - \sum_{i=1}^d \vartheta_i z^i, \quad z \in \mathbb{C}.
\end{equation}
It is well know that an AR(d) model with autoregressive parameters $\vartheta$ is stationary 
if, and only if, the (complex) zeros of the polynomial $t(z,\vartheta)$ are outside the unit circle.
%We further require them all to be distinct.
This motivates  the definition of the parameter 
sets $\Theta(\rho)$ for some $0<\rho<1$: %which we define as the closure of
\begin{equation}\label{eq:ar_p_stable_set}
\Theta(\rho)=\big\{\vartheta\in\mathbb{R}^d : \text{for all } |z|<\rho^{-1}, t(z, \vartheta)\neq 0 \big\},
\end{equation}
cf.\ \citep{Moulines:2005}
%Moulines et al.\ (2005)
who also showed that  the following embeddings hold:
\[
B_\infty\big( (\rho^{-2}+\dots+\rho^{-2d})^{-1/2} \big)\subseteq\Theta(\rho)\subseteq	
B_\infty\big( (1+\rho)^d -1 \big),
\]
where $B_\infty(r) = \{ \vartheta\in\mathbb{R}^d: \max_{1\le i\le d}|\vartheta_i| \le r\}$ 
is a uniform ball around zero in $\mathbb{R}^d$ with radius $r>0$.
This gives some feeling about the size of the parameter set $\Theta(\rho)$ and
%(This holds by Lemma 1 of~\citep{Moulines:2005}; note that the extra requirement that the zeros of 
%$t(z,\theta)$ be distinct does not affect the validity of this lemma.)
implies in particular that the set $\Theta(\rho)$ is non-empty and bounded for all $\rho\in(0,1)$.
%(Note that this implies in passing that AR models with parameters in $\Theta(\rho)$ are causal.)

The AR(d) model \eqref{eq:ar_p_model} can also be described by the following inhomogeneous 
difference equation
\begin{equation}\label{eq:ar_d_system_one_step}
X_{k,d} =\bm{C}(\theta_k)X_{k-1,d} + \bm{I}\bm{e}_1 \xi_k,
\end{equation}
where $\bm{e}_1=(1, 0, \dots, 0)\in\mathbb{R}^d$ and, for any $\vartheta\in\mathbb{R}^d$, 
$\bm{C}(\vartheta)$ is the square matrix of order $d$
\begin{equation}\label{eq:companion_matrix}
\bm{C}(\vartheta) =
\begin{bmatrix}
\vartheta_1	&	\vartheta_2	&	\cdots	&\vartheta_{d-1}	&	\vartheta_d\\
1		&	0		&	\cdots	&	0		&	0\\
0		&	1		&	\cdots	&	0		&	0\\
\vdots	&	\vdots	&	\ddots	&	\vdots	&	\vdots\\
0		&	0		&	\cdots	&	1		&	0
\end{bmatrix}.
\end{equation}
This matrix is usually called the \emph{companion matrix}  to the autoregressive polynomial 
$t(z, \vartheta)$; it is also sometimes called the \emph{state transition matrix}.
One can show that the eigenvalues of $\bm{C}(\vartheta)$ are exactly the reciprocals of the zeros of $t(z,\vartheta)$.
This means that the absolute values of the eigenvalues of $\bm{C}(\vartheta)$ for $\vartheta\in\Theta(\rho)$ 
are all at most $\rho<1$. This in turn implies that for any sequence of vectors 
$\theta_d, \theta_{d+1}, \dots\in\Theta(\rho)$, the pair of sequences
\[
\big((\bm{C}(\theta_d),\bm{C}(\theta_{d+1}), \ldots),(\bm{I}, \bm{I},\ldots)\big)
\]
forms a so called \emph{exponentially stable} pair (cf.~\citep{Antsaklis:2007}).
Among other things, this gives us that so long as the $p$-th moments of both the initial 
$X_{0,d}$ and the noise terms $\xi_k$ are bounded, then the $p$-th moments of 
all $X_k$, $k\in \mathbb{N}$, will be bounded as well (cf.\ Proposition 10 of \citep{Moulines:2005}).
%Moulines et al.\ (2005)).

In \citep{Moulines:2005} the model (\ref{eq:ar_p_model}) is considered with nonrandom 
but time varying $\theta_k=\theta(k/n)=(\theta_1(k/n), \ldots, 
\theta_d(k/n))$ for some smooth function $\theta(t) \in \mathbb{R}^d$, $t\in [0,1]$.
One has a triangular array of models and the studied asymptotics  is in sampling
 frequency $n\to \infty$; cf.\ Section \ref{sec:variational_setups:lipschitz}. 
%A particular gain function which can be used to track the parameters of 
%an autoregressive model can be found in Moulines et al.\ (2005).
The considered gain function  is an appropriately rescaled version of the gain 
from Remark \ref{gain_autoregression}, namely,
\begin{equation}
\label{gain_moulines}
G_k(\vartheta,\bm{X}_k) = \big(X_k - \vartheta^TX_{k-1,d}\big) \frac{X_{k-1,d}}
{1+\mu \|X_{k-1,d}\|^2},
\end{equation}
for an appropriately chosen $\mu>0$ depending on the observation frequency $n$.
\begin{remark}
One should mind the difference in indexing in our algorithm (\ref{eq:algorithm_main}) and 
algorithm (3) from \citep{Moulines:2005}. This is not an issue since we can  
make the correspondence between the algorithms exact by treating $\hat{\theta}_{k+1}$ as an estimate
of $\theta_k$ rather than of $\theta_{k+1}$, the error can be absorbed into the third term of 
the right hand side of (\ref{eq:boundE}).
\end{remark}
Although assumption (A2) is trivially satisfied, our general Theorem \ref{theo:bound}
cannot be applied for $d\ge 2$ because assumption (A1) does not hold.
Indeed,  
\[
g_k(\hat{\theta}_k, \theta_k| \bm{X}_{k-1})=-\frac{X_{k-1,d}X_{k-1,d}^T}
{1+\mu \|X_{k-1,d}\|^2}(\hat{\theta}_k-\theta_k) =-\bm{M}_k(\hat{\theta}_k-\theta_k),
\]
and the matrix $\bm{M}_k$ is of the form $\alpha xy^T$ for some $\alpha>0$ and 
column vectors $x,y \in \mathbb{R}^d$.
But the matrix $xy^T$ has $d-1$ zero eigenvalues and one eigenvalue $y^Tx$, so that 
always $\Lambda_{(1)}(\bm{M}_k)=0$ and thus (\ref{eq:Lambda}) does not hold.
\begin{remark}
\label{rem_mouilines}
On the other hand, the authors of \citep{Moulines:2005} do manage to establish a convergence results 
for the gain function (\ref{gain_moulines}). It is instructive to understand where the difference in the two approaches is.
Careful inspection of the proofs in \citep{Moulines:2005} reveals that the analogue 
of the lower bound (\ref{eq:Lambda}), the persistence of excitation condition, is established in
Lemma 17  (p.\ 2627 of \citep{Moulines:2005}).  
The basic difference is that the quantity to bound from below in our case is
the conditional expectation of the smallest eigenvalue of 
the matrix $\bm{M}_k$, whereas in \citep{Moulines:2005}) it is the smallest eigenvalue 
of the conditional expectation of the  matrix $\bm{M}_k$. 
For the gain function (\ref{gain_moulines}),
the lower bound for the former is zero (as is demonstrated above)
and it is positive for the latter (cf.\ Lemma 17 of \citep{Moulines:2005}). 
A way to fix this would be to establish a version of the general theorem,
where (\ref{different_bound}) is assumed instead of (\ref{eq:Lambda}),
see also  Remark \ref{rem_version_theorem}. We do not consider this here.
%This is outlined in Remark \ref{rem_version_theorem}.
\end{remark}

Consider the case $d=1$ and gain (\ref{gain_moulines}) for which Theorem \ref{theo:bound} can be applied.
%Then the  definition of the parameter set $\Theta(\rho)$ implies that 
%$|\theta_k| \le \rho <1$, $k\in \mathbb{N}_0$, almost surely.
Assume that  $\mathbb{E}X_0^2$ is bounded, 
$\mathbb{E}(\xi_k|\mathcal{F}_{k-1})=0$
 and $\mathbb{E}(\xi_k^2|\mathcal{F}_{k-1})=\sigma^2>0$, $k\in\mathbb{N}$. 
%and $\mathbb{E} \xi^4_k=c\sigma^4$,  
%for all $k\in\mathbb{N}$ and some constant $0\le c<5$.
%For example, we can take $c=3$ if the noise is Gaussian.
We have 
\begin{align*}
G_k(\hat{\theta}_k,\bm{X_k})&=\frac{X_{k-1} (X_k-\hat{\theta}_k X_{k-1})}{1+\mu X_{k-1}^2}, \quad
g_k(\hat{\theta}_k, \theta_k | \bm{X}_{k-1})  %=- \frac{X_{k-1}^2 (\hat{\theta}_k - \theta_k)}{1+\mu X_{k-1}^2}
=-M_k(\hat{\theta}_k - \theta_k),
\end{align*}
where $M_k=\frac{X_{k-1}^2}{1+\mu X_{k-1}^2} \le \frac{1}{\mu}$. Besides, 
 if $|X_{k-1}|^2 \ge c$, then $M_k \ge \frac{1}{c^{-1} +\mu}$, and  if $|X_{k-1}|^2 < c$, then 
\[
\mathbb{E}(M_k|\mathcal{F}_{k-2}) \ge \frac{ \mathbb{E}(|X_{k-1}|^2|\mathcal{F}_{k-2})}{1+\mu c}
=\frac{X_{k-2}^2\theta_{k-2}^2 + \sigma^2}{1+\mu c}
\ge
\frac{\sigma^2}{1+\mu c}.
\] 
Condition (A1) is fulfilled. As to condition (A2), 
\begin{align*}
\mathbb{E} |G_k - g_k |^2 &=
%\mathbb{E}\Big[\frac{X_{k-1}X_k -X_{k-1}^2\theta_k}{1+\mu X_{k-1}^2}\Big]^2 = 
\mathbb{E}\Big[\frac{X_{k-1}^2 \xi_k^2}{(1+\mu X_{k-1}^2)^2}\Big]
\le  \sigma^2 \max_{u>0} \Big\{\frac{u}{(1+\mu u)^2}\Big\} = \frac{\sigma^2}{4\mu}.
\end{align*}
Both (A1) and (A2) are thus satisfied. Interestingly, there is no issue of stability in this case:
we do not have to assume that $|\theta_k| \le \rho <1$, $k\in \mathbb{N}_0$, almost surely.
Just almost sure boundedness $|\theta_k|^2 \le C_\Theta$, $k\in \mathbb{N}_0$, 
for a constant $C_\Theta$, is sufficient.

Consider now another gain function for the case $d=1$. This time we assume that   
$\mathbb{E} [\xi_k|\bm{X}_{k-1}]=\mathbb{E} [\xi_k^3|\bm{X}_{k-1}]=0$, 
$\mathbb{E}[\xi_k^2|\bm{X}_{k-1}]=\sigma^2>0$, and $\mathbb{E}[\xi_k^4|\bm{X}_{k-1}]
=c \sigma^4$, $k\in\mathbb{N}$, for some constant $0<c<5$. The proposed gain and the corresponding average 
gain are as follows:
\begin{align}
G_k(\vartheta, \bm{X}_k)	&=
\frac{\min\{X_{k-1}^2,T\}}{X_{k-1}^2}
(X_k X_{k-1}- \vartheta X_{k-1}^2),\notag\\
\label{eq:gain_ar_1b}
g_k(\hat{\theta}_k, \theta_k |\bm{X}_{k-1})&=
- \min\{X_{k-1}^2,T\} (\hat{\theta}_k - \theta_k) = -M_k(\hat{\theta}_k - \theta_k),
\end{align}
with some $T \ge  (9-c)\sigma^2/4$.
Note that this is a rescaled gain function of type $\bar{G}$ from Section~\ref{sec:gains}.
Clearly, $M_k \le T$ and, according to Lemma~\ref{lemma:truncated_conditional}, 
\begin{equation}
\label{M_bound_1}
\mathbb{E}[M_k|\bm{X}_{k-2}] =
\mathbb{E}\big[\min\{X_{k-1}^2,T\}|\bm{X}_{k-2}\big] \ge \frac{(5-c)\sigma^2}{4},
\end{equation}
so that (A1) holds.
Assumption (A2) also holds since
\begin{align*}
\mathbb{E}\big|G_k(\hat{\theta}_k, \bm{X}_k) - g_k(\hat{\theta}_k, \theta_k | \bm{X}_{k-1})\big|^2
&=
\mathbb{E}\Big[\frac{\min\{X_{k-1}^2,T\}^2 \xi^2_k}{X_{k-1}^2}\Big] 
\le \max\{T^2,1\} \sigma^2.
\end{align*}

%% call the parameter space $\Theta_d$; is it the stability region???

%We assume that $\theta_k\in\Theta$, the stability region for this process, i.e., 
%the set of all vectors $\theta=(\theta_1, \theta_2, \dots, \theta_d)$ such that the the polynomial in $z$
%\[
%1 - {\sum}_{i=1}^d \theta_i z^i,
%\]
%has roots outside the unit circle (cf.~\citep{Fan:2003}.)

Finally consider a version of general AR(d) model. We will only outline the main steps, leaving out the details.
Assume that the noise terms $\xi_k$ in (\ref{eq:ar_p_model}) form a Gaussian white noise 
sequence with mean zero and variance $\sigma^2>0$ and that the  parameter process 
$\{\theta_i\}_{ i\in\mathbb{N}}$ is constant within the batch of $d$ consecutive observations.
For a $(2d-1)$-dimensional vector $m= (m_{-(d-1)}, \dots, m_{-1}, m_0, m_1, \dots, m_{d-1})$, 
introduce the Toeplitz matrix $\bm{T}(m) =(m_{ij})$ associated with that vector whose entries are  
$m_{ij}=m_{i-j}$, $i,j=1,\ldots d$, 
so that this matrix has constant (from left to right) diagonals.
Thus, $m$ is the column vector formed by starting at the top right element of $\bm{T}(m)$, 
going backwards along the top row of $\bm{T}(m)$ and then down the left column of $\bm{T}(m)$.
Denote $\vartheta=(\vartheta_1,\ldots,\vartheta_d)$ and introduce 
$\bm{A}(\vartheta)=\bm{T}(a(\vartheta))$ and $\bm{B}(\vartheta)=\bm{T}(b(\vartheta))$,
the Toeplitz matrices created from the vectors 
$a(\vartheta) = (-\vartheta_{d-1}, \dots, -\vartheta_1, 1, 0, \dots, 0)$ and 
$b(\vartheta)=(0, \dots, 0, \vartheta_d, \vartheta_{d-1}, \dots, \vartheta_1)$ respectively.
Under the imposed assumptions, we can rewrite the model 
(\ref{eq:ar_p_model}) as follows:
\begin{equation}
\label{eq:ar_d_system}
\bm{A}(\theta_k)X_{k,d} = \bm{B}(\theta_k)X_{k-d,d}+\xi_{k,d}.
\end{equation}
The matrix $\bm{A}(\theta_k)$ is upper triangular with a diagonal consisting of ones, whence invertible.
From this point on, we regard vector $X_{dk,d}$, $k\in\mathbb{N}$, as an observation 
at time moment $k$  so that we can specify our observation model in terms of conditional 
distribution of $X_{k,d}$ given $\bm{X}_{k-d}$:
\begin{equation}
\label{eq:ar_d_kernel}
X_{k,d}|\bm{X}_{k-d} \sim N\big(\bm{A}^{-1}(\theta_k)\bm{B}(\theta_k)X_{k-d,d},
\sigma^2\bm{A}^{-1}(\theta_k)(\bm{A}^{-1}(\theta_k))^T\big),
\end{equation}
where $\{\theta_k\}_{k\in\mathbb{N}}$ is a predictable process 
with respect to the filtration $\{\mathcal{F}_{kd}\}_{ k\in\mathbb{N}_0}$.
Notice that the observation process is of a Markov structure.
\begin{remark}
Even if the normality of the noise is assumed in the model (\ref{eq:ar_p_model}),
the models  (\ref{eq:ar_p_model}) and (\ref{eq:ar_d_kernel}) still differ
since in general the parameter process $\{\theta_k\}_{k\in\mathbb{N}}$ varies also 
within the batches of $d$ observations in the model (\ref{eq:ar_p_model}). 
However, this is not an issue. Indeed, even though 
the parameter is allowed to vary within each batch of $d$ observation, we still can 
use the gain function (which we derive below) as if the parameter process is constant within 
the batches and establish an upper bound of type (\ref{theo:bound}) 
for the quality of such a procedure.
The error that is made by pretending that the parameter is constant within the batches can 
be absorbed into the third term of the right hand side of (\ref{theo:bound}).
\end{remark}

In this case we propose a gain of the type (\ref{eq:gain_canonical_1}):
\begin{equation}
\label{eq:gain_ar_d}
G_{dk}(\vartheta,\bm{X}_{dk})=\nabla_{\vartheta}\log p_{\vartheta}\big(X_{k,d}|\bm{X}_{d(k-1)}\big),
\end{equation}
where $p_{\vartheta}(\cdot|\bm{X}_{d(k-1)})=p_{\vartheta}(\cdot|X_{d(k-1),d})$ is the conditional density of (\ref{eq:ar_d_kernel}).
%At the end of this section we explicitly compute \eqref{eq:gain_ar_d}; see~\eqref{eq:ar_p_explicit_gain}.
%For us, each data point will be a vector $\bm{X}_{dk,d}$, $k\in\mathbb{N}$ such that 
Thus, the tracking sequence is updated with batches of $d$ observations from the autoregressive process.
Below, to ease the notation, we will often write $X$ and $Y$ instead of $X_{dk,d}$ 
and $X_{d(k-1),d}$, respectively.
%This is necessary to make sure the representation (\ref{eq:ar_d_kernel}) is valid even if the 
%parameter $\theta$ is allowed to change among different batches of observations;
%otherwise the system (\ref{eq:ar_d_system}) would be under-determined.
%We must now establish that this gain function verifies (A1).
As explained in Section~\ref{sec:gains}, the corresponding average gain $g_{dk}$ can be found as minus the gradient 
of the Kullback-Leibler divergence between the two conditional distributions  %transition kernel 
with two different parameters.
This observation is particularly useful if we are able to write this Kullback-Leibler divergence as an appropriate quadratic form.
The Kullback-Leibler divergence between two $d$-dimensional multivariate normal distributions 
$\mathbb{P}_0=N(\mu_0, \bm{\Sigma}_0)$ and $\mathbb{P}_1=N(\mu_1,\bm{\Sigma}_1)$ is given by 
\begin{equation}
\label{eq:KL_normals}
K(\mathbb{P}_0, \mathbb{P}_1) = 
\frac{1}{2}\Big(\log \frac{\det(\bm{\Sigma}_1)}{\det(\bm{\Sigma}_0)} + \tr(\bm{\Sigma}_1^{-1}\bm{\Sigma}_0) - d 
+(\mu_1-\mu_0)^T\bm{\Sigma}_1^{-1}(\mu_1-\mu_0) \Big).
\end{equation}

Let $\theta, \vartheta \in \mathbb{R}^d$, i.e., 
$\theta=(\theta_1,\ldots, \theta_d)$ (not to be confused with the
vectors $\theta_k$, $k\in\mathbb{N}_0$) and $\vartheta=(\vartheta_1,\ldots, \vartheta_d)$. 
According to (\ref{eq:ar_d_kernel}), %write, for $\bm{y}\in\mathbb{R}^d$, 
$\mu(\theta,Y)=\bm{A}^{-1}(\theta)\bm{B}(\theta)Y$ and $\bm{\Sigma}(\theta) 
=\sigma^2\bm{A}^{-1}(\theta)\bm{A}^{-T}(\theta)$.
Now we compute $K\big(N(\mu(\theta,Y),\bm{\Sigma}(\theta)), N(\mu(\vartheta,Y),\bm{\Sigma}(\vartheta))\big)$.
Let $\bm{S}=\bm{T}(s)$ be the Toeplitz matrix associated with the vector 
$s=(0, \dots, 0,1,0, \dots, 0)\in\mathbb{R}^{2d-1}$ where $1$  is in the $(d-1)$-th position. 
Matrix $\bm{S}$ has ones above the main diagonal and zeros elsewhere  and it is
sometimes called \emph{upper shift matrix}.
For $i=2, \dots,d-1$,  the powers $\bm{S}^i$ are the Toeplitz matrices associated with the vectors 
$(0, \dots, 0,1,0, \dots, 0)\in\mathbb{R}^{2d-1}$ where $1$ occupies the $(d-i)$-th position,
$\bm{S}^d=\bm{O}$, the zero matrix of order $d$, and $\bm{S}^0$ should be read as $\bm{I}$, 
the identity matrix of order $d$. 
It follows that $\bm{A}(\vartheta) = \bm{I}-\bm{S}\vartheta_1 -\ldots - \bm{S}^d \vartheta_d$,
so that
\[
\bm{A}(\vartheta) -\bm{A}(\theta) =\bm{S} (\theta_1-\vartheta_1) +\bm{S}^2(\theta_2-\vartheta_2) 
+ \ldots +\bm{S}^d (\theta_d-\vartheta_d),
\]
and
\[
\bm{A}(\vartheta)\bm{A}^{-1}(\theta) =\bm{I} + \bm{S} \bm{A}^{-1}(\theta) (\theta_1-\vartheta_1) 
+ \bm{S}^2 \bm{A}^{-1}(\theta) (\theta_2-\vartheta_2) + \ldots + \bm{S}^d \bm{A}^{-1}(\theta) (\theta_d-\vartheta_d).
\]

For all $\vartheta\in\mathbb{R}^d$, the matrices $ \bm{A}(\vartheta)$ have all eigenvalues 
equal to one (so do their inverses), hence $\det(\bm{\Sigma}(\vartheta))=\sigma^{2d}$ and
we conclude that the logarithm in (\ref{eq:KL_normals}) is zero.
Also, using basic properties properties of the trace and the representation for 
$ \bm{A}(\theta) \bm{A}^{-1}(\vartheta)$ derived above,
\begin{align*}
&\tr\big[\bm{\Sigma}^{-1}(\vartheta) \bm{\Sigma}(\theta)\big]- d 
= \tr\Big[\big(\bm{A}^{-1}(\vartheta)\bm{A}^{-T}(\vartheta)\big)^{-1}\big(\bm{A}^{-1}(\theta)\bm{A}^{-T}(\theta)\big) \Big]- d\\
&= \tr\big[\bm{A}^T(\vartheta)\bm{A}(\vartheta)\bm{A}^{-1}(\theta)\bm{A}^{-T}(\theta)\big]- d 
= \tr\big[\big(\bm{A}(\vartheta)\bm{A}^{-1}(\theta)\big)^T\bm{A}(\vartheta)\bm{A}^{-1}(\theta)\big]-d\\
&= 2\sum_{i=1}^d \tr\big[\bm{S}^i\bm{A}^{-1}(\theta)\big](\theta_i-\vartheta_i)+ \sum_{i=1}^d\sum_{j=1}^d 
\tr\big[\bm{A}^{-T}(\theta)(\bm{S}^i)^T\bm{S}^j\bm{A}^{-1}(\theta) \big] (\theta_i-\vartheta_i)(\theta_j-\vartheta_j).
\end{align*}
Since the inverse of an upper-triangular matrix is upper-triangular, 
$\tr\big[\bm{S}^i\bm{A}^{-1}(\theta)\big]=0$, for all $i=1,\dots,d$ and $\vartheta \in \mathbb{R}^d$.
For any $(n\times m)$-matrix $\bm{M}$, denote by $\vect(\bm{M})$ the column vector containing the $n m$ 
entries of $\bm{M}$ in any (fixed) order. Let $v_i(\theta)=
\vect\big(\bm{S}^i\bm{A}^{-1}(\theta)\big)$, $i=1, \dots,d$, and note that 
$v_d(\theta)$ is always a zero vector.
Note that $\tr\big[\bm{A}^{-T}(\theta)(\bm{S}^i)^T\bm{S}^j\bm{A}^{-1}(\theta) \big]
=v_i^T(\theta)v_j(\theta)$, $i,j=1, \ldots, d$.
We conclude that the previous display can be written as
\[
\tr\big[\bm{\Sigma}^{-1}(\vartheta) \bm{\Sigma}(\theta)\big]- d =
(\vartheta-\theta)^T\big[v_1(\theta) v_2(\theta) \ldots v_d(\theta)\big]^T
\big[v_1(\theta) v_2(\theta) \ldots v_d(\theta)\big] (\vartheta-\theta),
\]
where the matrices on the right are comprised by columns.

Consider now the quadratic form in the Kullback-Leibler 
divergence (\ref{eq:KL_normals}). 
For any $\theta, \vartheta,Y\in\mathbb{R}^d$, 
\begin{align*}
&\big(\mu(\vartheta,Y)-\mu(\theta,Y)\big)^T \bm{\Sigma}^{-1}(\vartheta) \big(\mu(\vartheta,Y)-\mu(\theta,Y)\big) = \\
& \quad\quad =\sigma^{-2} Y^T \big(\bm{B}(\vartheta)-\bm{A}(\vartheta)\bm{A}^{-1}(\theta)\bm{B}(\theta)\big)^T
\big(\bm{B}(\vartheta)-\bm{A}(\vartheta)\bm{A}^{-1}(\theta)\bm{B}(\theta)\big)Y.
\end{align*}
Since $\bm{B}(\vartheta)=(\bm{S}^{d-1})^T \vartheta_1+\ldots +
\bm{S}^T \vartheta_{d-1}+\bm{I} \vartheta_d$, we have
\[
\bm{B}(\vartheta) = \bm{B}(\theta)+\big(\bm{S}^{d-1}\big)^T(\vartheta_1-\theta_1) + \ldots 
+\bm{S}^T(\vartheta_{d-1}-\theta_{d-1}) + \bm{I}(\vartheta_d-\theta_d),
\]
which, together with the representation for $\bm{A}(\vartheta)\bm{A}^{-1}(\theta)$ derived above, imply
\begin{align*}
&\bm{B}(\vartheta)-\bm{A}(\vartheta)\bm{A}^{-1}(\theta)\bm{B}(\theta)=\bm{C}_1(\theta) (\vartheta_1-\theta_1) 
+\bm{C}_2(\theta) (\vartheta_2-\theta_2) + \ldots + \bm{C}_d(\theta) (\vartheta_d-\theta_d),
\end{align*}
where $\bm{C}_i(\theta) = \big(\bm{S}^{d-i}\big)^T+\bm{S}^i\bm{A}^{-1}(\theta)\bm{B}(\theta)$, 
$i=1,\ldots, d$; notice also that $\bm{C}_d(\theta) =\bm{I}$.  Then
\[
\big(\bm{B}(\vartheta)-\bm{A}(\vartheta)\bm{A}^{-1}(\theta)\bm{B}(\theta)\big)Y= 
\big[\bm{C}_1(\theta)Y \ldots \bm{C}_d(\theta)Y \big](\vartheta-\theta).
\]

Summarizing, we obtained that
\[
K\big(N(\mu(\theta,Y),\bm{\Sigma}(\theta)), N(\mu(\vartheta,Y),\bm{\Sigma}(\vartheta))\big) = 
\frac{1}{2} (\vartheta-\theta)^T \bm{M}(\vartheta-\theta),
\]
with 
\begin{align}
\bm{M}=\bm{M}(\theta,Y)&= \big[v_1(\theta) v_2(\theta) \ldots v_d(\theta)\big]^T
\big[v_1(\theta) v_2(\theta) \ldots v_d(\theta)\big] \notag\\
\label{matrix_M}
&\quad + \sigma^{-2}
\big[\bm{C}_1(\theta)Y \ldots \bm{C}_d(\theta)Y \big]^T
\big[\bm{C}_1(\theta)Y \ldots \bm{C}_d(\theta)Y\big].
\end{align}

According to (\ref{KL-gain}), we can derive the expression for the average gain:
\begin{align*}
g_{dk}\big(\vartheta, \theta|\bm{X}_{d(k-1)}\big)=-\bm{M}(\theta,X_{d(k-1),d}) (\vartheta-\theta).
\end{align*}
where $\bm{M}$ is given by (\ref{matrix_M}).
Note that the matrix  $\bm{M}$ does not depend on $\vartheta$ and is clearly positive semidefinite.
We evaluate now its eigenvalues.
In the representation (\ref{matrix_M}), the first matrix in the sum 
is positive semi-definite but has at least one zero eigenvalue.
It is also clear that the entries of this matrix are polynomials of the coordinates of $\theta$, so that, if 
$\theta \in \Theta$ for a bounded set $\Theta$, then the largest eigenvalue of this matrix is upper bounded, 
uniformly over $\Theta$, by some constant, say, $K_1$.
%we remind that this diameter is at most $(1+\rho)^d-1<2^d-1$.
As to the second (also positive semidefinite) matrix in the sum of matrices  from (\ref{matrix_M}), note that
\begin{align*}
\tr\big(&\begin{bmatrix}\bm{C}_1(\theta)Y& \cdots& \bm{C}_d(\theta)Y\end{bmatrix}^T
\begin{bmatrix}\bm{C}_1(\theta)Y&\cdots& \bm{C}_d(\theta)Y\end{bmatrix} \big)\\
\quad &=Y^T\bm{C}_1^T(\theta)\bm{C}_1(\theta)Y+ \ldots + Y^T\bm{C}_d^T(\theta)\bm{C}_d(\theta)Y.
\end{align*}
The entries of the matrices $\bm{C}_i^T(\theta)\bm{C}_i(\theta)$, $i=1,\dots,d$, 
are polynomials in $\theta_1, \dots, \theta_d$ which are bounded uniformly over 
a bounded set $\Theta$.  Recall also that the trace of a matrix is equal to 
the sum of its eigenvalues. 
We conclude that  $\Lambda_{(d)}\big(\bm{M}(\theta, Y)\big) \le K_1 + K_2 \|Y\|^2$ uniformly in 
$\theta\in\Theta$ for any bounded $\Theta \subset \mathbb{R}^d$.

To derive a lower bound on the smallest eigenvalue of the matrix $\bm{M}(\theta, Y)$,
note that this matrix can be rewritten in the form
\[
\left[
\begin{array}{ccc|c}
v_{1,1}(\theta)	&	\cdots	&	v_{1,d-1}(\theta)	&	0\\
\vdots			&	\ddots	&	\vdots			&	\vdots\\
v_{d-1,1}(\theta)	&	\cdots	&	v_{d-1,d-1}(\theta)	&	0\\
\hline
0			&	\cdots	&	0	& Y^TY
\end{array}
\right]
+
\left[
\begin{array}{ccc|c}
c_{1,1}(\theta)	&	\cdots	&	c_{1,d-1}(\theta)	&	c_{1,d}(\theta)\\
\vdots			&	\ddots	&	\vdots			&	\vdots\\
c_{d-1,1}(\theta)	&	\cdots	&	c_{d-1,d-1}(\theta)	&	c_{d-1,d}(\theta)\\
\hline
c_{d,1}(\theta)	&	\cdots	&	c_{d,d-1}(\theta)
&	0
\end{array}
\right]
\]
for $v_{i,j}(\theta) =v_i^T(\theta)v_j(\theta)$ and 
$c_{i,j}(\theta) = \sigma^{-2}Y^T\bm{C}_i^T(\theta)\bm{C}_j(\theta)Y$, 
where we swapped the $(d,d)$-th entries of the matrices in the sum from (\ref{matrix_M}).
We used also that $v_d(\theta)=0$ and $\bm{C}_d(\theta)=\bm{I}$.

Note that the top left matrices in the block matrices above are Gram matrices and therefore 
positive semidefinite. %; the full block matrices are triangular by blocks.
The matrix $[v_{i,j}(\theta)]_{i,j=1,\dots,d-1}$ is the Gram matrix associated with 
the vectors $v_1(\theta), \dots, v_{d-1}(\theta)$. Since $\bm{A}^{-1}(\theta)$ 
is a triangular matrix with $1$'s in its main diagonal, it follows that these vectors 
are linearly independent. Hence the associated Gramian is actually positive definite for each $\theta$.
The determinant of this Gramian is a polynomial in the entries of the matrix which in turn are polynomials 
in $\theta_1, \dots, \theta_d$. If $\theta$ lies in a compact set $\Theta$,   
the infimum of the determinant of this matrix over $\theta\in\Theta$ must be lower bounded by some positive constant, say, $K_3$.
Using the same reasoning, we conclude that its determinant is upper bounded by some constant $K_4$.
A lower bound on the smallest eigenvalue can then be obtained by noting that for any positive definite matrix $\bm{M}$ 
of order $d$,
\[
\lambda_{(1)}(\bm{M})\ge\frac{\det(\bm{M})}{\lambda_{(d)}^{d-1}(\bm{M})}\ge\frac{K_3}{K_4^{d-1}} =K_5>0.
\]
%for some constant $\nu$ depending only on $d$ and say, the diameter of the parameter set $\Theta(\rho)$.

We conclude that the smallest eigenvalue of the block matrix on the left is at least $\min(K_5,\|Y\|^2)$.
The block matrix on the right is clearly positive semidefinite. We conclude that the smallest eigenvalue of the matrix 
$\Lambda_{(1)}\big(\bm{M}(\theta, Y)\big) \ge\min\{K_5, \|Y\|^2\}$ 
by using Weyl's Monotonicity Theorem, see for example \citep{Bathia:1997}.
This means that $\Lambda_{(1)}\big(\bm{M}(\theta, Y)\big) \ge \|Y\|^2$ for all $\|Y\| \le K_5$. 
Swapping back the $(d,d)$-th entries of the matrices in the sum from (\ref{matrix_M}),
we see that there must exist an $Y$ such that $\|Y\| \le K_5$ and for which 
the smallest eigenvalue of the second matrix in the sum from (\ref{matrix_M}) is bounded from below by 
$c_1\|Y\|^2$ for some $c_1>0$.
By using renormalization arguments, we conclude that 
$\Lambda_{(1)}\big(\bm{M}(\theta, Y)\big) \ge c_2\|Y\|^2$ for some $c_2>0$.

Condition (A2) is not difficult to check.
The gain (\ref{eq:gain_ar_d}) can be written in the following form
\begin{align*}
G_{dk}(\vartheta,\bm{X}_{dk})
%&=-\frac{\sigma^{-2}}{2}\nabla_{\theta}(\bm{A}(\theta)\bm{X}_{dk,d}-\bm{B}(\theta)\bm{X}_{d(k-1),d})^T 
%(\bm{A}(\theta)\bm{X}_{dk,d}-\bm{B}(\theta)\bm{X}_{d(k-1),d})\\
&=-\sigma^{-2}(\bm{A}(\theta)X_{dk,d}-\bm{B}(\theta)X_{d(k-1),d})^T
\frac{\bm{\partial}\big(\bm{A}(\theta)X_{dk,d}-\bm{B}(\theta)X_{d(k-1),d}\big)}{\bm{\partial}\theta},
%&=\sigma^{-2}\big( A(\theta)\bm{x} - B(\theta)\bm{y} \big)^T D(\bm{x},\bm{y}) J,
\end{align*}
where $\bm{\partial}/\bm{\partial}\theta$ represents the Jacobian\index{Jacobian matrix} operator.
To verify (A2), it suffices to check that the expectation of the norm of $G_{dk}$ is bounded.
We omit the details but it is clear from the expression derived above that the norm of the gain 
function squared is a polynomial of degree four in the coordinates of $X_{dk,2d}$.
We have already mentioned that if the initial values for the autoregressive process and 
the noise terms have uniformly bounded second moments, then this transfers to the each 
observation $X_k$, provided that the sequence of parameters of the model, $\theta_k$, 
lives in the parameter set $\Theta(\rho)$ for some $\rho<1$.

For the matrix  $\bm{M}$ defined by (\ref{matrix_M}), 
we established above that almost surely
\[
c\|Y\|^2 \le \Lambda_{(1)}\big(\bm{M}(\theta, Y)\big) \le \Lambda_{(d)}\big(\bm{M}(\theta, Y)\big)  \le 
K_1 + K_2 \|Y\|^2
\]
uniformly in $\theta\in\Theta$ for any bounded $\Theta \subset \mathbb{R}^d$.
%the eigenvalues of  appearing in the conditional gain vector 
%$g_{dk}$ are upper and lower bounded by multiples of $\|\bm{X}_{d(k-1),d}\|^2$.
We can get rid of the dependence on $\|Y\|^2$ (recall that $Y=X_{d(k-1),d}$) 
by using the rescaled gain $\bar{G}_{dk}$ 
defined in Section~\ref{sec:gains}:
\[
\bar{G}_{dk}(\vartheta,\bm{X}_{dk}) = \frac{\min\{ \kappa, \|X_{d(k-1),d}\|^2\}}{\|X_{d(k-1),d}\|^2} 
G_{dk}(\vartheta,\bm{X}_{dk}), \quad \kappa>0.
\]
Then 
$\bar{g}_{dk}\big(\vartheta, \theta|\bm{X}_{d(k-1)}\big)=
- \bar{\bm{M}}(\theta,X_{d(k-1),d}) (\vartheta-\theta)$, 
where 
\[
\bar{\bm{M}}(\theta,X_{d(k-1),d})=\frac{\min\{ \kappa, \|X_{d(k-1),d}\|^2\}}{\|X_{d(k-1),d}\|^2}\bm{M}(\theta,X_{d(k-1),d})
\]
and $\bm{M}$ is defined by (\ref{matrix_M}).
The argument in \eqref{eq:scaled_gain} shows that (A2) still holds for this rescaled gain.
Next, $\Lambda_{(d)}\big(\bar{\bm{M}}(\theta,X_{d(k-1),d})\big) \le C$ almost surely by construction.
%The largest eigenvalue of the matrix in $\bar{g}_{dk}$ corresponding to $\bar{G}_{dk}$ 
%is going to be almost surely bounded by construction.
Thus to establish (A1), we need to verify that 
$\mathbb{E}\big[\Lambda_{(1)}\big(\bar{\bm{M}}(\theta,X_{d(k-1),d})\big) \big| \bm{X}_{d(k-2)}\big] \ge c$.
%We need then to verify that the smallest eigenvalue of the matrix in $\bar{g}_{dk}$ 
%has conditional expectation bounded away from zero so that (A1) holds.
One can then proceed as in (\ref{M_bound_1}) (and  Lemma~\ref{lemma:truncated_conditional}) to show that 
for an appropriately large $\kappa$, $\mathbb{E}\big[\min\{\kappa, \|X_{d(k-1),d}\|^2\}| \bm{X}_{d(k-2)}\big] >c$, 
we omit this derivation. 
\begin{remark}
One could drop the requirement for the errors to be Gaussian and still use the same gain 
$G_{dk}$. We expect the same results to hold, under appropriate moment assumptions.
Instead of using the Kullback-Leibler representation \eqref{eq:KL_normals}, one has to 
work with quantities $G_{dk}$  and $g_{dk}$ directly and assure the validity of (A1) and (A2) 
based on moment assumptions on the error terms (and possibly the initial conditions) 
as we did in the one dimensional case.
\end{remark}

\section{Proofs of the lemmas}
\label{sec:proofs}

\begin{proof}[Proof of Lemma \ref{lemma1}]
First suppose that $y=\bm{M}x$ for some symmetric positive definite matrix 
$\bm{M}$ such that  $0<\lambda_1 \le \lambda_{(1)}(\bm{M}) \le \lambda_{(d)}(\bm{M}) 
\le \lambda_2 <\infty$. 
Then $\left\langle x,y\right\rangle=x^T\bm{M}x$ and therefore 
\[
0<\lambda_1\|x\|^2 \le \lambda_{(1)}(\bm{M})\|x\|^2  
\le \left\langle x,y\right\rangle \le  \lambda_{(d)}(\bm{M})\|x\|^2  \le \lambda_2\|x\|^2
\]
and
\[
\|y\|^2 = \left\langle y,y\right\rangle = x^T\bm{M}^T\bm{M}x = x^T\bm{M}^2x \le \lambda_2^2 \|x\|^2.
\]

Now we prove the converse assertion. 
Suppose $x,y \in \mathbb{R}^d$ and $0<  \lambda_1' \|x\|^2 \le \left\langle x,y\right\rangle \le \lambda_2' \|x\|^2 < \infty$
for some $\lambda_1',\lambda_2' \in \mathbb{R}$ such that $0<\lambda_1' \le \lambda_2' <\infty$ and that $\|y\|\le C\|x\|$.
Let $V=\{v=a x + by: \, a,b \in \mathbb{R}\}$ be the linear space spanned by $x$ and $y$.
First consider the case $\mbox{dim}(V) =1$, i.e., $y=\alpha x$ for some $\alpha \in \mathbb{R}$. 
Then $\left\langle y,x\right\rangle=\alpha \|x\|^2$ so that $0<\lambda_1'\le \alpha \le \lambda_2'<\infty$.
Thus $y= \alpha x  = \bm{M} x$ with symmetric and positive $\bm{M}=\alpha I$ 
so that $0<\lambda_1'\le \alpha=\lambda_{(1)}(\bm{M})=\lambda_{(d)}(\bm{M})\le\lambda_2'<\infty$.

Now consider the case $\mbox{dim}(V) =2$.
%Suppose we constructed a symmetric positive  $M$ such that 
%$y=Mx$ and $0<\lambda_1\le\lambda_{(1)}(M)\le\lambda_{(d)}(M)\le\lambda_2<\infty$.
%Then $(x,y) = x^T Mx$ and 
%since $\lambda_{(1)}(M)\|x\|^2  \le x^TMx \le  \lambda_{(d)}(M)\|x\|^2$.
%Thus we only need to show the existence of the representation 
%$y=Mx$ for some symmetric positive  $M$.
Let $e_1=x/\|x\|$ and $\{e_1,e_2\}$ be an orthonormal basis of $V$.
Then 
\begin{align*}
x&=\|x\| e_1 \\
y &= \alpha e_1 +\beta e_2.
\end{align*}
The conditions $\lambda_1'\|x\|^2 \le \left\langle x,y\right\rangle =\alpha \|x\| \le \lambda_2' \|x\|^2$ 
and $\|y\|=\sqrt{\alpha^2+\beta^2} \le C \|x\|$ imply that
\[
\lambda_1'\|x\|\le \alpha\le \min\{\lambda_2',C\} \|x\|, \quad |\beta| \le C \|x\|.
\]

Let $e_2$ be chosen in such a way that  $\beta>0$ (which is always possible.)
Now, we change the basis of $V$ as follows:
\begin{align*}
e'_1 &= \cos(\theta) e_1 - \sin(\theta) e_2, \\ 
e'_2 &=\sin(\theta) e_1 + \cos(\theta) e_2. 
\end{align*}
We thus rotate the basis $\{e_1,e_2\}$ by the angle $\theta$.
In these new basis we have 
\begin{align*}
x& = \|x\|\cos(\theta)e'_1 + \|x\|\sin(\theta) e'_2=\alpha_x e'_1+ \beta_x e'_2, \\ 
y& = (\alpha\cos(\theta)-\beta\sin(\theta))e'_1 +(\alpha\sin(\theta)+\beta\cos(\theta))e'_2 =\alpha_y e'_1 + \beta_y e'_2. 
\end{align*}

%such that 
%$x=\alpha_x e_1 +\beta_x e_2$, $y=\alpha_y e_1 +\beta_y e_2$,
%where $\alpha_x$ and   $\alpha_y$ are of the same sign, also 
%$\beta_x$ and $\beta_y$ are of the same sign.

Recall that $\alpha,\beta>0$.
Take $\theta \in (0,\pi/2)$ such that $\alpha\cos(\theta)-\beta\sin(\theta)=\frac{1}{2} \alpha\cos(\theta)$,
i.e.,\ $\tan(\theta)=\frac{\alpha}{2\beta}$.
Then we have that 
\[
\frac{\lambda_1'}{2}\le \frac{\alpha}{2\|x\|}= \frac{\alpha_y}{\alpha_x} %\le \frac{\alpha}{\|x\|_2}
\le\frac{\min\{\lambda_2',C\}}{2}, 
\;\; 
\lambda_1'\le \frac{\alpha}{\|x\|}\le \frac{\beta_y}{\beta_x} \le \frac{\alpha}{\|x\|} + 
\frac{2\beta^2}{\alpha\|x\|}\le \min\{\lambda_2',C\}+\frac{2C^2}{\lambda_1'}.
\]
Take then $\lambda_1=\lambda_1'/2$ and $\lambda_2 = \min\{\lambda_2',C\}+2C^2/\lambda_1'$.

Let $\{e'_3,\ldots, e'_d\}$ be the orthonormal basis of $V^\bot$, so that $\bm{b}=\{e'_1,e'_2,e'_3, \ldots ,e'_d\}$
is an orthonormal basis of $\mathbb{R}^d$. 
Take
\[
\tilde{\bm{M}}=\left[\begin{array}{c|c}
\bm{D} & \bm{O}\\
\hline
\bm{O} & \bm{I}_{d-2}
\end{array}\right]
\quad \hbox{ with }\quad
\bm{D}=\left[\begin{array}{cc}
\alpha_y/\alpha_x & 0\\
0 & \beta_y/\beta_x
\end{array}\right]
\]
where the $\bm{O}$'s indicate null matrices of the appropriate dimensions.
We then have $y=\tilde{\bm{M}} x$ in the basis $\bm{b}$ and 
$\lambda_1 \le \lambda_{(1)}(\tilde{\bm{M}}) \le \lambda_{(d)}(\tilde{\bm{M}}) \le \lambda_2$.
We can finally obtain $\bm{M}$ by using the orthogonal matrix 
$\bm{T}$ to change the basis $b$ to the canonical basis of $\mathbb{R}^d$ as 
$\bm{M}=\bm{T}^{-1}\tilde{\bm{M}}\bm{T}=\bm{T}^T\tilde{\bm{M}}\bm{T}$. 
Clearly, $\bm{M}$ has the same eigenvalues 
as $\tilde{\bm{M}}$ and is symmetric.
\end{proof}

\begin{proof}[Proof of Lemma \ref{lemma_bound}]
For the sake of  brevity, we use the notations $\theta_k=\theta_k(\bm{X}_{k-1})$, 
$G_k=G(\hat{\theta}_k, X_k|\bm{X}_{k-1})$ and 
$g_k = g(\hat{\theta}_k,\theta_k|\bm{X}_{k-1})$, $k\in\mathbb{N}_0$.
%$\mathcal{F}_k = \sigma(\bm{X}_k)$ is the $\sigma$-field generated by $\bm{X}_k=(X_0, X_2, \dots, X_k)$.

Recall that $\Theta$ is compact and $\sup_{\theta \in \Theta} \|\theta\|_2^2 \le C_\Theta$.
By iterating (\ref{eq:algorithm_main}), %and using the condition that $\EE\|G_k\|^2 < \infty$,
it is easy to see that $\EE \|\hat{\theta}_k\|^2 <\infty$ for each $k \in \mathbb{N}_0$.
First assume $\EE \|\hat{\theta}_k\|^2 \le K C_\Theta$, for some $K>0$ to be chosen later. 
By  (\ref{bound_G_k}), we obtain 
$\EE\|G_k\|^2 \le 2 C_g + 4 \lambda_2^2 C_\Theta  + 4 \lambda_2^2K C_\Theta=\bar{C}_g$, 
which implies, in view of (\ref{eq:algorithm_main})  and the fact that $\gamma_k\le\Gamma$,
\[
\EE  \|\hat{\theta}_{k+1}\|^2 \le 2 \EE \|\hat{\theta}_k\|^2 + 2\gamma_k^2 
\EE  \|G_k\|^2  \le 2 K C_\Theta+ 2 \Gamma^2\bar{C}_g=C_2.  
\]

Next, consider the case $\EE \|\hat{\theta}_k\|^2 > K C_\Theta$, which, together with 
(\ref{C_Theta}), implies that $\EE \|\hat{\theta}_k\|^2 > K\EE \|\theta_k\|^2$.
%(Note that for sufficiently large $K$ we can assume $\EE \|\hat{\theta}_k\|_2^2\ge1$.)
Recall that, in view of (A1), $\bm{M}_k$ is a symmetric positive definite matrix  such that 
$0<\lambda_1 \le \Lambda_{(d)}(\bm{M}_k) \le \lambda_2<\infty$
almost surely. Therefore we obtain that, almost surely, 
$\hat{\theta}_k^T \bm{M}_k\hat{\theta}_k =\Lambda_{(d)}(\bm{M}_k)\|\hat{\theta}_k\|^2 
\ge \lambda_1 \|\hat{\theta}_k\|^2$ and, by the Cauchy-Schwarz inequality,   
\[
\hat{\theta}_k^T\bm{M}_k \theta_k 
\le | \hat{\theta}_k^T\bm{M}_k \theta_k|
\le 
\big(\hat{\theta}_k^T \bm{M}_k \hat{\theta}_k\big)^{1/2} 
\big(\theta_k^T \bm{M}_k \theta_k\big)^{1/2} 
 \le \lambda_2 \|\hat{\theta}_k\| \|\theta_k\|.
\]
By using the last two relation, (\ref{C_Theta}), (\ref{eq:D}), (\ref{eq:g2}) and 
(\ref{eq:algorithm_main}),  %and Lemma \ref{lemma1}, 
we evaluate $\EE\|\hat\theta_{k+1}\|^2$:
\begin{align*}
\EE\|\hat\theta_{k+1}\|^2 \le & \EE \|\hat{\theta}_k\|^2 + 
2 \gamma_k \EE \EE \big[\hat{\theta}_k^T \EE(G_k|\bm{X}_{k-1})\big] + \gamma^2_k \EE  \|G_k\|^2\\
\le &
\EE \|\hat{\theta}_k\|^2 - 2 \gamma_k \EE  \big(\hat{\theta}_k^T \bm{M}_k (\hat{\theta}_k - \theta_k)\big) +
\gamma_k^2 \bar{C}_g\\
\le &
\EE \|\hat{\theta}_k\|^2 - 2 \gamma_k \big[\lambda_1 \EE  \|\hat{\theta}_k\|^2 
- \EE \big(\hat{\theta}_k^T\bm{M}_k \theta_k) \big]+\gamma_k^2 \bar{C}_g(KC_\Theta)^{-1} \EE \|\hat{\theta}_k\|^2\\
\le &
\EE \|\hat{\theta}_k\|^2 - 2 \gamma_k \big[\lambda_1 \EE \|\hat{\theta}_k\|^2 - 
\lambda_2 \EE \big(\|\hat{\theta}_k\| \|\theta_k\|\big)\big] +
\gamma_k^2 C_3\EE \|\hat{\theta}_k\|^2,
\end{align*}
with $C_3 = \bar{C}_g/(KC_\Theta)$.

Since $\EE \|\hat{\theta}_k\|^2 > K \EE \|\theta_k\|^2$, %and the Cauchy-Schwarz inequality, 
it follows that 
$\EE \big(\|\hat{\theta}_k\| \|\theta_k\|\big) \le \big(\EE \|\hat{\theta}_k\|^2\EE \|\theta_k\|^2\big)^{1/2} 
\le \frac{\EE \|\hat{\theta}_k\|^2}{\sqrt{K}}$.
%, it follows that $B (\EE \|\hat{\theta}_k\|_2^2)^{1/2}/\sqrt{K} - B (\EE \|\theta_k\|_2)^{1/2}\ge0$.
%Using this, the previous display  and H\"older's\index{Holder inequality@H\"older inequality}
% inequality $\EE \|\hat{\theta}_k\|_2 \|\theta_k\|_2 \le \big(\EE \|\hat{\theta}_k\|_2^2\EE \|\theta_k\|_2^2\big)^{1/2}$
%and $\EE \|\hat{\theta}_k\|_2^2 > K\EE \|\theta_k\|_2^2$
Using this, we proceed by bounding the previous display as follows:
\begin{align*}
%\le &
%\EE \|\hat\theta_k\|^2 -2\gamma_k \big[\lambda_1 \EE \|\hat\theta_k\|^2- \lambda_2 (\EE \|\theta_k\|^2 
%\EE \|\hat{\theta}_k\|^2)^{1/2}\big] + 2 C\gamma_k^2 \EE \|\hat{\theta}_k\|^2\\
\le &
\EE \|\hat{\theta}_k\|^2 -2\gamma_k  \EE \|\hat{\theta}_k\|^2
\Big(\lambda_1-\frac{\lambda_2}{\sqrt{K}}\Big) + \gamma_k^2 C_3 \EE \|\hat{\theta}_k\|^2\\
= &
%\EE  \|\hat{\theta}_k\|_2^2 - \gamma_k \EE  \|\hat{\theta}_k\|_2^2 
%\Big(2A-\frac{2B}{\sqrt{K}}-\gamma_k  4L\Big) +\gamma_k^2(2C+4 LC_\Theta^2) \frac{\EE \|\hat{\theta}_k\|_2^2}{K C^2_\Theta}\\
%=&
\EE  \|\hat{\theta}_k\|^2 - 2\gamma_k \EE  \|\hat{\theta}_k\|^2 
\Big(\lambda_1-\frac{\lambda_2}{\sqrt{K}}-\frac{C_3\gamma_k}{2}\Big)
\le \EE  \|\hat{\theta}_k\|^2, 
\end{align*}
for sufficiently large $K$ and sufficiently small $\gamma_k$.
Thus,  for sufficiently large $K$ and sufficiently small $\gamma_k$, 
$\EE\|\hat\theta_{k+1}\|^2\le C_2$ with $C_2$ as defined above. 
%can at most be $KC^2_\Theta +C_1$.
%We thus have that for all $k\in \mathbb{N}$, 
%$E \|\hat{\theta}_k\|^2 \le C$. Indeed, we start with 
%$\hat{\theta}_0 = 0$ so that 
%$\|\hat{\theta}_0\| \le K \|\theta_0\|^2$ and therefore
%$E\|\hat{\theta}_1\|^2 \le C$. 
%Then either $E \|\hat{\theta}_1\| \le K \|\theta_1\|^2$ or $E
%\|\hat{\theta}_1\|> K \|\theta_1\|^2$.
%In the both cases, we will have that 
%$E \|\hat{\theta}_2\|^2 \le C$. Iterating this reasoning,
%we derive the statement of the lemma.
\end{proof}

\begin{lemma}
\label{lemma:eig}
Let  $\bm{M}$ be a symmetric positive definite matrix of order $d$, $p\ge1$ and a constant 
$\gamma>0$ be such that $\gamma\Lambda_{(d)}(\bm{M})<1$. 
%with (increasing) eigenvalues $\lambda_{(i)}(M)$, 
%the smallest and largest eigenvalues of which we denote as 
%$\lambda_{(1)}(M)$ and $\lambda_{(d)}(M)$ respectively.
Then $\|\bm{I}-\gamma \bm{M}\| = 1-\gamma \Lambda_{(1)}(\bm{M})$ and
\[
0<1-\gamma \Lambda_{(d)}(\bm{M}) = \Lambda_{(1)}(\bm{I}-\gamma \bm{M}) 
\le \Lambda_{(d)}(\bm{I} - \gamma \bm{M})=1-\gamma \Lambda_{(1)}(\bm{M})<1.
\]
Besides,  $\|\bm{M}\|_p \le K_p(d) \|\bm{M}\|=K_p(d) \Lambda_{(d)}(\bm{M})$ 
for some constant $K_p(d)>0$.
%where %for $p\in\mathbb{N}$, 
%$\| M \|_p$ is the operator norm\index{Operator norm} induced by $l_p$.
\end{lemma}
\begin{proof}
Let $\lambda_i$'s be the eigenvalues of $\bm{M}$, so that the matrix $\bm{I}-\gamma \bm{M}$ 
has eigenvalues $1-\gamma \lambda_i$, $i=1, \dots, d$. 
Since $\gamma\Lambda_{(d)}(\bm{M})<1$, then, for all $i=1, \dots, d$, 
$0<\gamma\Lambda_{(1)}(\bm{M}) \leq \gamma\lambda_i \leq \gamma\Lambda_{(d)}(\bm{M})<1$, implying 
$1> 1-\gamma\Lambda_{(1)}(\bm{M})\geq 1-\gamma\lambda_i\geq1-\gamma\Lambda_{(d)}(\bm{M})>0$, so that
$\|\bm{I}-\gamma \bm{M}\| =\max_ i|1-\gamma\lambda_i| = 1-\gamma\Lambda_{(1)}(\bm{M})<1$.
The first two assertions follow.

It remains to prove the last assertion.
For $x\in\mathbb{R}^d$, let $R_2^p(d)= \max_{x\neq 0}\|x\|_p/\|x\|_2$ and
$R_p^2(d)= \max_{x\neq 0}\|x\|_2/\|x\|_p$.
According to Theorem 5.6.18 from \citep{Horn&Johnson:1991}, %(cf. \citet[Theorem 5.6.18]{Horn&Johnson:1991})
\[
\max_{M\neq O}\frac{\|\bm{M}\|_p}{\|\bm{M}\|_2} = R_2^p(d)R_p^2(d)=K_p(d).
\]
Recall that $\|\bm{M}\|_2=\|\bm{M}\|=\lambda_{(d)}(\bm{M})$ and 
$\|x\|_s \le \|x\|_r \le d^{1/r-1/s}\|x\|_s$ for any $x\in \mathbb{R}^d$ and $s\ge r\ge 1$.
From the last relation it is easy to get the following bounds:
$R_2^p(d) \le 1$ if $p\ge 2$, $R_2^p(d) \le d^{(2-p)/(2p)}$ if $1\le p<2$;
$R_p^2(d) \le d^{(p-2)/(2p)}$ if $p\ge 2$, $R_p^2(d) \le 1$ if $1\le p<2$. 
These bounds imply that $K_p(d) \le d^{(p-2)/(2p)}$ if $p\ge 2$ and 
$K_p(d) \le d^{(2-p)/(2p)} \le d^{1/2}$ if $1\le p<2$.
%We then have (cf.\ \citet[Section 5.6.6]{Horn&Johnson:1991}) that for $M$ a real, symmetrical, positive definite matrix, 
%where $\lambda_{(i)}(M)$ is the $i$-th largest eigenvalue of a matrix $M$,
%$\| M \|_2= \max_i\sqrt{ \lambda_i(M^TM)}=\max_i\sqrt{ \lambda_{(i)}(M^2)} =
%\lambda_{(d)}(M)$.
This completes the proof of the lemma.
%Note that by application of the H\"older inequality, we have $\|x\|_p\le d^{(q-p)/(qp)}\|x\|_q$ for $p\le q$ 
%and so we can take $K_p=d^{(p-2)/(2p)}$ if $p\ge2$ and $K_p=d^{(2-p)/(2p)}$ if $1\le p<2$.
\end{proof}

\begin{lemma}[Abel tranformation]
\label{lemma:abel}
Suppose $d_1,d, k_0, k \in \mathbb{N}$ and $k_0 \le k$. Let $\bm{B}_i$ be 
$(d_1\times d)$-matrices, $a_i \in \mathbb{R}^d$ and $A_i=\sum_{j=k_0}^i a_j$, $i=k_0,\ldots, k$.
Then
\[
\sum_{i=k_0}^k \bm{B}_i a_i = \sum_{i=k_0}^{k-1}(\bm{B}_i-\bm{B}_{i+1})A_i + \bm{B}_k A_k.
\]
\end{lemma}
\begin{proof}
We prove this by induction in $k$.
For $k=k_0$ we simply have $\bm{B}_{k_0} a_{k_0} = \bm{B}_{k_0} A_{k_0} = \bm{B}_{k_0} a_{k_0}$ 
and the assertion holds true.
%For $N=1$ we have $B_0a_0 + B_1a_1 = (B_0-B_1)a_0+(a_0+a_1)B_1 = B_0a_0+B_1a_1$.
Assume that the equality holds for $k=n$ and let us prove the result for $k=n+1$.
We have
\begin{align*}
\sum_{i=k_0}^{n+1} \bm{B}_i a_i&
=\sum_{i=k_0}^n \bm{B}_i a_i + \bm{B}_{n+1}a_{n+1} 
= \sum_{i=k_0}^{n-1}(\bm{B}_i-\bm{B}_{i+1})A_i + \bm{B}_nA_n + \bm{B}_{n+1}a_{n+1} \\
&=\sum_{i=k_0}^{n}(\bm{B}_i-\bm{B}_{i+1}) A_i - (\bm{B}_n-\bm{B}_{n+1})A_n 
+ \bm{B}_n A_n + \bm{B}_{n+1}a_{n+1} \\
&=\sum_{i=k_0}^{n}(\bm{B}_i-\bm{B}_{i+1})A_i +\bm{B}_{n+1}A_{n+1}.
\end{align*}
\end{proof}

\begin{lemma}
\label{lemma:truncated_conditional}
Consider an AR(1)-model with a measurable %random, drifting parameter 
$\theta_k=\theta_k(\bm{X}_{k-1})$:
\[
X_k = X_{k-1} \theta_k + \xi_k, \quad k\in\mathbb{N},
\]
where %$X_0$ is such that $\EE X_0^4 <\infty$, the noise $\xi_k$ is independent of $X_{k-1}$, 
$\mathbb{E} [\xi_k|\bm{X}_{k-1}]=\mathbb{E} [\xi_k^3|\bm{X}_{k-1}]=0$, $\mathbb{E}[\xi_k^2|\bm{X}_{k-1}]=\sigma^2>0$, 
and $\mathbb{E}[\xi_k^4|\bm{X}_{k-1}]=c \sigma^4$, $k\in\mathbb{N}$, for some constant $0<c<5$.
%Let also $X_0$ be such that $\EE X_0^2$ and $\EE X_0^4$ are bounded.
%We assume that the drifting parameter $\theta_k$ is measurable with respect to $\sigma(\bm{X}_{k-1})$, 
%and verifies $|\theta_k|\le q<1$, almost surely, for every $k\in\mathbb{N}$.
Then, for any $T$ such that $T\ge(9-c)\sigma^2/4$,
\[
\mathbb{E}\big[\min(X_k^2,T)|\bm{X}_{k-1}\big] \ge \frac{(5-c)\sigma^2}{4},\quad k \in \mathbb{N}.
\]
\end{lemma}
\begin{proof}
We  compute  
\begin{align*}
\mathbb{E}[X_k^2|\bm{X}_{k-1}]& = X_{k-1}^2 \theta_k^2 + 2X_{k-1}\theta_k 
\mathbb{E}[\xi_k|\bm{X}_{k-1}] + \mathbb{E}[\xi_k^2|\bm{X}_{k-1}] = X_{k-1}^2 \theta_k^2+\sigma^2,\\
\mathbb{E}[X_k^4 |\bm{X}_{k-1}]	
& = X_{k-1}^4\theta_k^4 - 4X_{k-1}^3\theta_k^3 \mathbb{E}[\xi_k|\bm{X}_{k-1}] 
+ 6 X_{k-1}^2\theta_k^2 \mathbb{E}[\xi_k^2|\bm{X}_{k-1}] \\
& \;\;\; - 
4X_{k-1}\theta_k \mathbb{E}[\xi_k^3|\bm{X}_{k-1}]+\mathbb{E}[\xi_k^4|\bm{X}_{k-1}]
 = X_{k-1}^4\theta_k^4+6X_{k-1}^2\theta_k^2\sigma^2+c\,\sigma^4.
\end{align*}
For $a,b\in\mathbb{R}$ we have $\min(a,b) = (a+b)/2 - |a-b|/2$. Using this relation, 
conditional version of Jensen's inequality and the last display, we derive:
\begin{align*}
\mathbb{E} \big[\min\big(X_k^2, \rho\,\sigma^2\big)|\bm{X}_{k-1}\big] &= 
\frac{1}{2}\mathbb{E}\big[X_k^2 +\rho\sigma^2-|X_k^2-\rho\sigma^2| \big|\bm{X}_{k-1}\big]\\
& \ge \frac{1}{2} \Big[X_{k-1}^2\theta_k^2 +(\rho+1)\sigma^2
-\Big(\mathbb{E}\big[\big(X_k^2-\rho\sigma^2\big)^2|\bm{X}_{k-1}\big]\Big)^{1/2}\Big],
\end{align*}
for $\rho>0$.
We now have, by plugging in the expressions derived above and simplifying,
\begin{align*}
\mathbb{E}\big[\big(X_k^2 & -\rho\sigma^2\big)^2|\bm{X}_{k-1}\big]
= \mathbb{E}\big[X_k^4|\bm{X}_{k-1}\big] - 2\rho\sigma^2\mathbb{E}\big[X_k^2|\bm{X}_{k-1}\big]
+\rho^2\sigma^4\\
&=X_{k-1}^4\theta_k^4 + 2(3-\rho)X_{k-1}^2\theta_k^2\sigma^2
+(c-2\rho+\rho^2)\sigma^4 = \Big(X_{k-1}^2\theta_k^2+\frac{c+3}4\sigma^2\Big)^2,
\end{align*}
if we pick $\rho=(9-c)/4>1$.
Combining the previous two displays, we conclude that
\[
\mathbb{E}\big[\min\big(X_k^2, (9-c)\sigma^2/4\big)\big|\bm{X}_{k-1}\big] \ge 
\frac{(5-c)\sigma^2}{4}, \quad k \in \mathbb{N},
\]
and the statement of the lemma follows.
\end{proof}

\bibliographystyle{apa}
%\bibliography{timeseriestracking_v13}

\end{document}